\documentclass[11pt]{article}

\usepackage{amsmath, amsthm,amsfonts,amssymb,mathtools,mathdots,scalerel}

\usepackage{enumitem, prooftree, tocloft}
\usepackage[hidelinks]{hyperref}
\usepackage{tikz, stmaryrd, cmll}
\usepackage[british]{babel}
\usepackage[a4paper]{geometry}

\newtheorem{thm}{Theorem}[section]
\newtheorem{prop}[thm]{Proposition}
\newtheorem{cor}[thm]{Corollary}
\newtheorem{lem}[thm]{Lemma}
\theoremstyle{definition}
\newtheorem{dfn}[thm]{Definition}
\theoremstyle{remark}
\newtheorem{remark}[thm]{Remark}
\newtheorem{cons}[thm]{Construction}
\newtheorem{exm}[thm]{Example}
\numberwithin{equation}{section}

\usetikzlibrary{arrows, decorations.markings, shapes.geometric, decorations.pathmorphing, decorations.pathreplacing, intersections, patterns,calc, backgrounds}

\tikzset{
	0c/.style={circle, draw, fill, inner sep=1.5pt},
	1c/.style={->, thick, shorten <=2pt, shorten >=2pt},
	edge/.style={thick, shorten <=2pt, shorten >=2pt},
	1cdot/.style={->, densely dashed, thick, shorten <=2pt, shorten >=2pt},
	2c/.style={double, thick, shorten <=6pt, shorten >=8pt, decoration={markings,mark=at position -6pt with {\arrow[scale=1.75]{>}}}, preaction={decorate}},
	follow/.style={->, >=stealth, very thick, shorten <=3pt, shorten >=3pt},
	every node/.style={scale=.8},
}

\newcommand\bord[2]{\partial_{#1}^{#2}}
\newcommand\idd[1]{\mathrm{id}_{#1}}
\newcommand\opp[1]{{#1}^\mathrm{op}}
\newcommand\coo[1]{{#1}^\mathrm{co}}
\newcommand\pback[2]{\;{}_{#1}\!\!\times_{#2}}
\newcommand\invrs[1]{#1^{-1}}
\newcommand\cutt[1]{\mathrm{cut}_{#1}}
\newcommand\mrg[2]{\mathrm{mrg}_{#1}^{#2}}
\newcommand\thin[1]{\epsilon_{#1}}
\newcommand\sequ[1]{\langle{ #1 }\rangle}
\newcommand\tensor{\,{\scaleobj{0.75}{\boxtimes}}\,}
\newcommand\imonad[1]{\mathcal{I}#1}
\newcommand\mmonad[1]{\mathcal{M}#1}
\newcommand\tmonad[1]{\mathcal{T}#1}

\newcommand\cat[1]{\mathbf{#1}}
\newcommand\rpol{\cat{CPol}}
\newcommand\rtwopol{\cat{r2Pol}}
\newcommand\pbcat{\cat{rPolBiCat}}
\newcommand\mbcat{\cat{rMulBiCat}}
\newcommand\pcat{\cat{rPolCat}}
\newcommand\mcat{\cat{rMulCat}}
\newcommand\bicat{\cat{BiCat}}
\newcommand\moncat{\cat{MonCat}}
\newcommand\linbicat{\cat{LinBiCat}}
\newcommand\lindist{\cat{LinDistCat}}
\newcommand\staraut{\cat{*AutCat}}
\newcommand\twoglob{\cat{2Gph}}
\newcommand\chu[1]{\textit{Chu}(#1)}
\newcommand\mrgpol{\cat{MrgBiCat}}

\newcommand\rimp[2]{#1\!\multimap\!#2}
\newcommand\limp[2]{#2\!\multimapinv\!#1}
\newcommand\rcimp[2]{#1 \diagdown #2}
\newcommand\lcimp[2]{#2 \diagup #1}

\title{\Large\bfseries Weak units, universal cells, and \\ coherence via universality for bicategories}

\author{Amar Hadzihasanovic\\[3pt] RIMS, Kyoto University}

\date{v1: March 2018 --- v3: September 2019}

\begin{document}
\maketitle

\vspace{-7pt}
\begin{minipage}{0.9\linewidth}
\small Poly-bicategories generalise planar polycategories in the same way as bicategories generalise monoidal categories. In a poly-bicategory, the existence of enough 2-cells satisfying certain universal properties (representability) induces coherent algebraic structure on the 2-graph of single-input, single-output 2-cells. A special case of this theory was used by Hermida to produce a proof of strictification for bicategories. No full strictification is possible for higher-dimensional categories, seemingly due to problems with 2-cells that have degenerate boundaries; it was conjectured by C.\ Simpson that semi-strictification excluding units may be possible.

We study poly-bicategories where 2-cells with degenerate boundaries are barred, and show that we can recover the structure of a bicategory through a different construction of weak units. We prove that the existence of these units is equivalent to the existence of 1-cells satisfying lower-dimensional universal properties, and study the relation between preservation of units and universal cells.

Then, we introduce merge-bicategories, a variant of poly-bicategories with more composition operations, which admits a natural monoidal closed structure, giving access to higher morphisms. We derive equivalences between morphisms, transformations, and modifications of representable merge-bicategories and the corresponding notions for bicategories. Finally, we prove a semi-strictification theorem for representable merge-bicategories with a choice of composites and units.
\end{minipage}

\tableofcontents

\section{Introduction}

This article is a study in the theory of bicategories with an eye towards weak higher categories, in particular the open problems regarding their strictifiability.

Its most immediate influence is Hermida's work \cite{hermida2000representable}, which established a kind of ``Grothendieck construction'' for monoidal categories (and, by extension, bicategories), whereby the coherent structure that makes a category monoidal is subsumed by a property of a multicategory, that of being \emph{representable}: namely, the existence of enough \emph{universal} 2-cells, through which other 2-cells factor uniquely whenever possible. 

An interesting outcome of this construction was a new proof that each monoidal category is equivalent to a strict monoidal category, coming from the multicategorical side: representable multicategories with a choice of universal 2-cells are seen as pseudoalgebras for a 2-monad on the 2-category of multicategories, and such pseudoalgebras are equivalent to strict ones. Hermida dubbed this general strategy \emph{coherence via universality} \cite{hermida2001coherent}.

It is clear from the author's words that this was meant as a first step towards higher-dimensional generalisations: indeed, the construction of bicategorical structure from universal 2-cells is, essentially, a low-dimensional instance of the multitopic approach to weak higher categories, proposed by Hermida, Makkai, and Power \cite{hermida2000weak}, equivalent to the opetopic approach of Baez and Dolan \cite{baez1998higher}. However, years later, coherence theorems for dimensions higher than 3 remain elusive. Our work is an attempt to advance the programme.

A first, necessary observation is that the obvious generalisation of the strictification theorem for bicategories is false: a result of C.\ Simpson \cite[Theorem 4.4.2]{simpson2009homotopy} settled definitively what was previously folklore, that general tricategories are not equivalent to strict ones. The very case of tricategories as a pseudoalgebra for a 2-monad was picked by Shulman \cite{shulman2012notevery} as a counterexample to the idea that ``most pseudoalgebras are equivalent to strict ones''.

By varying the algebraic theory, we may hope to use the same general approach to prove weaker, semi-strictification results. As Shulman's work suggests, strictifiability is theory-dependent, so we do not expect the general algebra of 2-monads to indicate the right direction: we believe, rather, that a specific understanding of the combinatorics of higher categories is necessary. 

In fact, while we follow Hermida in studying strictification from the side of objects-with-properties, rather than objects-with-structure, we diverge in most other aspects: in particular, we do not use the same 2-categorical algebraic framework, and opt instead for an elementary, combinatorial approach.

\subsubsection*{Many inputs and many outputs}
String diagrams and their generalisations, such as surface diagrams, are one of the main sources of intuition about higher categories. Recently, they have been at the centre of an approach to semistrict algebraic higher categories \cite{bar2016data} built around the proof assistant Globular \cite{vicary2016globular}; less recently, they have been the basis of many coherence theorems, either implicitly \cite{kelly1980coherence} or explicitly \cite{joyal1991geometry}. 

In most useful applications, string diagrams represent cells with many inputs and many outputs, while multicategories and opetopic/multitopic sets only have many-to-one cells. This makes the description of higher morphisms --- transformations, modifications, etc --- quite unwieldy. The simplest description of such higher structure is based on a monoidal closed structure on the category of higher categories, such as the (lax) Gray product of strict $\omega$-categories, and corresponding internal homs \cite{steiner2004omega}. However, opetopes/multitopes are not closed under such products; this is why, for example, the only transformations of functors that appear naturally in a multicategorical approach to bicategories are icons \cite{lack2010icons}, that is, transformations whose 1-cell components are identities.

It is now due saying that an instance of ``coherence via universality'' earlier than Hermida's appeared in the work of Cockett and Seely on linearly distributive categories and representable polycategories \cite{cockett1997weakly}, and later, with Koslowski, on linear bicategories and representable poly-bicategories \cite{cockett2000introduction,cockett2003morphisms}. Poly-bicategories, a ``coloured'' generalisation of polycategories, have 2-cells with many inputs and many outputs, but a restricted form of composition, meant to interpret the cut rule of sequent calculus. 

As such, representability of poly-bicategories induces not one, but two compatible structures of bicategory --- what was called a linear bicategory structure in \cite{cockett2000introduction} --- which is probably too expressive for our purposes. Moreover, the cut-rule composition does not interact well with Gray products. This is why, in Section \ref{sec:probicat}, we will introduce \emph{merge-bicategories}, a modified version of poly-bicategories whose two potential bicategory structures collapse to one, and which admits a monoidal closed structure.

So we favour spaces of cells with many inputs and many outputs, but apart from the added complexity in the combinatorics of cell shapes, this choice presents several technical subtleties. If we simply allow unrestricted $\omega$-categorical shapes, we obtain the category of polygraphs or computads \cite{burroni1993higher,street1976limits}, which, unlike one would expect, is not a presheaf category, as shown by Makkai and Zawadowski \cite{makkai2008category}. 

In \cite{cheng2012direct}, Cheng pointed at 2-cells with degenerate, 0-dimensional boundaries as the source of this failure. These show up, directly or indirectly, in a number of other ``no-go'' theorems: the aforementioned proof by C.\ Simpson that strict 3-groupoids are not equivalent to 3-dimensional homotopy types, and the impossibility of defining proof nets for multiplicative linear logic with units \cite{heijltjes2014proof}, which would yield a simpler proof of coherence for linearly distributive categories compared to \cite{blute1996natural}.

Incidentally, in most examples of poly-bicategories that we know of, there is no natural definition of 2-cells with no inputs and no outputs; this is true even of the main example of \cite{cockett2003morphisms}, the poly-bicategorical Chu construction, as discussed in our Example \ref{exm:chu}. This leads us to our second technical choice.

\subsubsection*{The regularity constraint}
Our response is to restrict $n$-dimensional cell shapes (or ``pasting diagrams'') to those that have a geometric realisation that is homeomorphic to an $n$-dimensional ball; we call this \emph{regularity}, in analogy with a similar constraint on CW complexes. In general, this leads to a notion of \emph{regular polygraph}, that we introduced in \cite[Section 3.2]{hadzihasanovic2017algebra}, and developed in \cite{hadzihasanovic2018combinatorial, hadzihasanovic2019rds}. In the 2-dimensional case that we consider here, regularity simply bars cells with no inputs or no outputs. We introduce corresponding notions of regular poly-bicategory and regular multi-bicategory, which have an underlying regular 2-polygraph.

The construction of a weakly associative composition from Cockett, Seely, and Hermida carries through to regular contexts with no alteration. On the other hand, these authors relied on 2-cells with degenerate boundaries in order to construct weak units: for this, we develop a new approach. 

We will see in Section \ref{sec:strictify} that strictification for associators also carries through from Hermida's work, leading to a proof of semi-strictification for bicategories. This is, of course, weaker than the well-known full strictification result, to the point that it may look like a meagre payoff at the end of the article. Nevertheless, after proving his no-go theorem, C.\ Simpson offered evidence that, while homotopy types are not modelled by strict higher groupoids, they may still be equivalent to higher groupoids that satisfy strict associativity and interchange, but have weak units; this is now known as \emph{C.\ Simpson's conjecture}. In this sense, this is a strictification method with the potential to generalise: it is the proofs and the constructions, rather than the result, that we consider to be of interest.

A variant of C.\ Simpson's conjecture, restricted to the groupoidal case (all cells are invertible), was recently proved by S.\ Henry \cite{henry2018regular}, who came independently to the regularity constraint on polygraphs \cite{henry2017non}. In its original and most general form, the conjecture is considered to have been settled only in the 3-dimensional case by Joyal and J.\ Kock \cite{joyal2007weak}: for this purpose, they considered a notion of weak unit based on a suggestion of Saavedra-Rivano \cite{saavedra1972tannakiennes}, whose defining characteristics are \emph{cancellability} and \emph{idempotence}.

The theory of \emph{Saavedra units} was developed by J.\ Kock in \cite{kock2008elementary}, the second largest influence of our article. We will show that a notion of Saavedra unit can be formulated in the context of regular poly-bicategories or merge-bicategories, and is subsumed by the existence of 2-cells that satisfy two universal properties at once: the one typical of ``representing'' cells as in Cockett, Seely, and Hermida, and the one typical of ``internal homs'' or ``Kan extensions''.

Most interestingly, we will show that the existence of these weak units is equivalent to the existence of certain \emph{universal 1-cells}, where universality is formulated with respect to an ``internal'' notion of composition, witnessed by universal 2-cells. These form an elementary notion of equivalence, independent of the pre-existence of units, yet formulated in such a way that dividing a universal 1-cell by itself, we obtain a weak unit. Universal 1-cells are similar to the universal 1-cells in the opetopic approach, but have some subtle technical advantages, given by our consideration of different universal properties at once. These advantages are not visible in the 2-dimensional case, and we will postpone their discussion to future work.

\subsubsection*{All concepts are universal cells}

The construction of units from universal cells allows us to formulate representability for merge-bicategories --- which induces the structure of a bicategory --- in a particularly uniform way, stated purely in terms of the existence of universal 1-cells and 2-cells, and to define morphisms of representable merge-bicategories, inducing functors of bicategories, as those that preserve universal cells. ``Mapping universal cells to universal cells'' makes sense even for higher morphisms, when dimensions do not match and ``mapping units to units'' may produce a typing error; see the definition of fair transformations in Section \ref{sec:probicat}. 

This reformulation of basic bicategory theory, centred on \emph{universality} as the one fundamental notion, is the other main contribution of this article next to the semi-strictification strategy. While the opetopic approach also constructs both composites and units from the single concept of universal cells, the correspondence is slightly different: in both approaches, $n$-dimensional composites come from ``binary'' universal $(n+1)$-cells; however, in the opetopic approach, $n$-dimensional units come from ``nullary'' universal $(n+1)$-cells, while in ours they come from ``unary'' universal $n$-cells. As far as we know, the monoidal closed structure, transformations, and modifications had not been treated in any earlier work on the same line.

The reason why we think this is relevant, and not ``yet another equivalent formalism'' for an established theory, is that, as observed among others by Gurski \cite{gurski2009coherence}, we may have an idea of what a weakly associative composition looks like in arbitrary dimensions, be it via Stasheff polytopes \cite{stasheff1970hspaces} or orientals \cite{street1987algebra}; but we do not know at all what the equations for weak units should be. Gurski's coherence equations are not generalisations of Mac Lane's triangle in any obvious way, and the triangle equation itself has been criticised for its reliance on associators in the formulation of a unitality constraint \cite{kock2008elementary}.

Universal cells, on the other hand, have analogues in any dimension \cite[Appendix B]{hadzihasanovic2018combinatorial}, and if we can argue that they subsume both composition and units, then they have both a technical and a conceptual advantage. Taking the effort to show that the main fundamental notions in bicategory theory can be recovered (with the exception, admittedly, of lax functors, which do not translate well) establishes a stable ground for the theory in higher dimensions, where clear reference points may not be at hand.

\subsubsection*{Logical aspects}

We have given ample indication of how we expect to follow up on this article, for what concerns the theory of higher categories. Here, we would like to briefly discuss some logical aspects of our work, that we also hope to develop in the future.

As evidenced by Cockett and Seely in \cite{cockett1997weakly}, there is a deep connection between universal 2-cells in polycategories and left and right introduction rules in sequent calculus: roughly, division corresponds to one introduction rule for a connective, and composition with the universal cell to the other; whether the first or the second is a left rule defines the \emph{polarity} of the connective \cite{andreoli1992logic}. When units are produced as ``nullary composites'', witnessed by a universal 2-cell with a degenerate boundary, division corresponds to an introduction rule, and the unit has the same polarity as the connective for which it is a unit. 

However, our notion of units as universal 1-cells does not fit properly into this scheme. Even if we considered coherent witnesses of unitality, as in Theorem \ref{thm:polycoherent}, to be the universal 2-cells associated to the unit as a logical connective, we would find that division produces an \emph{elimination} rule, and the unit has opposite polarity compared to its connective. It would seem that the two notions of unit, even though they induce the same coherent structure, are really logically inequivalent, a fact which needs further exploration on the proof-theoretic side.

On a related note, the regularity constraint on polycategories, as a semantics for sequent calculus, bars the emptying of either side of a sequent, and forces one to introduce units before moving formulas from one side to the other. For example, the following are attempts at proving double negation elimination in a ``non-regular'' and ``regular'' sequent calculus, respectively, for noncommutative linear logic (thanks to P.\ Taylor's proof tree macros):
\begin{equation*}
\begin{minipage}{0.45\linewidth}
	\begin{prooftree}
		 	\[ { \quad } \justifies \bot \vdash \quad \using{\bot_L} \] \quad
			\[ \[ \[  { \quad } \justifies A \vdash A \using{\textsc{ax}} \] 
			\justifies A \vdash \bot, A \using{\bot_R} \] 
				\justifies \quad \vdash \rimp{A}{\bot}, A \using{\multimap_R}\]
		 \justifies \limp{(\rimp{A}{\bot})}{\bot} \vdash A \using{\multimapinv_L}
	\end{prooftree}
\end{minipage}
\begin{minipage}{0.45\linewidth}
	\begin{prooftree}
		 	\[ { \quad } \justifies \bot \vdash \bot \using{\textsc{ax}} \] \quad
			\[ \[ \[  { \quad } \justifies A \vdash A \using{\textsc{ax}} \] 
			\justifies A, 1 \vdash \bot, A \using{1_L, \bot_R} \] 
				\justifies 1 \vdash \rimp{A}{\bot}, A \using{\multimap_R}\]
		 \justifies \limp{(\rimp{A}{\bot})}{\bot}, 1 \vdash \bot, A \using{\multimapinv_L}
	\end{prooftree}
\end{minipage}
\end{equation*}
In the ``regular'' proof, there is a residual unit on both sides. This makes us wonder: is the minimal number of residual units in a regular proof an interesting classifier for non-regular proofs? And does the regularity constraint affect the complexity of proof search or proof equivalence? 

\subsubsection*{Structure of the article}
We start in Section \ref{sec:polybicat} by defining regular poly-bicategories and their sub-classes --- regular multi-bicategories, polycategories, and multicategories --- and by fixing a notation for universal 2-cells. Unlike those of \cite{cockett2003morphisms}, our poly-bicategories do not come with unit 2-cells as structure; we take the opportunity to show their construction from universal 2-cells (Proposition \ref{prop:2units}), as a warm-up to the construction of unit 1-cells from universal 1-cells.

In Section \ref{sec:weakunits}, we develop the elementary theory of weak units and universal 1-cells in regular poly-bicategories. After introducing tensor and par universal 1-cells, we show that they satisfy a ``two-out-of-three'' property, that is, they are closed under composition and division (Theorem \ref{thm:2outof3}). Moreover, dividing a universal 1-cell by itself produces a unit, so the existence of universal 1-cells is equivalent to the existence of units (Theorem \ref{thm:0representable}). In the rest of the section, we study conditions under which universality is equivalent to invertibility (Proposition \ref{prop:isomorphism} and \ref{prop:isoclosed}), and preservation of universal cells by a morphism is equivalent to preservation of units.

Section \ref{sec:coherence} begins with the important technical result that the 2-cells witnessing the unitality of unit 1-cells can be chosen in a ``coherent'' way (Theorem \ref{thm:polycoherent}). We use this to derive results analogous to those of \cite{hermida2000representable} (Corollary \ref{cor:equiv_bicat_rep}) and of \cite{cockett2003morphisms} (Corollary \ref{cor:equiv_linbicat_rep}) under the regularity constraint: representable regular multi-bicategories and tensor strong morphisms are equivalent to bicategories and functors, and representable regular poly-bicategories and strong morphisms are equivalent to linear bicategories and strong linear functors. We conclude the section with an analysis of the Chu construction for poly-bicategories (Example \ref{exm:chu}), showing that unit 1-cells cannot be constructed as ``nullary composites'', which contradicts \cite[Example 2.5(4)]{cockett2003morphisms}, but can be constructed with our methods.

In Section \ref{sec:probicat}, we introduce merge-bicategories as a setting for our strictification strategy, and a first step for higher-dimensional generalisations. Their theory of universal cells is simpler (Proposition \ref{prop:mergeunital} and Corollary \ref{cor:tensorparunitcollapse}), yet representable merge-bicategories and their strong morphisms are equivalent to bicategories and functors (Theorem \ref{mrgpol_equivalence}). We show that merge-bicategories have a natural monoidal closed structure (Proposition \ref{prop:graymonoidal}), which gives us access to higher morphisms; we then relate them to higher morphisms of bicategories (Theorem \ref{thm:fair_coherence}).

Finally, in Section \ref{sec:strictify}, we prove a semi-strictification theorem for representable merge-bicategories. We encode a strictly associative choice of units and composites in the structure of an algebra for a monad $\tmonad{}$; the key point is that $\tmonad{}$ splits into two components, $\imonad{}$ and $\mmonad{}$. Then, any representable merge-bicategory $X$ admits an $\imonad{}$-algebra structure (Proposition \ref{prop:ialg_0rep}), and its inclusion into $\mmonad{X}$ is an equivalence (Lemma \ref{lem:merge_inclusion}), realising the semi-strictification of $X$ (Theorem \ref{thm:semistrict}).

\subsubsection*{Acknowledgements}
This work was supported by a JSPS Postdoctoral Research Fellowship and by JSPS KAKENHI Grant Number 17F17810. The author would like to thank Joachim Kock and Jamie Vicary for their feedback on the parts which overlap with the author's thesis, Alex Kavvos and the participants of the higher-categories meetings in Oxford for useful comments in the early stages of this work, and Paul-Andr\'e Melli\`es for pointing him to the work of Hermida back in 2016. The revised version has benefitted from helpful comments by the anonymous referee and by Pierre-Louis Curien.

\section{Regular poly-bicategories and universal 2-cells} \label{sec:polybicat}

In this section, we present our variants of the definitions of \cite{cockett2003morphisms}, and some basic results. The following terminology is borrowed from \cite{burroni1993higher}, for what could equally be called a regular 2-computad, following \cite{street1976limits}.

\begin{dfn}
A \emph{regular 2-polygraph} $X$ consists of a diagram of sets
\begin{equation*}
\begin{tikzpicture}
	\node[scale=1.25] (0) at (0,0) {$X_0$};
	\node[scale=1.25] (1) at (2.5,0) {$X_1$};
	\draw[1c] (1.west |- 0,.15) to node[auto,swap] {$\bord{}{+}$} (0.east |- 0,.15);
	\draw[1c] (1.west |- 0,-.15) to node[auto] {$\bord{}{-}$} (0.east |- 0,-.15);
\end{tikzpicture}
\end{equation*}
together with diagrams of sets
\begin{equation*}
\begin{tikzpicture}
	\node[scale=1.25] (0) at (0,0) {$X_1$};
	\node[scale=1.25] (1) at (2.5,0) {$X_2^{(n,m)},$};
	\draw[1c] (1.west |- 0,.15) to node[auto,swap] {$\bord{j}{+}$} (0.east |- 0,.15);
	\draw[1c] (1.west |- 0,-.15) to node[auto] {$\bord{i}{-}$} (0.east |- 0,-.15);
	\node[scale=1.25] at (6,.25) {$i = 1,\ldots,n,$};
	\node[scale=1.25] at (6,-.25) {$j = 1,\ldots,m,$};
\end{tikzpicture}
\end{equation*}
for all $n, m \geq 1$, satisfying
\begin{align*}
	& \bord{}{-}\bord{1}{-} = \bord{}{-}\bord{1}{+}\;, \\ 
	& \bord{}{+}\bord{n}{-} = \bord{}{+}\bord{m}{+}\;, \\
	& \bord{}{+}\bord{i}{-} = \bord{}{-}\bord{i+1}{-}\;, \quad \quad \;i=1,\ldots,n-1, \\
	& \bord{}{+}\bord{j}{+} = \bord{}{-}\bord{j+1}{+}\;, \quad \quad \;j=1,\ldots,m-1.
\end{align*}
The elements of $X_0$ are called \emph{0-cells}, those of $X_1$ \emph{1-cells}, and those of $X_2^{(n,m)}$ \emph{2-cells}, for any $n, m$. The functions $\bord{}{-}$, $\bord{i}{-}$ are called ($i$-th) \emph{input}, and the $\bord{}{+}$, $\bord{j}{+}$ ($j$-th) \emph{output}; the inputs form the \emph{input boundary}, and the outputs the \emph{output boundary} of a cell.

Given regular 2-polygraphs $X$ and $Y$, a \emph{morphism} $f: X \to Y$ is a morphism of diagrams, in the obvious sense. Regular 2-polygraphs and their morphisms form a large category $\rtwopol$.
\end{dfn}

We picture a 2-cell $p \in X_2^{(n,m)}$ as
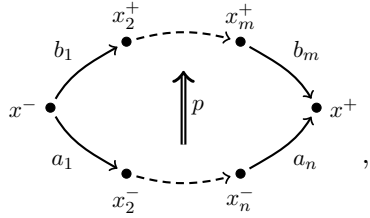
\begin{equation*}
\input{img/s2_2cell.tex}
\end{equation*}
where $\bord{i}{-}p = a_i$, and $\bord{j}{+}p = b_j$, for $i = 1, \ldots, n$, and $j = 1,\ldots, m$; then, $\bord{}{-}b_1 = x^-$, $\bord{}{+}b_1 = x_2^+$, and so on. Dashed arrows stand for a (possibly empty) sequence of 1-cells, while solid arrows stand for a single 1-cell. We will also write
\begin{equation*}
	p: (a_1, \ldots, a_n) \to (b_1, \ldots, b_m)
\end{equation*}
to indicate the same cell together with the 1-cells in its boundary, which are implied to have compatible boundaries (``be composable''). Similarly, we will write $a: x \to y$ for a 1-cell with $\bord{}{-}a = x$, and $\bord{}{+}a = y$. We will often use capital Greek letters $\Gamma, \Delta, \ldots$ for composable sequences of 1-cells.

From any regular 2-polygraph, we obtain three others where the direction of 1-cells, or 2-cells, or both, is reversed.
\begin{dfn}
Given a regular 2-polygraph $X$, the regular 2-polygraph $\opp{X}$ has the same 0-cells as $X$, and
\begin{enumerate}
	\item a 1-cell $\opp{a}: y \to x$ for each 1-cell $a: x \to y$ of $X$, and
	\item a 2-cell $\opp{p}: (\opp{a}_n, \ldots, \opp{a}_1) \to (\opp{b}_m, \ldots, \opp{b}_1)$ for each $p: (a_1,\ldots,a_n) \to (b_1,\ldots,b_m)$ of $X$.
\end{enumerate}
The regular 2-polygraph $\coo{X}$ has the same 0-cells as $X$, the same 1-cells as $X$, and a 2-cell $\coo{p}: (b_1, \ldots, b_m) \to (a_1, \ldots, a_n)$ for each 2-cell $p: (a_1,\ldots,a_n) \to (b_1,\ldots,b_m)$ of $X$.

Extending the definition to morphisms in the obvious way, for example letting $\opp{f}(\opp{a}) := \opp{f(a)}$, defines involutive endofunctors $\opp{(-)}$ and $\coo{(-)}$ on $\rtwopol$.
\end{dfn}

A regular 2-polygraph $X$ induces a 2-graph, or 2-globular set \cite[Section 1.4]{leinster2004higher}, by restriction to 2-cells that have a single input and a single output 1-cell.
\begin{dfn}
Given a regular 2-polygraph $X$, the 2-graph $GX$ is the sub-diagram 
\begin{equation*}
\begin{tikzpicture}
	\node[scale=1.25] (0) at (0,0) {$X_0$};
	\node[scale=1.25] (1) at (2.5,0) {$X_1$};
	\node[scale=1.25] (2) at (5,0) {$X_2^{(1,1)}.$};
	\draw[1c] (1.west |- 0,.15) to node[auto,swap] {$\bord{}{+}$} (0.east |- 0,.15);
	\draw[1c] (1.west |- 0,-.15) to node[auto] {$\bord{}{-}$} (0.east |- 0,-.15);
	\draw[1c] (2.west |- 0,.15) to node[auto,swap] {$\bord{1}{+}$} (1.east |- 0,.15);
	\draw[1c] (2.west |- 0,-.15) to node[auto] {$\bord{1}{-}$} (1.east |- 0,-.15);
\end{tikzpicture}
\end{equation*}
of $X$. 

Given a morphism $f: X \to Y$ of regular 2-polygraphs, its restriction to $GX$ induces a morphism $Gf: GX \to GY$ of 2-graphs. This defines a functor $G: \rtwopol \to \twoglob$, where $\twoglob$ is the category of 2-graphs.
\end{dfn}

\begin{remark}
The functor $G$ has a left adjoint $J: \twoglob \to \rtwopol$, identifying a 2-graph with a regular 2-polygraph which has only single-input, single-output 2-cells.
\end{remark}

We now want to introduce an algebraic composition of 2-cells, where composable pairs are those in any of the following geometric setups:
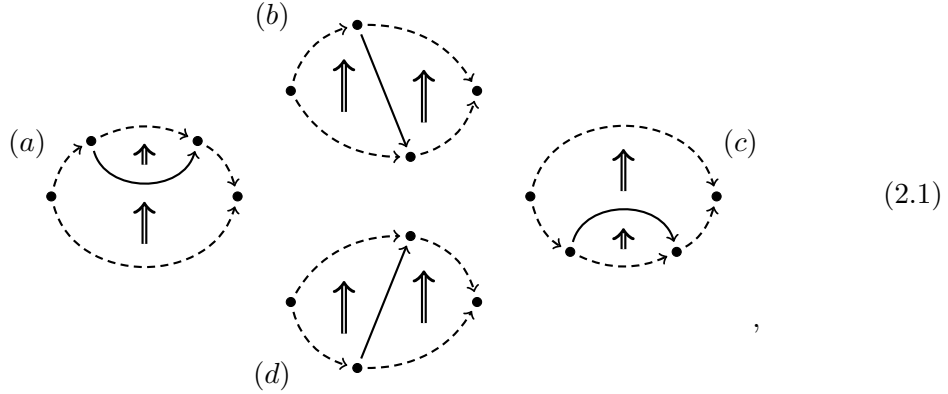
\begin{equation} \label{eq:composable}
\input{img/s2_composable.tex}
\end{equation}
and only regularity constraints apply outside of the shared boundary. The composition corresponds to ``merging'' the two 2-cells, while fixing the overall boundary, as in the following picture:
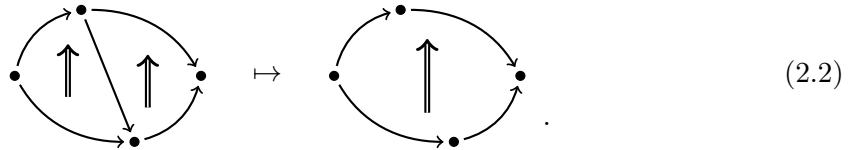
\begin{equation} \label{eq:merging}
\input{img/s2_merging.tex}
\end{equation}

\begin{dfn}
A \emph{(non-unital) regular poly-bicategory} is a regular 2-polygraph $X$ together with functions
\begin{equation*}
\begin{tikzpicture}
	\node[scale=1.25] (0) at (0,0) {$X_2^{(n,m)} \pback{\bord{j}{+}}{\bord{i}{-}} X_2^{(p,q)}$};
	\node[scale=1.25] (1) at (4.5,0) {$X_2^{(n+p-1,m+q-1)},$};
	\draw[1c] (0) to node[auto] {$\cutt{j,i}$} (1);
\end{tikzpicture}
\end{equation*}
whenever $i, j$ satisfy the two conditions on any side of the following square:
\begin{equation} \label{eq:poly-boundaries}
\begin{tikzpicture}[baseline={([yshift=-.5ex]current bounding box.center)}]
	\node[scale=1.25] (0) at (-1.5,.75) {$i=1$};
	\node[scale=1.25] (1) at (1.5,.75) {$j=m$};
	\node[scale=1.25] (2) at (-1.5,-.75) {$i=p$};
	\node[scale=1.25] (3) at (1.5,-.75) {$j=1.$};
	\draw[edge] (0.east) to node[auto] {$(b)$} (1.west);
	\draw[edge] (2.east) to node[auto,swap] {$(d)$} (3.west);
	\draw[edge] (0.south) to node[auto,swap] {$(a)$} (2.north);
	\draw[edge] (1.south) to node[auto] {$(c)$} (3.north);
\end{tikzpicture}
\end{equation}
These interact with the boundaries as follows, depending on which pair of conditions is satisfied:
\begin{enumerate}[label=($\alph*$)]
	\item $\bord{k}{-}\cutt{j,1}(t,s) = \bord{k}{-}t, \\
	\bord{k}{+}\cutt{j,1}(t,s) = \begin{cases}
		\bord{k}{+}t, & 1 \leq k \leq j-1, \\
		\bord{k-j+1}{+}s, & j \leq k \leq j+q-1, \\
		\bord{k-q+1}{+}t, & j+q \leq k \leq m+q-1;
	\end{cases}$
		
	\item $\bord{k}{-}\cutt{m,1}(t,s) = \begin{cases}
		\bord{k}{-}t, & 1 \leq k \leq n, \\
		\bord{k-n+1}{-}s, & n+1 \leq k \leq n+p-1,
	\end{cases} \\
	\bord{k}{+}\cutt{m,1}(t,s) = \begin{cases}
		\bord{k}{+}t, & 1 \leq k \leq m-1, \\
		\bord{k-m+1}{+}s, & m \leq k \leq m+q-1;
	\end{cases}$
	
	\item $\bord{k}{-}\cutt{1,i}(t,s) = \begin{cases}
		\bord{k}{-}s, & 1 \leq k \leq i-1, \\
		\bord{k-i+1}{-}t, & i \leq k \leq i+n-1, \\
		\bord{k-n+1}{-}s, & i+n \leq k \leq n+p-1,
	\end{cases} \\ \bord{k}{+}\cutt{1,i}(t,s) = \bord{k}{+}s;$
	
	\item $\bord{k}{-}\cutt{1,p}(t,s) = \begin{cases}
		\bord{k}{-}s, & 1 \leq k \leq p-1, \\
		\bord{k-p+1}{-}t, & p \leq k \leq n+p-1;
	\end{cases}\\ 
	\bord{k}{+}\cutt{1,p}(t,s) = \begin{cases}
		\bord{k}{+}s, & 1 \leq k \leq q, \\
		\bord{k-q+1}{+}t, & q+1 \leq k \leq m+q-1.
	\end{cases}$
\end{enumerate}
All of these are evident from the geometric picture (\ref{eq:composable}).

Moreover, the $\cutt{j,i}$ satisfy associativity and interchange equations, expressing the fact that when three 2-cells can be composed in two different orders, the result is independent of the order. Schemes of associativity equations are classified by the following 9 pictures, where the direction of 2-cells is from bottom to top:
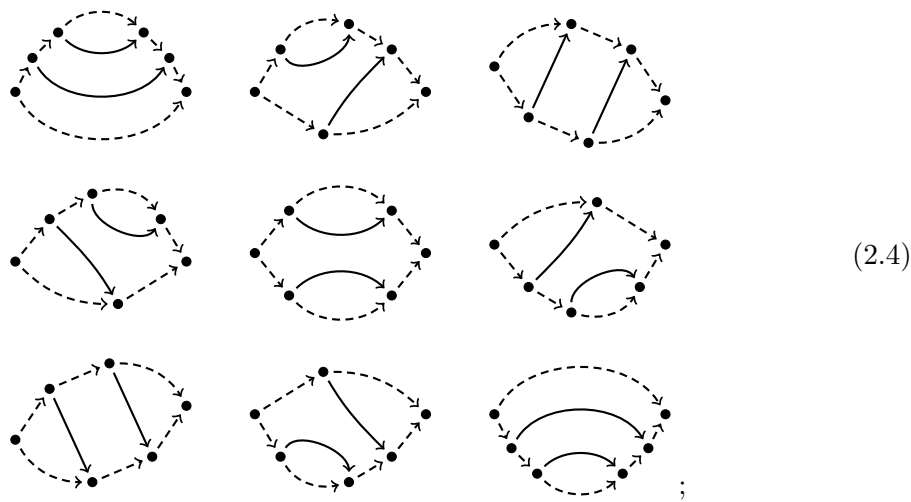
\begin{equation} \label{asso_scheme}
\input{img/s2_associativity.tex}
\end{equation}
for example, the leftmost picture in the top row should be read as the equation scheme $\cutt{j,1}(t, \cutt{i,1}(s,r)) = \cutt{i+j-1,1}(\cutt{j,1}(t,s),r)$.

Schemes of interchange equations are classified by the following 8 pictures:
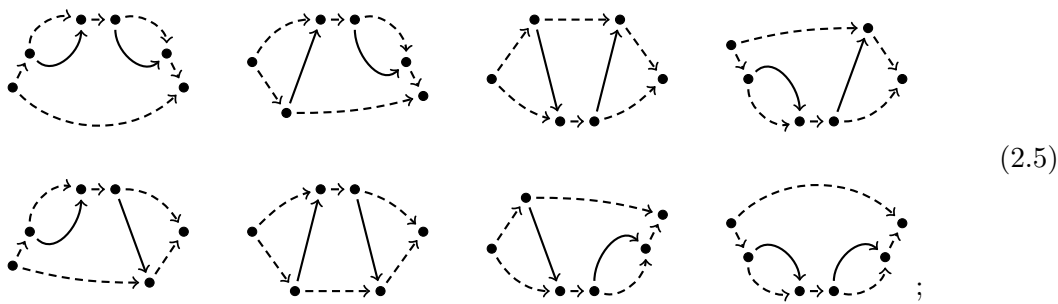
\begin{equation} \label{inter_scheme}
\input{img/s2_interchange.tex}
\end{equation}
for example, the leftmost picture in the top row should be read as the equation scheme $\cutt{i,1}(\cutt{i+k,1}(t,s),s') = \cutt{i+k+q-1,1}(\cutt{i,1}(t,s'),s)$, where $s'$ has $q$ outputs.

Given two regular poly-bicategories $X, Y$, a \emph{morphism} $f: X \to Y$ is a morphism of the underlying regular 2-polygraphs that commutes with the $\cutt{i,j}$ functions. Regular poly-bicategories and their morphisms form a large category $\pbcat$.
\end{dfn}

\begin{remark}
When picturing 2-cells in diagrams, we will casually identify a pasting diagram of 2-cells with its composite. Uniqueness of the composite, in the case where there are different orders of composition, can be deduced from the soundness of the circuit diagram language for poly-bicategories, stated in \cite[Appendix A]{cockett2000introduction}, since circuit diagrams are dual to pasting diagrams.
\end{remark}

There is an evident forgetful functor $U: \pbcat \to \rtwopol$, whose composition with $G: \rtwopol \to \twoglob$ we still denote with $G$. The functor $U$ is monadic: the left adjoint $F: \rtwopol \to \pbcat$ freely adds all cut-composable pasting diagrams of 2-cells to a regular 2-polygraph.

\begin{dfn}
A \emph{regular multi-bicategory} is a regular poly-bicategory $X$ such that $X_2^{(n,m)}$ is empty whenever $m > 1$. A \emph{regular polycategory} is a regular poly-bicategory with a single 0-cell. A \emph{regular multicategory} is a regular multi-bicategory with a single 0-cell. 

We write $\mbcat$, $\pcat$, and $\mcat$, respectively, for the corresponding full subcategories of $\pbcat$.
\end{dfn}

\begin{remark}
Any ordinary poly-bicategory or multi-bicategory yields a regular one by restricting to 2-cells with at least one input and one output. While in some cases this restriction may seem unnatural, in many others it is the other way around, and cells with no inputs or outputs are added in an \emph{ad hoc} manner, for instance by implicitly relying on a unit 1-cell. 

That is, in particular, the case for the poly-bicategorical version of the Chu construction, as discussed in Example \ref{exm:chu}.
\end{remark}

\begin{exm} \label{exm:lattice}
Every lattice $(L, \land, \lor)$ can be identified with a regular polycategory $L$ which has a unique 2-cell $(a_1,\ldots,a_n) \to (b_1,\ldots,b_m)$ whenever $a_1 \land \ldots \land a_n \leq b_1 \lor \ldots \lor b_m$ in the induced poset. Notice that ``having a greatest or least element'' is a property of a lattice, yet to define a non-regular polycategory one needs them as structure: 2-cells with nullary input are identified \emph{by construction} with 2-cells with input $(\top)$, and 2-cells with nullary output with 2-cells with output $(\bot)$.
\end{exm}

We now give an elementary definition of universality of a 2-cell at a 1-cell in its boundary, which encompasses the strongly universal arrows of \cite{hermida2000representable}, and the homs, cohoms and representing cells of \cite{cockett2003morphisms}.

\begin{dfn}
Let $t \in X_2^{(n,m)}$, $s \in X_2^{(p,q)}$ be two 2-cells in a regular poly-bicategory $X$, and let $\pi^1, \pi^2$ be the two projections of $X_2^{(n,m)} {}_{\bord{j}{+}}\!\!\times_{\bord{i}{-}} X_2^{(p,q)}$, for some $i,j$ compatible with (\ref{eq:poly-boundaries}). We say that
\begin{equation*}
	\cutt{j,i}(t,x) = s
\end{equation*}
is a \emph{well-formed equation} in the indeterminate $x$ if, for any $k, k'$, and $\alpha \in \{+,-\}$ such that $\bord{k}{\alpha}\cutt{j,i} = \bord{k'}{\alpha}\pi^1$, it holds that $\bord{k}{\alpha}s = \bord{k'}{\alpha}t$. Similarly, 
\begin{equation*}
	\cutt{j,i}(y,t) = s
\end{equation*}
is a well-formed equation in the indeterminate $y$ if $\bord{k}{\alpha}s = \bord{k'}{\alpha}t$ whenever $\bord{k}{\alpha}\cutt{j,i} = \bord{k'}{\alpha}\pi^2$.
\end{dfn}

In other words, $\cutt{j,i}(t,x) = s$ is well-formed if there can exist some $r$ such that $\cutt{j,i}(t,r) = s$, compatibly with the structure of a regular poly-bicategory; similarly for the other case.

\begin{dfn}
Let $t \in X_2^{(n,m)}$ be a 2-cell in a regular poly-bicategory $X$. We say that $t$ is \emph{universal at $\bord{j}{+}$} if, for all 2-cells $s$ and well-formed equations $\cutt{j,i}(t,x) = s$, there exists a unique 2-cell $r$ such that $\cutt{j,i}(t,r) = s$. We say that $t$ is \emph{universal at $\bord{i}{-}$} if, for all 2-cells $s$ and well-formed equations $\cutt{j,i}(x,t) = s$, there exists a unique 2-cell $r$ such that $\cutt{j,i}(r,t) = s$.

We say that $t$ is \emph{everywhere universal} if it is universal at $\bord{i}{-}$ and at $\bord{j}{+}$ for all $i= 1,\ldots,n$, and $j = 1,\ldots,m$. 

If $n = m = 1$, we simply say that $t$ is \emph{universal} if it is universal at $\bord{1}{+}$ and at $\bord{1}{-}$.
\end{dfn}

So universality of a 2-cell $t$ at a location in its boundary means that any 2-cell that can, in principle, be factorised as the composition of $t$ with another 2-cell at that same location, does factor through $t$, and does so uniquely.

\begin{remark}
Any universal property of a 2-cell implies a \emph{cancellability} property. For example, if $t$ is universal at $\bord{j}{+}$, uniqueness of factorisations through $t$ implies that for any pair of 2-cells $r_1$ and $r_2$ such that $\cutt{j,i}(t,r_1)$ and $\cutt{j,i}(t,r_2)$ are defined,
\begin{equation*}
	\cutt{j,i}(t,r_1) = \cutt{j,i}(t,r_2) \quad \text{implies} \quad r_1 = r_2.
\end{equation*}
\end{remark}

\begin{dfn}
A 2-cell $\idd{a}: (a) \to (a)$ in a regular poly-bicategory $X$ is a \emph{unit} on the 1-cell $a$ if, for all $t \in X_2^{(n,m)}$, if $\bord{j}{+}t = a$ then $\cutt{j,1}(t,\idd{a}) = t$, and if $\bord{i}{-}t = a$ then $\cutt{1,i}(\idd{a},t) = t$.
\end{dfn}
It is clear that any unit is a universal 2-cell. The reason why we did not require regular poly-bicategories to have units as structure is because we want to show the following constructions, which, while quite trivial by themselves, are a simpler analogue of the constructions of Section \ref{sec:weakunits}, and can serve as a warm-up for our later results.

\begin{lem} \label{lem:2divunit}
Let $p: (a) \to (a')$ be a universal 2-cell in a regular poly-bicategory $X$, and let $\idd{a}: (a) \to (a)$, $\idd{a'}: (a') \to (a')$ be the unique 2-cells such that 
\begin{equation*}
	\cutt{1,1}(\idd{a}, p) = p, \quad\quad \cutt{1,1}(p, \idd{a'}) = p.
\end{equation*}
Then $\idd{a}$ is a unit on $a$, and $\idd{a'}$ is a unit on $a'$.
\end{lem}
\begin{proof}
Consider a 2-cell $q: (\Gamma_1, a, \Gamma_2) \to (\Delta)$ in $X$; by universality of $p$ at $\bord{1}{+}$, there exists a unique $q': (\Gamma_1, a', \Gamma_2) \to (\Delta)$ such that $\cutt{1,i}(p,q') = q$. Then,
\begin{align*}
	\cutt{1,i}(\idd{a}, q) = \cutt{1,i}(\idd{a}, \cutt{1,i}(p,q')) = \cutt{1,i}(\cutt{1,1}(\idd{a},p),q') = \cutt{1,i}(p,q') = q.
\end{align*}
Next, consider a 2-cell $r: (\Gamma) \to (\Delta_1, a, \Delta_2)$. Post-composing it with $p$, we find
\begin{equation*}
	\cutt{j,1}(r, p) = \cutt{j,1}(r, \cutt{1,1}(\idd{a},p)) = \cutt{j,1}(\cutt{j,1}(r,\idd{a}),p);
\end{equation*}
therefore, by uniqueness of factorisations through $p$, we obtain $r = \cutt{j,1}(r,\idd{a})$. This proves that $\idd{a}$ is a unit on $a$; the same proof applied to $\coo{X}$ proves that $\idd{a'}$ is a unit on $a'$.
\end{proof}

\begin{prop} \label{prop:2units}
Let $X$ be a regular poly-bicategory. The following are equivalent:
\begin{enumerate}
	\item for all 1-cells $a$ of $X$, there exist a 1-cell $\overline{a}$ and a universal 2-cell $p: (a) \to (\overline{a})$;
	\item for all 1-cells $a$ of $X$, there exist a 1-cell $\overline{a}$ and a universal 2-cell $p': (\overline{a}) \to (a)$;
	\item for all 1-cells $a$ of $X$, there exists a (necessarily unique) unit $\idd{a}$ on $a$.
\end{enumerate}
\end{prop}
\begin{proof}
Let $a$ be a 1-cell of $X$. If there exists a unit $\idd{a}: (a) \to (a)$, it is clearly universal, and fulfils the other two conditions for $a$.

Conversely, by Lemma \ref{lem:2divunit}, from any universal 2-cell $e: (a) \to (\overline{a})$, and from any universal 2-cell $e': (\overline{a}) \to (a)$, we can construct a unit on $a$.
\end{proof}

\begin{dfn}
A regular poly-bicategory $X$ is \emph{unital} if it satisfies any of the equivalent conditions of Proposition \ref{prop:2units}.
\end{dfn} 

In the presence of units, we retrieve the usual notion of universal 2-cells as isomorphisms, that is, 2-cells with an inverse, and one-sided universality suffices for a 2-cell to be an isomorphism.
\begin{prop} \label{cor:oneuniversal}
Let $X$ be a unital regular poly-bicategory, and $p: (a) \to (a')$ a 2-cell of $X$ with a single input and output. The following are equivalent:
\begin{enumerate}
	\item $p$ is universal at $\bord{1}{+}$;
	\item $p$ is universal at $\bord{1}{-}$;
	\item $p$ is an isomorphism, that is, it has a unique inverse $\invrs{p}: (a') \to (a)$ such that $\cutt{1,1}(p,\invrs{p}) = \idd{a}$, and $\cutt{1,1}(\invrs{p},p) = \idd{a'}$.
\end{enumerate}
\end{prop}
\begin{proof}
Suppose $p$ is universal at $\bord{1}{+}$. Dividing $\idd{a}$ by $p$, we obtain a unique 2-cell $\invrs{p}: (a') \to (a)$ such that $\cutt{1,1}(p,\invrs{p}) = \idd{a}$. Since 
\begin{equation*}
	\cutt{1,1}(p, \cutt{1,1}(\invrs{p},p)) = \cutt{1,1}(\cutt{1,1}(p,\invrs{p}),p) = \cutt{1,1}(\idd{a},p) = p = \cutt{1,1}(p, \idd{a'}),
\end{equation*}
it follows by uniqueness that $\cutt{1,1}(\invrs{p},p) = \idd{a'}$. Similarly when $p$ is universal at $\bord{1}{-}$. The converse implication is obvious.
\end{proof}

Given two 1-cells $a, b: x \to y$ in a unital regular poly-bicategory, we will write $a \simeq b$ to mean ``there exists an isomorphism $p: (a) \to (b)$''.

\begin{prop} \label{prop:unitmorph}
Let $X, Y$ be unital regular poly-bicategories, and $f: X \to Y$ a morphism. Then $f$ preserves units if and only if it preserves universal 2-cells.
\end{prop}
\begin{remark}
In the statement, ``universal'' is without restriction, hence applies only to single-input, single-output 2-cells; other universal properties are not necessarily preserved.
\end{remark}
\begin{proof}
If $f$ preserves units, then it also preserves isomorphisms, hence universal 2-cells. Conversely, if $f$ preserves universal 2-cells, because units are universal, $f(\idd{a})$ is universal in $Y$ for each unit $\idd{a}$ in $X$. Then, 
\begin{equation*}
	\cutt{1,1}(f(\idd{a}),\idd{f(a)}) = f(\idd{a}) = f(\cutt{1,1}(\idd{a},\idd{a})) = \cutt{1,1}(f(\idd{a}),f(\idd{a})), 
\end{equation*}
and by uniqueness of factorisations through $f(\idd{a})$, it follows that $f(\idd{a}) = \idd{f(a)}$.
\end{proof}

\begin{dfn}
A morphism $f: X \to Y$ of unital regular poly-bicategories is \emph{unital} if it satisfies either of the equivalent conditions of Proposition \ref{prop:unitmorph}. 
\end{dfn}

The homs and input-representing 2-cells of \cite{cockett2003morphisms} correspond to the three possible universal properties of 2-cells $p: (a,b) \to (c)$ with two inputs and one output.
\begin{dfn}
Let $a$, $b$ be two 1-cells in a regular poly-bicategory $X$, with compatible boundaries as required separately by each definition. 

A \emph{tensor} of $a$ and $b$ is a 1-cell $a \otimes b$, together with a 2-cell $t_{a,b}: (a,b) \to (a\otimes b)$ that is universal at $\bord{1}{+}$. A \emph{right hom} from $a$ to $b$ is a 1-cell $\rimp{a}{b}$, together with a 2-cell $e^R_{a,b}: (a, \rimp{a}{b}) \to (b)$ that is universal at $\bord{2}{-}$. A \emph{left hom} from $a$ to $b$ is a 1-cell $\limp{a}{b}$, together with a 2-cell $e^L_{a,b}: (\limp{a}{b}, a) \to (b)$ that is universal at $\bord{1}{-}$.

We say that $X$ is \emph{tensor 1-representable}, \emph{right 1-closed}, and \emph{left 1-closed}, respectively, if it is unital and has tensors, right homs, and left homs, respectively, for all pairs of 1-cells in the appropriate configuration.
\end{dfn}

Dually, the notions of cohoms and output-representing 2-cells correspond to the different universal properties of 2-cells $p': (c) \to (a,b)$.
\begin{dfn}
Let $a$, $b$ be two 1-cells in a regular poly-bicategory $X$, with compatible boundaries as required separately by each definition. 

A \emph{par} of $a$ and $b$ is a 1-cell $a \parr b$, together with a 2-cell $c_{a,b}: (a\parr b) \to (a,b)$ that is universal at $\bord{1}{-}$. A \emph{right cohom} from $a$ to $b$ is a 1-cell $\rcimp{a}{b}$, together with a 2-cell $c^R_{a,b}: (b) \to (a, \rcimp{a}{b})$ that is universal at $\bord{2}{+}$. A \emph{left cohom} from $a$ to $b$ is a 1-cell $\lcimp{a}{b}$, together with a 2-cell $c^L_{a,b}: (b) \to (\lcimp{a}{b}, a)$ that is universal at $\bord{1}{+}$.

We say that $X$ is \emph{par 1-representable}, \emph{right 1-coclosed}, and \emph{left 1-coclosed}, respectively, if it is unital and has pars, right cohoms, and left cohoms, respectively, for all pairs of 1-cells in the appropriate configuration.
\end{dfn}

A tensor (left hom, right hom) in $X$ is the same as a par (left cohom, right cohom) in $\coo{X}$, and a left hom (left cohom) in $X$ is the same as a right hom (right cohom) in $\opp{X}$.

The following is easily proved in the same way as \cite[Corollary 8.6]{hermida2000representable}.
\begin{lem}
Let $p$ be a 2-cell in a unital regular poly-bicategory, and suppose $p$ is universal at $\bord{j}{+}$. Then:
\begin{enumerate}[label=(\alph*)]
	\item if $p'$ is another 2-cell universal at $\bord{j}{+}$, which has the same boundaries as $p$ except at $\bord{j}{+}$, then factorising $p'$ through $p$ produces an isomorphism $q: (\bord{j}{+}p) \to (\bord{j}{+}p')$;
	\item if $q: (\bord{k}{+}p) \to (a)$ and $q': (b) \to (\bord{k'}{-}p)$ are isomorphisms, the 2-cells $\cutt{k,1}(p,q)$ and $\cutt{1,k'}(q',p)$ are also universal at $\bord{j}{+}$.
\end{enumerate}
Dual results hold when $p$ is universal at $\bord{i}{-}$.
\end{lem}
In particular, in a unital regular poly-bicategory, tensors, pars, homs, and cohoms are unique up to isomorphism. In the next section, we will need the following technical lemma about 2-cells satisfying two different universal properties at once. 

\begin{lem} \label{lem:twouniversal}
Let $t: (a, b) \to (c)$ be a 2-cell in a unital regular poly-bicategory $X$, and suppose $t$ satisfies two different universal properties. Then:
\begin{itemize}
	\item if $t$ is universal at $\bord{1}{+}$ and $\bord{2}{-}$, then any 2-cell $u: (a, b) \to (c')$ that is universal at $\bord{1}{+}$ is also universal at $\bord{2}{-}$, and any 2-cell $u: (a, b') \to (c)$ that is universal at $\bord{2}{-}$ is also universal at $\bord{1}{+}$;
	\item if $t$ is universal at $\bord{1}{-}$ and $\bord{2}{-}$, then any 2-cell $u: (a', b) \to (c)$ that is universal at $\bord{1}{-}$ is also universal at $\bord{2}{-}$, and any 2-cell $u: (a, b') \to (c)$ that is universal at $\bord{2}{-}$ is also universal at $\bord{1}{-}$.
\end{itemize}
All duals of these statements through $\opp{(-)}$ and $\coo{(-)}$ also hold.
\end{lem}
\begin{proof}
We only consider the case where $t$ is universal at $\bord{1}{+}$ and $\bord{2}{-}$, and $u: (a, b') \to (c)$ is universal at $\bord{2}{-}$; the others are completely analogous. By factorising $u$ through $t$, we obtain
\begin{equation*}
\input{img/s2_twouniversal.tex}
\end{equation*} 
for a unique $p: (b') \to (b)$, which, since $X$ is unital, must be an isomorphism. As composition with isomorphisms does not affect universal properties, and $t$ is universal at $\bord{1}{+}$, it follows that $u$ is also universal at $\bord{1}{+}$.
\end{proof}

Appropriate compositions of binary tensors, homs, pars, and cohoms can be used to produce $n$-ary ones. Moreover, certain divisions preserve universal properties. The following lemma collects some useful closure properties of this sort.

\begin{lem} \label{lem:comp_divis_closure}
Let $t \in X^{(n,1)}$, $r \in X^{(m,1)}$ be 2-cells in a regular poly-bicategory $X$, such that $\bord{1}{+}t = \bord{i}{-}r$, and let $s := \cutt{1,i}(t,r)$. 
\begin{enumerate}[label=(\alph*)]
	\item Suppose $t$ is universal at $\bord{1}{+}$. Then:
	\begin{itemize}
		\item $r$ is universal at $\bord{1}{+}$ if and only if $s$ is universal at $\bord{1}{+}$;
		\item if $i \neq 1$, then $r$ is universal at $\bord{1}{-}$ if and only if $s$ is universal at $\bord{1}{-}$.
	\end{itemize}
	\item Suppose $i = 1$ and $r$ is universal at $\bord{1}{-}$. Then $t$ is universal at $\bord{1}{-}$ if and only if $s$ is universal at $\bord{1}{-}$.
\end{enumerate}
All duals of these statements through $\opp{(-)}$ and $\coo{(-)}$ also hold.
\end{lem}
\begin{proof}
Suppose $t$ and $r$ are both universal at $\bord{1}{+}$. If $\cutt{1,j}(s,x) = p$ is a well-formed equation, then so is $\cutt{1,i+j-1}(t,x) = p$, which has a unique solution $p'$. Then $\cutt{1,j}(r,x) = p'$ is also well-formed and has a unique solution $q$, so
\begin{equation*}
	\cutt{1,j}(s,q) = \cutt{1,j}(\cutt{1,i}(t,r),q) = \cutt{1,i+j-1}(t,\cutt{1,j}(r,q)) = \cutt{1,i+j-1}(t,p') = p.
\end{equation*}
Cancellability of both $t$ and $r$ implies that this solution is unique. This proves that $s$ is universal at $\bord{1}{+}$. 

Now, suppose that $t$ and $s$ are both universal at $\bord{1}{+}$, and let $\cutt{1,j}(r,x) = p'$ be a well-formed equation. Then $\cutt{1,j}(s,x) = \cutt{1,i+j-1}(t,p')$ is also well-formed, and has a unique solution $q$. It follows that
\begin{equation*}
	\cutt{1,i+j-1}(t,p') = \cutt{1,j}(s,q) = \cutt{1,j}(\cutt{1,i}(t,r),q) = \cutt{1,i+j-1}(t,\cutt{1,j}(r,q)),
\end{equation*}
and by cancellability of $t$, we have $p' = \cutt{1,j}(r,q)$; uniqueness of the solution follows from the cancellability of $s$.

The proof of point \emph{(b)} is entirely analogous.

Suppose $i \neq 1$, $t$ is universal at $\bord{1}{+}$, and $r$ is universal at $\bord{1}{-}$. Let $p \in X^{(k,\ell)}$ be such that $\cutt{\ell,1}(x,s) = p$ is well-formed. Then $\cutt{1,k'+i-1}(t,x) = p$ is also well-formed for $k' = k-n-m+2$, and has a unique solution $p'$, with the property that $\cutt{\ell,1}(x,r) = p'$ is well-formed. Solving it yields a unique $q$ such that
\begin{equation*}
	\cutt{\ell,1}(q,s) = \cutt{\ell,1}(q,\cutt{1,i}(t,r)) = \cutt{1,k'+i-1}(t,\cutt{\ell,1}(q,r)) = \cutt{1,k'+i-1}(t,p') = p.
\end{equation*}
Uniqueness is a consequence of the cancellability properties of $t$ and $r$. This proves that $s$ is universal at $\bord{1}{-}$.

Finally, suppose $i \neq 1$, $t$ is universal at $\bord{1}{+}$, and $s$ is universal at $\bord{1}{-}$. If $p' \in X^{(\tilde{k},\ell)}$ is such that $\cutt{\ell,1}(x,r) = p'$ is well-formed, then $\cutt{1,\ell}(x,s) = \cutt{1,k'+i-1}(t,p')$ is also well-formed for $k' = \tilde{k} - m + 1$. Its unique solution $q$ satisfies
\begin{equation*}
	\cutt{1,k'+i-1}(t,p') = \cutt{\ell,1}(q,s) = \cutt{\ell,1}(q,\cutt{1,i}(t,r)) = \cutt{1,k'+i-1}(t,\cutt{\ell,1}(q,r)),
\end{equation*}
so $p' = \cutt{\ell,1}(q,r)$. Again, uniqueness follows from cancellability of $t$ and $s$. 
\end{proof}

The theory of universal 2-cells in a unital regular poly-bicategory is just a specialisation of the one developed by Hermida, Cockett, Seely, and Koslowski; we refer to \cite{cockett2003morphisms} for more details, and move on to the original treatment of weak units.

\section{Weak units and universal 1-cells} \label{sec:weakunits}

In the proof of Proposition \ref{prop:2units}, we assumed a representability condition, which provided us with 2-cells that are ``equivalences'' for an elementary, unit-independent notion of equivalence. From that, we constructed units for a strictly associative, algebraic composition. That proof is going to be our blueprint for the construction of weak unit 1-cells, relative to the two internal, weakly associative notions of composition given by tensors and pars. We will proceed as follows:
\begin{enumerate}
	\item first, we will define a notion of weak unit 1-cell appropriate for regular poly-bicategories;
	\item then, we will introduce an elementary notion of universal 1-cell, and prove that weak units can be constructed from universal 1-cells, so their existence is equivalent to a lower-dimensional representability condition;
	\item finally, in Section \ref{sec:coherence} we will show that our weak units induce the intended coherent structure on $GX$.
\end{enumerate}
Our notion of weak units is based on \emph{Saavedra units} in monoidal categories, as defined by J.\ Kock in \cite{kock2008elementary}, based on Saavedra-Rivano's work \cite{saavedra1972tannakiennes}. Kock summarises the defining properties of a Saavedra unit $i$ in a monoidal category as \emph{cancellability} (morphisms $i \otimes a \to i \otimes b$ and $a \otimes i \to b \otimes i$ correspond uniquely to morphisms $a \to b$) and \emph{idempotence} ($i \otimes i$ is isomorphic to $i$). 

From the first condition applied to $i \otimes (i \otimes a) \simeq (i \otimes i) \otimes a \simeq i \otimes a$, one obtains left actions $i \otimes a \to a$ of the unit on any object, themselves isomorphisms; similarly one obtains right actions. This means that Saavedra units can always be introduced to the left or right of any object; cancellability means that after their introduction, they can be eliminated, as long as there is an object on which they are acting. 

This should be contrasted with the way units are obtained in \cite{hermida2000representable, cockett2003morphisms} and related works: being exhibited by universal 2-cells with a degenerate, 0-dimensional boundary, units can \emph{always} be eliminated by composition, even when they are ``on their own''; that is, from a 2-cell $(i) \to (a)$, one can obtain a 2-cell $() \to (a)$ by composition with a universal 2-cell.

We claim that the notion of Saavedra unit is captured at the poly-bicategorical level by the following definition, where we restrict our attention to tensor units first. Because compositions of 1-cells are subsumed by universal properties, we do not need to formulate any property with respect to a specified composition, and can use universality as the one fundamental concept.

\begin{dfn}
Let $x$ be a 0-cell in a regular poly-bicategory $X$. A 1-cell $1_x: x \to x$ is a \emph{tensor unit} on $x$ if, for each $a: x \to y$ and each $b: z \to x$, there exist 2-cells
\begin{equation*}
\input{img/s3_tensorunit.tex}
\end{equation*}
that are, respectively, universal at $\bord{1}{+}$ and $\bord{2}{-}$, and universal at $\bord{1}{+}$ and $\bord{1}{-}$: that is, $l_a$ exhibits $a$ as both $1_x \otimes a$ and $\rimp{1_x}{a}$, and $r_b$ exhibits $b$ as both $b \otimes 1_x$ and $\limp{1_x}{b}$.
\end{dfn}

\begin{remark} \label{remark:units}
By Lemma \ref{lem:twouniversal}, if $X$ is unital, and $1_x: x \to x$ is a tensor unit on $x$, a seemingly stronger claim can be made: if a 2-cell of the form $t: (1_x, a) \to (a')$ is universal at $\bord{1}{+}$, then it is also universal at $\bord{2}{-}$; if a 2-cell of the form $u: (1_x, a') \to (a)$ is universal at $\bord{2}{-}$, then it is also universal at $\bord{1}{+}$ and so on. 

In particular, any tensor $t_{1_x,a}: (1_x, a) \to (1_x \otimes a)$ satisfies both universal properties. If we fix such a tensor for each 1-cell $x \to y$, we obtain a bijection between 2-cells $p': (1_x \otimes a) \to (1_x \otimes b)$ and 2-cells $p: (a) \to (b)$ by considering 
\begin{equation*}
\input{img/s3_cancellability.tex}
\end{equation*}
either as an equation to be solved for $p$ given $p'$, using the universality of $t_{1_x,b}$ at $\bord{2}{-}$, or as an equation to be solved for $p'$ given $p$, using the universality of $t_{1_x,a}$ at $\bord{1}{+}$. 

Dually, any tensor $t_{a,1_y}: (a,1_y) \to (a \otimes 1_y)$ induces a bijection between 2-cells $(a \otimes 1_y) \to (b \otimes 1_y)$ and 2-cells $(a) \to (b)$. In this sense, tensor units subsume the cancellability of Saavedra units. Idempotence, on the other hand, is witnessed by the universality at $\bord{1}{+}$ of either $l_{1_x}$ or $r_{1_x}$, which produces an isomorphism between $1_x \otimes 1_x$ and $1_x$. 
\end{remark}

Next, we define an elementary notion of universality of 1-cells, relative to the composition exhibited by universal 2-cells.
\begin{dfn}
Let $e: x \to x'$ be a 1-cell in a regular poly-bicategory $X$. We say that $e$ is \emph{tensor left universal} if, for each $a: x \to y$ and each $a': x' \to y$, a right hom and tensor
\begin{equation*}
\input{img/s3_tensordivleft.tex}
\end{equation*}
exist and are universal both at $\bord{1}{+}$ and at $\bord{2}{-}$: that is, $e^R_{e,a}$ also exhibits $a$ as $e \otimes (\rimp{e}{a})$, and $t_{e,a'}$ also exhibits $a'$ as $\rimp{e}{(e \otimes a')}$.

Dually, we say that $e$ is \emph{tensor right universal} if, for each $b: z \to x$ and each $b': z \to x'$, a left hom and tensor
\begin{equation*}
\input{img/s3_tensordivright.tex}
\end{equation*}
exist and are universal both at $\bord{1}{+}$ and at $\bord{1}{-}$: that is, $e^L_{e,b'}$ also exhibits $b'$ as $(\limp{e}{b'}) \otimes e$, and $t_{b,e}$ also exhibits $b$ as $\limp{e}{(b \otimes e)}$. 

A 1-cell $e$ is \emph{tensor universal} if it is both tensor left and tensor right universal.
\end{dfn}

\begin{remark}
If $X$ is unital and $e$ is tensor left universal, since $e^R_{e,a}$ and $t_{e,\rimp{e}{a}}$ both exist and are universal at $\bord{1}{+}$, factorising one through the other produces a unique isomorphism $e \otimes (\rimp{e}{a}) \simeq a$; similarly, factorising $t_{e,a'}$ through $e^R_{e,e \otimes a'}$ produces a unique isomorphism $\rimp{e}{(e\otimes a')} \simeq a'$.

Dually, if $e$ is tensor right universal, we obtain unique isomorphisms $(\limp{e}{b'}) \otimes e \simeq b'$ and $\limp{e}{(b \otimes e)} \simeq b$.
\end{remark}

\begin{remark}
A \emph{quasigroup}, in the equational formulation \cite[Section 1.2]{smith2006introduction}, is a set $Q$ together with three binary operations $\;\cdot\;, \diagup, \diagdown$ satisfying the axioms
\begin{align*}
	& x \cdot (\rcimp{x}{y}) = y, & \rcimp{x}{(x \cdot y)} = y, \\
	& (\lcimp{x}{y}) \cdot x = y, & \lcimp{x}{(y \cdot x)} = y.
\end{align*}
The isomorphisms enforced by tensor universality, in a unital regular poly-bicategory, can be seen as a categorified version of these equations, with $\otimes$ corresponding to $\cdot$, $\multimap$ to $\diagdown$, and $\multimapinv$ to $\diagup$.
\end{remark}

An important property of tensor universal 1-cells in a unital regular poly-bicategory is that they satisfy a ``two-out-of-three'' property --- a common requirement for classes of weak equivalences --- in the following sense.

\begin{thm} \label{thm:2outof3}
Let $e: x \to x'$, $e': x' \to x''$, $e'': x \to x''$ be 1-cells in a unital regular poly-bicategory $X$, and suppose that
\begin{equation*}
\input{img/s3_2outof3.tex}
\end{equation*}
satisfies two different universal properties. If two of the three 1-cells $e$, $e'$, and $e''$ are tensor universal, then the third is also tensor universal, and $p$ is everywhere universal.
\end{thm}
\begin{proof}
Suppose first that $e$ and $e'$ are tensor universal. Because $p$ satisfies two different universal properties, it is universal at $\bord{1}{-}$ or at $\bord{2}{-}$. Factorising it through $e^L_{e',e''}$, in the first case, or through $e^R_{e,e''}$, in the second case, we find that it is also universal at $\bord{1}{+}$, that is, it exhibits $e''$ as $e \otimes e'$. Finally, factorising through $t_{e,e'}$, which, by tensor universality of both $e$ and $e'$ and an application of Lemma \ref{lem:twouniversal}, is everywhere universal, we obtain that $p$ has the same property. 

To check that $e''$ is tensor left universal, consider an arbitrary 1-cell $a : x \to y$. By universality of $p$ at $\bord{1}{+}$, the equation
\begin{equation} \label{eq:2outof3_1}
\input{img/s3_2outof3_tensor.tex}
\end{equation} 
holds for a unique 2-cell $q$. By Lemma \ref{lem:comp_divis_closure}, the left-hand side is universal both at $\bord{1}{+}$ and at $\bord{3}{-}$ (for the latter, use the $\opp{(-)}$-dual of point \emph{(b)}); it follows from point \emph{(a)} and its duals that $q$ is universal both at $\bord{1}{+}$ and at $\bord{2}{-}$.

Similarly, given an arbitrary 1-cell $a': x'' \to y$, the unique 2-cell $q'$ obtained in the factorisation
\begin{equation*} 
\input{img/s3_2outof3_tensor2.tex}
\end{equation*} 
using the universality of $p$ at $\bord{1}{+}$ is universal both at $\bord{1}{+}$ and at $\bord{2}{-}$. This proves that $e''$ is tensor left universal; a dual argument shows that it is tensor right universal.

Next, we consider the case in which $e$ and $e''$ are tensor universal. Because $p$ satisfies two different universal properties, it is universal at $\bord{1}{+}$ or at $\bord{2}{-}$; factorising it through $t_{e,e'}$, in the first case, or through $e^R_{e,e''}$, in the second case, we find that it must actually satisfy both of them. 

To check that $e'$ is tensor left universal, consider a 1-cell $b : x' \to y$. The unique 2-cell $r$ obtained in the factorisation
\begin{equation} \label{eq:2outof3_2}
\input{img/s3_2outof3_righthom.tex}
\end{equation} 
using the universality of $t_{e,b}$ at $\bord{2}{-}$ is universal at $\bord{2}{-}$ by Lemma \ref{lem:comp_divis_closure}. To prove that it is also universal at $\bord{1}{+}$, consider a 2-cell
\begin{equation*}
	p': (\Gamma_1, e', \rimp{e''}{(e \otimes b)}, \Gamma_2) \to (\Delta).
\end{equation*}
Suppose that $(\Gamma_1) = (\Gamma_1', c)$ for some $c: z \to x'$; then, we can perform the following sequence of factorisations (labels of 0-cells are omitted):
\begin{equation*} 
\input{img/s3_2outof3_factor.tex}
\end{equation*} 
\begin{equation*} 
\input{img/s3_2outof3_factor2.tex}
\end{equation*} 
for unique 2-cells $\tilde{p}$ and $p''$, where we used first the universality at $\bord{1}{+}$ of the two sides of equation (\ref{eq:2outof3_2}), then the universality at $\bord{1}{+}$ of $e^L_{e,c}$. Cancelling the latter, we obtain a necessarily unique factorisation of $p'$ through $r$. In case $\Gamma_1$ is empty, let $(\Delta) = (c', \Delta')$ for some $c': x' \to z$, and apply the same reasoning to the postcomposition of $p'$ with $t_{e,c'}: (e, c') \to (e \otimes c')$.

Next, consider a 1-cell $b': x'' \to y$. The unique 2-cell $r'$ obtained in the factorisation
\begin{equation*} 
\input{img/s3_2outof3_righthom2.tex}
\end{equation*} 
using the universality of $e^R_{e,e'' \otimes b'}$ at $\bord{2}{-}$ is universal at $\bord{2}{-}$ by Lemma \ref{lem:comp_divis_closure}, and at $\bord{1}{+}$ by a similar argument. This proves that $e'$ is tensor left universal.

The proof that $e'$ is tensor right universal is similar, and involves the 2-cells $s, s'$ obtained from factorisations
\begin{equation*} 
\input{img/s3_2outof3_lefthom.tex}
\end{equation*} 
\begin{equation*} 
\input{img/s3_2outof3_lefthom2.tex}
\end{equation*} 
Once we have established that $e'$ is tensor universal, the fact that $p$ is also universal at $\bord{1}{+}$ is a consequence of Lemma \ref{lem:twouniversal}.

Finally, the statement in the case where $e', e''$ are tensor universal follows from the previous case applied to $\opp{X}$.
\end{proof}

\begin{cor} \label{cor:div_closure}
The class of tensor universal 1-cells in a unital regular poly-bicategory is closed under tensors, left homs, and right homs.
\end{cor}
\begin{proof}
If $e''$ is obtained as a tensor, right hom, or left hom of tensor universal 1-cells, then the 2-cell that exhibits it falls under the hypotheses of Theorem \ref{thm:2outof3}. It follows that $e''$ is also tensor universal.
\end{proof}

Armed with this result, we can prove a lower-dimensional version of Lemma \ref{lem:2divunit} and of Proposition \ref{prop:2units}.

\begin{lem} \label{lem:1sideunit}
Let $e: x \to x'$ be a tensor universal 1-cell in a unital regular poly-bicategory. Then $\limp{e}{e}: x \to x$ is a tensor unit on $x$, and $\rimp{e}{e}: x' \to x'$ is a tensor unit on $x'$.
\end{lem}
\begin{proof}
Let $a: x \to y$ be a 1-cell. The unique 2-cell $p$ obtained in the factorisation
\begin{equation*}
\input{img/s3_leftunit.tex}
\end{equation*} 
using the universality of $e^R_{e,a}$ at $\bord{1}{+}$ is universal at $\bord{1}{+}$ by Lemma \ref{lem:comp_divis_closure}. By Corollary \ref{cor:div_closure}, $\limp{e}{e}$ is tensor universal, so by Lemma \ref{lem:twouniversal} $p$ is also universal at $\bord{2}{-}$.

Similarly, take any $a': x' \to y$. The unique 2-cell $p'$ obtained in the factorisation
\begin{equation*}
\input{img/s3_leftunitimpl.tex}
\end{equation*} 
using the universality of $t_{e,a'}$ at $\bord{2}{-}$ is itself universal at $\bord{2}{-}$. Since $\rimp{e}{e}$ is tensor universal, $p'$ is also universal at $\bord{1}{+}$. 

This proves the left tensor unit condition for both $\limp{e}{e}$ and $\rimp{e}{e}$; a dual argument in $\opp{X}$ leads to the right tensor unit condition.
\end{proof}

\begin{remark} \label{rmk:booleanalgebra}
Note that the converse does not hold: even if $\limp{e}{e}$ and $\rimp{e}{e}$ are tensor units, $e$ may not be tensor universal. 

For example, a Boolean algebra $(B, \top, \bot, \land, \lor, \neg)$ can be identified with a regular polycategory, as a special case of Example \ref{exm:lattice}. Then $B$ is unital, and has $\top$ as a tensor unit. 

For all $a,b \in B$, the element $\neg a \lor b$ is a right and left hom from $a$ to $b$, witnessed by $a \land (\neg a \lor b) = a \land b \leq b$, so in particular $\rimp{a}{a} = \limp{a}{a} = \neg a \lor a = \top$ for all $a$; but $\top$ is the only tensor universal 1-cell.
\end{remark}

\begin{thm} \label{thm:0representable}
Let $X$ be a unital regular poly-bicategory. The following conditions are equivalent:
\begin{enumerate}
	\item for all 0-cells $x$ of $X$, there exist a 0-cell $\overline{x}$ and a tensor universal 1-cell $e: x \to \overline{x}$;
	\item for all 0-cells $x$ of $X$, there exist a 0-cell $\overline{x}$ and a tensor universal 1-cell $e': \overline{x} \to x$;
	\item for all 0-cells $x$ of $X$, there exists a tensor unit $1_x$ on $x$.
\end{enumerate}
\end{thm}
\begin{proof}
Let $x$ be a 0-cell of $X$. If there exists a tensor unit $1_x: x \to x$, then it is clearly tensor universal, and fulfils the condition for $x$ on both sides.

Conversely, suppose $e: x \to \overline{x}$ is a tensor universal 1-cell, and define $1_x := \limp{e}{e}$. By Lemma \ref{lem:1sideunit}, $1_x$ is a tensor unit on $x$. Similarly, if $e': \overline{x} \to x$ is tensor universal, then ${1_x}':= \rimp{e'}{e'}$ is a tensor unit on $x$. This completes the proof.
\end{proof}

\begin{dfn} 
A unital regular poly-bicategory $X$ is \emph{tensor 0-representable} if it satisfies any of the equivalent conditions of Theorem \ref{thm:0representable}.
\end{dfn}

Instantiating the definitions in $\coo{X}$, we obtain dual notions of \emph{par units} $\bot_x: x \to x$, \emph{par universal 1-cells}, and \emph{par 0-representability}; the dual of Theorem \ref{thm:0representable}, that the existence of par units is equivalent to the existence of enough par universal cells, holds. 

\begin{dfn} 
A regular poly-bicategory $X$ is \emph{tensor representable} if it is tensor 0-representable and tensor 1-representable; it is \emph{par representable} if it is par 0-representable and par 1-representable. 

We say that $X$ is \emph{right closed (left closed)} if it is tensor 0-representable and right 1-closed (left 1-closed); it is \emph{right coclosed (left coclosed)} if it is par 0-representable and right 1-coclosed (left 1-coclosed).

We say that $X$ is \emph{representable} if it is tensor representable and par representable. We say that $X$ is \emph{$*$-autonomous} if it is representable, left and right closed, and left and right coclosed.
\end{dfn}

In the next section, we will see how representability conditions on $X$ induce coherent algebraic structure on $GX$. Consequently, we want to restrict our attention to morphisms that preserve this structure in a suitably weak sense: for tensors, pars, homs, and cohoms, this is achieved by requiring morphisms to preserve the corresponding universal properties of 2-cells. 

However, while Theorem \ref{thm:0representable} shows that the existence of enough universal 1-cells is equivalent to the existence of tensor units, preservation of the two classes is not equivalent for arbitrary morphisms of regular poly-bicategories: for example, if $f: X \to Y$ does not preserve any universal property of 2-cells, even if $f(1_x)$ is universal in $Y$ for a tensor unit $1_x$, it may not be ``tensor-idempotent'' in $Y$.

\begin{exm}
The regular multicategory $S$ with a single 1-cell $a$ and a single 2-cell $(\underbrace{a,\ldots,a}_n) \to (a)$ for each $n \geq 1$ is tensor representable: $a$ is the tensor unit, and all 2-cells are everywhere universal. If $X$ is a tensor representable regular multicategory whose every 1-cell is tensor universal (for example, the regular multicategory corresponding to a Picard groupoid \cite[Section 7]{benabou1967introduction}), then any morphism $f: S \to X$ preserves tensor universality of 1-cells, even though it may not send $a$ to a tensor unit.
\end{exm}

We will now show that if a morphism also preserves certain universal properties of 2-cells, it will map tensor units to tensor units.

\begin{dfn}
A morphism $f: X \to Y$ of unital regular poly-bicategories is \emph{tensor strong} if it is unital, preserves tensor universality of 1-cells, and preserves tensors of 1-cells. It is \emph{par strong} if $\coo{f}: \coo{X} \to \coo{Y}$ is tensor strong. It is \emph{strong} if it is tensor strong and par strong.
\end{dfn}

\begin{prop} \label{prop:unitpreserve1}
Let $X, Y$ be unital regular poly-bicategories, and $f: X \to Y$ a tensor strong morphism. Then $f$ maps tensor units in $X$ to tensor units in $Y$.
\end{prop}
\begin{proof}
Let $1_x: x \to x$ be a tensor unit in $X$. Then $f(1_x): f(x) \to f(x)$ is tensor universal, so we can divide it by itself, and obtain a tensor universal $\rimp{f(1_x)}{f(1_x)}$, and an everywhere universal 2-cell $p: (f(1_x), \rimp{f(1_x)}{f(1_x)}) \to (f(1_x))$, as in Theorem \ref{thm:2outof3}. 

Because $f$ preserves tensors, a 2-cell $t: (1_x, 1_x) \to (1_x)$ exhibiting the idempotence of $1_x$ is mapped to a 2-cell $f(t): (f(1_x), f(1_x)) \to (f(1_x))$ that is universal at $\bord{1}{+}$. By tensor universality of $f(1_x)$, we can also construct $q: (f(1_x), f(1_x)) \to (f(1_x) \otimes f(1_x))$ that is everywhere universal; by Lemma \ref{lem:twouniversal}, $f(t)$ is also everywhere universal, and factorising $p$ through $f(t)$, we obtain an isomorphism $f(1_x) \simeq \rimp{f(1_x)}{f(1_x)}$. Lemma \ref{lem:1sideunit} allows us to conclude that $f(1_x)$ is a tensor unit.
\end{proof}

\begin{prop} \label{prop:unitpreserve2}
Let $X, Y$ be unital regular poly-bicategories, and $f: X \to Y$ a unital morphism. Suppose that $f$ preserves right homs (or left homs) and tensor universality of 1-cells. Then $f$ maps tensor units in $X$ to tensor units in $Y$.
\end{prop}
\begin{proof}
Let $1_x: x \to x$ be a tensor unit in $X$; by assumption, $f(1_x): f(x) \to f(x)$ is tensor universal in $X$. 

Suppose that $f$ preserves right homs; then $\rimp{f(1_x)}{f(1_x)} \simeq f(\rimp{1_x}{1_x}) \simeq f(1_x)$. But by Lemma \ref{lem:1sideunit}, we know that $\rimp{e}{e}$ is a tensor unit for all tensor universal 1-cells $e$. The proof when $f$ preserves left homs is similar.
\end{proof}

There are, as usual, dual results concerning par strong morphisms, or ones that preserve right or left cohoms. 

Notice that both Proposition \ref{prop:unitpreserve1} and Proposition \ref{prop:unitpreserve2} have a ``local'' character: if a tensor unit exists on a 0-cell in $X$, it will be mapped to a tensor unit in $Y$, irrespective of any property of $Y$ (except being unital, which is the minimal assumption for most of our results). 

In order to have a converse result, that a morphism preserving tensor units also preserves tensor universality, we are forced to renounce this local character: if $e: x \to x'$ is tensor universal in $X$, even if $f: X \to Y$ preserves tensor units and all universal properties of 2-cells, there is no guarantee, for instance, that tensors of $f(e)$ with arbitrary 1-cells $a: f(x') \to y$ and $b: z \to f(x)$ of $Y$ should exist. Therefore a ``global'' representability condition must be imposed.

\begin{prop} \label{prop:isomorphism}
Let $X$ be a tensor 1-representable regular poly-bicategory. Then, for a 1-cell $e: x \to y$, the following are equivalent:
\begin{enumerate}
	\item $e$ is tensor universal;
	\item there exist tensor units $1_x$ on $x$ and $1_y$ on $y$, a 1-cell $e^*: y \to x$, and 2-cells exhibiting $e \otimes e^* \simeq 1_x$ and $e^* \otimes e \simeq 1_y$.
\end{enumerate}
\end{prop}
\begin{proof}
Suppose $e: x \to y$ is tensor universal; by Lemma \ref{lem:1sideunit}, we know that there are tensor units $1_x$ on $x$ and $1_y$ on $y$. Dividing $1_x$ and $1_y$ by $e$, we obtain 1-cells $\rimp{e}{1_x}$ and $\limp{e}{1_y}$, and 2-cells exhibiting $e \otimes (\rimp{e}{1_x}) \simeq 1_x$, and $(\limp{e}{1_y}) \otimes e \simeq 1_y$. A standard argument then shows that $\rimp{e}{1_x} \simeq \limp{e}{1_y}$, so we can call either of them $e^*$. This proves one implication.

Now, suppose there are tensor units $1_x$ on $x$ and $1_y$ on $y$, a 1-cell $e^*: y \to x$, and 2-cells $h: (e,e^*) \to (1_x)$ and $h': (e^*, e) \to (1_y)$, both universal at $\bord{1}{+}$. We will show that $e$ is tensor left universal. 

Let $a: y \to z$ be an arbitrary 1-cell; by tensor 1-representability, there is a tensor $t: (e,a) \to (e\otimes a)$. We use it to construct other tensors of the same type, by the following sequence of factorisations (as usual, we omit labels of 0-cells):
\begin{equation} \label{eq:isomorphism_proof1}
\input{img/s3_isomorphism_proof1.tex}
\end{equation} 
for a unique 2-cell $s$, universal at $\bord{1}{+}$; then,
\begin{equation} \label{eq:isomorphism_proof2}
\input{img/s3_isomorphism_proof2.tex}
\end{equation} 
for a unique 2-cell $t'$, universal at $\bord{1}{+}$; finally,
\begin{equation} \label{eq:isomorphism_proof3}
\input{img/s3_isomorphism_proof3.tex}
\end{equation} 
for a unique $\tilde{t}$, universal at $\bord{1}{+}$. All these universal properties follow from Lemma \ref{lem:comp_divis_closure}. We will show that $\tilde{t}$ is also universal at $\bord{2}{-}$ (which by Lemma \ref{lem:twouniversal} implies that $t$ and $t'$ are, too).

Let $p: (e, \Gamma) \to (e \otimes a, \Delta)$ be a 2-cell, and precompose it with $l_e: (1_x, e) \to (e)$ and then $h: (e, e^*) \to (1_x)$. We have the following sequence of factorisations:
\begin{equation*}
\input{img/s3_isomorphism_proof4.tex}
\end{equation*} 
\begin{equation*}
\input{img/s3_isomorphism_proof5.tex}
\end{equation*} 
\begin{equation*}
\input{img/s3_isomorphism_proof6.tex}
\end{equation*} 
for some unique $p'$, $p''$, and $\tilde{p}$. Comparing the first and the last diagram, and cancelling $l_e$ and $h$, we obtain a factorisation of $p$ through $\tilde{t}$, which is easily determined to be unique.

Next, let $b: x \to z'$ be another 1-cell. By tensor 1-representability, there is a tensor $u: (e^*, b) \to (e^* \otimes b)$; then the factorisation
\begin{equation*}
\input{img/s3_isomorphism_proof7.tex}
\end{equation*} 
produces a 2-cell $v: (e, e^*\otimes b) \to (b)$ that is universal at $\bord{1}{+}$. Proceeding as in the first part of the proof, we can show that $v$ is also universal at $\bord{2}{-}$. This proves that $e$ is tensor left universal. 

Due to the symmetry between $e$ and $e^*$ in the hypothesis, we immediately obtain that $e^*$ is also tensor left universal; then, a dual argument in $\opp{X}$ shows that both $e$ and $e^*$ are also tensor right universal. This completes the proof.
\end{proof}

\begin{prop} \label{prop:isoclosed}
Let $X$ be a right and left 1-closed regular poly-bicategory. Then, for a 1-cell $e: x \to y$, the following are equivalent:
\begin{enumerate}
	\item $e$ is tensor universal;
	\item there exist tensor units $1_x$ on $x$ and $1_y$ on $y$, a 1-cell $e^*: y \to x$, and 2-cells exhibiting $e \otimes e^* \simeq 1_x$ and $e^* \otimes e \simeq 1_y$.
\end{enumerate}
\end{prop}
\begin{proof}
The implication from $(a)$ to $(b)$ is proved in the same way as in Proposition \ref{prop:isomorphism}. For the other implication, suppose there are tensor units $1_x$ on $x$ and $1_y$ on $y$, a 1-cell $e^*: y \to x$, and 2-cells $h: (e,e^*) \to (1_x)$ and $h': (e^*, e) \to (1_y)$, both universal at $\bord{1}{+}$. We will show that $e$ is tensor left universal.

Let $a: x \to z$ be a 1-cell; since $X$ is right closed, it has a right hom $t: (e, \rimp{e}{a}) \to (a)$. We use it to construct other right homs of the same type, as follows. Factorise
\begin{equation}
\input{img/s3_isoclosed_proof1.tex}
\end{equation} 
for a unique 2-cell $u$, universal at $\bord{2}{-}$; then,
\begin{equation} \label{eq:isoclosed2}
\input{img/s3_isoclosed_proof2.tex}
\end{equation} 
for a unique 2-cell $t'$, universal at $\bord{2}{-}$; finally,
\begin{equation} \label{eq:isoclosed3}
\input{img/s3_isoclosed_proof3.tex}
\end{equation}
for a unique 2-cell $\tilde{t}$, also universal at $\bord{2}{-}$. All these universal properties follow from Lemma \ref{lem:comp_divis_closure}. We will show that $\tilde{t}$ is also universal at $\bord{1}{+}$, implying that $t$ and $t'$ are, too.

Let $p: (\Gamma_1,e,\rimp{e}{a},\Gamma_2) \to (\Delta)$ be a 2-cell, and precompose it with $r_e: (e,1_y) \to (e)$ and with $h': (e^*,e) \to (1_y)$. We have the following sequence of factorisations:
\begin{equation*}
\input{img/s3_isoclosed_proof4.tex}
\end{equation*}
\begin{equation*}
\input{img/s3_isoclosed_proof5.tex}
\end{equation*}
\begin{equation*}
\input{img/s3_isoclosed_proof6.tex}
\end{equation*}
for some unique $p'$, $p''$, and $\tilde{p}$. Comparing the first and the last diagram, and cancelling $r_e$ and $h'$, we obtain a unique factorisation of $p$ through $\tilde{t}$.

Next, let $b: y \to z'$ be another 1-cell, and take a right hom $u: (e^*,\rimp{e^*}{b}) \to (b)$. The factorisation
\begin{equation*}
\input{img/s3_isoclosed_proof7.tex}
\end{equation*}
produces a 2-cell $v$, universal at $\bord{2}{-}$. Proceeding as in the first part of the proof, we can show that $v$ is also universal at $\bord{1}{+}$. This proves that $e$ is tensor left universal. 

Due to the symmetry between $e$ and $e^*$ in the hypothesis, we immediately obtain that $e^*$ is also tensor left universal; then, a dual argument in $\opp{X}$, which is right closed because $X$ is also left closed, shows that both $e$ and $e^*$ are also tensor right universal. This completes the proof.
\end{proof}

\begin{remark}
The proofs show that the implication from $(a)$ to $(b)$ holds in any unital regular poly-bicategory: that is, a tensor universal 1-cell is always a ``tensor equivalence'', in the sense of having an ``inverse up to isomorphism'' with respect to a tensor unit. The converse, however, only holds in a tensor 1-representable, or a right and left 1-closed regular poly-bicategory.
\end{remark}

\begin{cor} \label{cor:isomorphism_cor}
Let $X, Y$ be unital regular poly-bicategories, and let $f: X \to Y$ be a morphism that is unital, preserves tensor units, and preserves tensors of 1-cells. Suppose that $Y$ is tensor 1-representable, or that $Y$ is right and left 1-closed. Then $f$ is tensor strong.
\end{cor}
\begin{proof}
Let $e: x \to y$ be a tensor universal 1-cell in $X$; by Lemma \ref{lem:1sideunit}, we know that there are tensor units $1_x$ on $x$ and $1_y$ on $y$. Dividing $1_x$ and $1_y$ by $e$, as in the proof of Proposition \ref{prop:isomorphism}, we obtain a 1-cell $e^*: y \to x$ and 2-cells exhibiting $e \otimes e^* \simeq 1_x$, and $e^* \otimes e \simeq 1_y$. 

Because $f$ preserves tensors, there are 2-cells of $Y$ that exhibit $f(e) \otimes f(e^*) \simeq f(1_x)$ and $f(e^*) \otimes f(e) \simeq f(1_y)$, and because $f$ preserves tensor units, $f(1_x)$ and $f(1_y)$ are tensor units in $Y$. Proposition \ref{prop:isomorphism} or \ref{prop:isoclosed} allows us to conclude that $f(e)$ is tensor universal.
\end{proof}

In particular, when restricting to tensor 1-representable, or to left and right closed regular poly-bicategories, we can replace the condition of ``preserving tensor universality'' with the condition of ``preserving tensor units'' in the definition of a tensor strong morphism. 

However, for morphisms that preserve homs, but not tensors, preservation of tensor units is strictly weaker than preservation of tensor universal 1-cells, as shown by the following example.
\begin{exm}
Let $X$ be the regular multicategory with two 1-cells, $0$ and $1$, and a unique 2-cell $(b_1,\ldots,b_n) \to (b)$ whenever $\sum_{i=1}^n b_i \equiv b\,\mathrm{mod}\,2$, for all $b_i, b \in \{0,1\}$. Then $X$ has $0$ as a tensor unit, it is tensor representable, and it is closed on both sides, with
\begin{equation*}
	b \otimes b' = \limp{b}{b'} = \rimp{b}{b'} := \begin{cases}	0 & \text{if } b = b', \\
	1 & \text{if } b \neq b', \end{cases}
\end{equation*}
for all $b, b' \in \{0,1\}$.

Let $Y$ be the regular multicategory with $\mathbb{N}$ as the set of 1-cells, and a unique 2-cell $(k_1, \ldots, k_n) \to (k)$ whenever $\sum_{i=1}^n k_i = k + 2m$ for some $m \geq 0$, for all $k_i, k \in \mathbb{N}$. Then $Y$ has $0$ as a tensor unit, it is tensor representable, and it is closed on both sides, with
\begin{align*}
	k \otimes j & \; := k+j, \\
	\limp{k}{j} = \rimp{k}{j} & \; := \begin{cases} j-k & \text{if } k < j, \\
	(k-j)\,\mathrm{mod}\,2 & \text{if } k \geq j, \end{cases}
\end{align*}
for all $k, j \in \mathbb{N}$.

Consider the obvious inclusion of multicategories $\imath: X \to Y$. This preserves the tensor unit and all homs in $X$, but does not preserve tensors, as $\imath(1 \otimes 1) = 0$ is not isomorphic to $\imath(1) \otimes \imath(1) = 2$. Moreover, $1$ is tensor universal in $X$, but it is not tensor universal in $Y$, as there are no 2-cells $(0) \to (1 \otimes n)$ for any $n$ in $Y$.
\end{exm}

Thus, there is no converse to Proposition \ref{prop:unitpreserve2}, and it makes sense to distinguish between closed morphisms that preserve units, and others that preserve universal 1-cells. We will only name the first, as we do not know if the latter have any significance.
\begin{dfn}
A morphism $f: X \to Y$ of unital regular poly-bicategories is \emph{right closed} if it is unital, preserves tensor units, and preserves right homs. It is \emph{left closed} if $\opp{f}$ is right closed, and \emph{right coclosed (left coclosed)} if $\coo{f}$ is right closed (left closed).
\end{dfn}

We now define some categories on which we will focus in the next section. 
\begin{dfn}
We write $\mbcat_\otimes$ for the subcategory of tensor representable regular multi-bicategories and tensor strong morphisms in $\mbcat$, and $\mcat_\otimes$ for its full subcategory on regular multicategories. 

We write $\pbcat_\otimes^\parr$ for the subcategory of representable regular poly-bicategories and strong morphisms in $\pbcat$, and $\pcat_\otimes^\parr$ for its full subcategory on regular polycategories. 

We write $\pcat_*$ for the subcategory of $*$-autonomous regular polycategories and strong morphisms in $\pcat$.
\end{dfn} 

\begin{remark}
While we prefer to keep foundational aspects at a minimum, it is due remarking that the ``coherence via universality'' theorems concern the production of a uniformly defined coherent structure from a representability property, given as a ``for all... there exists'' statement; some form of choice is therefore required, which may or may not be problematic depending on one's preferred foundations.
\end{remark}

\section{Coherence for units via universality} \label{sec:coherence}

The definition of tensor 0-representability asks that ``locally'' for each 0-cell $x$, some 1-cell $1_x: x \to x$ exists, and for each 1-cell $a: x \to y$ and $b: z \to x$, some 2-cells $l_a$, $r_b$ exist and are appropriately universal; no compatibility between them is required. The following result shows that any uniform choice of units and witnesses of unitality can be made to satisfy poly-bicategorical versions of ``naturality'' and of ``triangle'' equations.
\begin{thm} \label{thm:polycoherent}
Let $X$ be a tensor 0-representable regular poly-bicategory, and suppose tensor units $\{1_x: x \to x\}$ have been chosen for all $x \in X_0$. Then there exist witnesses $\{\tilde{l}_a: (1_x,a) \to (a), \tilde{r}_a: (a,1_y) \to (a)\}$ of unitality, indexed by 1-cells $a: x \to y$ of $X$, such that for all 2-cells $p, p', q$ as pictured, the following equations hold:
\begin{equation} \label{eq:natural}
\input{img/s4_natural.tex}
\end{equation}
\begin{equation} \label{eq:natural2}
\input{img/s4_natural2.tex}
\end{equation}
\begin{equation} \label{eq:trianglepoly}
\input{img/s4_trianglepoly.tex}
\end{equation}
\end{thm}
\begin{proof}
Pick arbitrary witnesses of unitality $\{l_a, r_a\}$ for each 1-cell of $X$; by Theorem \ref{thm:2outof3}, for all units $1_x$, we can assume that $l_{1_x} = r_{1_x}: (1_x,1_x) \to (1_x)$, since they are both everywhere universal. Then, we define a new family $\{\tilde{l}_a, \tilde{r}_a\}$ as follows: let $\tilde{l}_a := \cutt{1,1}(l_a,e_a)$, where $e_a: (a) \to (a)$ is the unique automorphism obtained by the factorisation
\begin{equation*} 
\input{img/s4_natural_factor.tex}
\end{equation*}
and $\tilde{r}_a := \cutt{1,1}(r_a,e'_a)$, where $e'_a: (a) \to (a)$ is obtained by the factorisation
\begin{equation*} 
\input{img/s4_natural_factor2.tex}
\end{equation*}
To prove equation (\ref{eq:natural}), precompose the left-hand side with $l_{1_x}$. We have two possible factorisations: 
\begin{equation*} 
\input{img/s4_natural_proof1.tex}
\end{equation*}
by definition of $\tilde{l}_b$; and, for a unique $\tilde{p}$ given by universality of $l_a$ at $\bord{1}{+}$,
\begin{equation*} 
\input{img/s4_natural_proof2.tex}
\end{equation*}
\begin{equation*} 
\input{img/s4_natural_proof3.tex}
\end{equation*}
Comparing the two, and using the cancellability of $\tilde{l}_b = \cutt{1,1}(l_b,e_b)$ at $\bord{2}{-}$, we obtain equation (\ref{eq:natural}). Equation (\ref{eq:natural2}) is proved similarly.

For equation (\ref{eq:trianglepoly}), we adapt the argument of \cite[Proposition 2.6]{kock2008elementary}. Consider the following two factorisations: first,
\begin{equation*} 
\input{img/s4_triangle_proof1.tex}
\end{equation*}
\begin{equation*} 
\input{img/s4_triangle_proof2.tex}
\end{equation*}
for a unique $\tilde{q}$ given by universality of $\tilde{l}_d$ at $\bord{1}{+}$, where the second equation uses the definition of $\tilde{l}_d$; symmetrically,
\begin{equation*} 
\input{img/s4_triangle_proof3.tex}
\end{equation*}
\begin{equation*} 
\input{img/s4_triangle_proof4.tex}
\end{equation*}
for some unique $\overline{q}$ given by universality of $\tilde{r}_a$ at $\bord{1}{+}$. Comparing the two, and using the cancellability of $l_{1_y}$, we obtain equation (\ref{eq:trianglepoly}).
\end{proof}

In order to fix the notation, we recall the definition of bicategory \cite{benabou1967introduction}.
\begin{dfn}
A \emph{bicategory} $B$ is a 2-graph
\begin{equation*}
\begin{tikzpicture}
	\node[scale=1.25] (0) at (0,0) {$B_0$};
	\node[scale=1.25] (1) at (2.5,0) {$B_1$};
	\node[scale=1.25] (2) at (5,0) {$B_2,$};
	\draw[1c] (1.west |- 0,.15) to node[auto,swap] {$\bord{}{+}$} (0.east |- 0,.15);
	\draw[1c] (1.west |- 0,-.15) to node[auto] {$\bord{}{-}$} (0.east |- 0,-.15);
	\draw[1c] (2.west |- 0,.15) to node[auto,swap] {$\bord{}{+}$} (1.east |- 0,.15);
	\draw[1c] (2.west |- 0,-.15) to node[auto] {$\bord{}{-}$} (1.east |- 0,-.15);
\end{tikzpicture}
\end{equation*}
together with the following data, where we write $p: a \to b$ for a 1-cell or 2-cell with $\bord{}{-}p = a$ and $\bord{}{+}p = b$:
\begin{enumerate}
	\item a family of 2-cells, the \emph{vertical composites} $\{p;q: a \to c\}$, indexed by composable pairs of 2-cells $p: a \to b$ and $q: b \to c$;
	\item a family of 2-cells, the \emph{vertical units} $\{\idd{a}: a \to a\}$, indexed by the 1-cells of $B$;
	\item a family of 1-cells $\{a \otimes b: x \to z\}$, indexed by composable pairs of 1-cells $a: x \to y$, $b: y \to z$, and a family of 2-cells $\{p \otimes q: a \otimes c \to b \otimes d\}$, indexed by pairs of 2-cells $p: a \to b$, $q: c \to d$ such that $\bord{}{+}\bord{}{+}p = \bord{}{-}\bord{}{-}q$, both called \emph{horizontal composites};
	\item a family of 1-cells, the \emph{horizontal units} $\{1_x: x \to x\}$, indexed by the 0-cells of $B$;
	\item a family of 2-cells, the \emph{associators} $\{\alpha_{a,b,c}: (a \otimes b) \otimes c \to a \otimes (b \otimes c)\}$, indexed by composable triples of 1-cells $a: x \to y$, $b: y \to z$, $c: z \to w$;
	\item two families of 2-cells, the \emph{left unitors} $\{\lambda_a: 1_x \otimes a \to a\}$ and the \emph{right unitors} $\{\rho_a: a \otimes 1_y \to a\}$, indexed by 1-cells $a: x \to y$.
\end{enumerate}
These are subject to the following conditions:
\begin{enumerate}
	\item vertical composition is associative, and unital with the vertical units, that is, $(p;q);r = p;(q;r)$, and $p;\idd{b} = p = \idd{a};q$, whenever both sides make sense;
	\item horizontal composition is natural with respect to vertical composition and units, that is, $(p_1;p_2) \otimes (q_1;q_2) = (p_1 \otimes q_1);(p_2 \otimes q_2)$ and $\idd{a} \otimes \idd{b} = \idd{a \otimes b}$, whenever the left-hand side makes sense;
	\item the associators and the unitors are natural in their parameters, that is, for all $p: a \to a'$, $q: b \to b'$, and $r: c \to c'$, the following diagrams commute:
	\begin{equation*}
\begin{tikzpicture}[baseline={([yshift=-.5ex]current bounding box.center)}]
\begin{scope}
	\node[scale=1.25] (0) at (-1.5,.75) {$(a \otimes b) \otimes c$};
	\node[scale=1.25] (1) at (2,.75) {$a \otimes (b \otimes c)$};
	\node[scale=1.25] (2) at (-1.5,-.75) {$(a' \otimes b') \otimes c'$};
	\node[scale=1.25] (3) at (2,-.75) {$a' \otimes (b' \otimes c')$};
	\draw[1c] (0.east) to node[auto] {$\alpha_{a,b,c}$} (1.west);
	\draw[1c] (2.east) to node[auto,swap] {$\alpha_{a',b',c'}$} (3.west);
	\draw[1c] (0.south) to node[auto,swap] {$(p \otimes q) \otimes r$} (2.north);
	\draw[1c] (1.south) to node[auto] {$p \otimes (q \otimes r)$} (3.north);
	\node[scale=1.25] at (3.5,-.9) {,};
\end{scope}
\begin{scope}[shift={(6,0)}]
	\node[scale=1.25] (0) at (-1,.75) {$1_x \otimes a$};
	\node[scale=1.25] (1) at (1,.75) {$a$};
	\node[scale=1.25] (0') at (3,.75) {$a \otimes 1_y$};
	\node[scale=1.25] (2) at (-1,-.75) {$1_x \otimes a'$};
	\node[scale=1.25] (3) at (1,-.75) {$a'$};
	\node[scale=1.25] (2') at (3,-.75) {$a' \otimes 1_y$};
	\draw[1c] (0.east) to node[auto] {$\lambda_a$} (1.west);
	\draw[1c] (2.east) to node[auto,swap] {$\lambda_{a'}$} (3.west);
	\draw[1c] (0'.west) to node[auto, swap] {$\rho_a$} (1.east);
	\draw[1c] (2'.west) to node[auto] {$\rho_{a'}$} (3.east);
	\draw[1c] (0.south) to node[auto,swap] {$1_x \otimes p$} (2.north);
	\draw[1c] (0'.south) to node[auto] {$p \otimes 1_y$} (2'.north);
	\draw[1c] (1.south) to node[auto] {$p$} (3.north);
\end{scope}
\end{tikzpicture}
\end{equation*}
	commute;
	\item the associators and unitors satisfy the pentagon and triangle equations, that is, for all $a: x \to y$, $b: y \to z$, $c: z \to w$, $d: w \to v$, the following diagrams commute:
	\begin{equation*}
\begin{tikzpicture}[baseline={([yshift=-.5ex]current bounding box.center)}]
	\node[scale=1.25] (0) at (-5,.75) {$((a \otimes b) \otimes c) \otimes d$};
	\node[scale=1.25] (1) at (5,.75) {$(a \otimes b) \otimes (c \otimes d)$};
	\node[scale=1.25] (2) at (-5,-.75) {$(a \otimes (b \otimes c)) \otimes d$};
	\node[scale=1.25] (2b) at (0,-.75) {$a \otimes ((b \otimes c) \otimes d)$};
	\node[scale=1.25] (3) at (5,-.75) {$a \otimes (b \otimes (c \otimes d))$};
	\draw[1c] (0.east) to node[auto] {$\alpha_{a\otimes b,c,d}$} (1.west);
	\draw[1c] (2.east) to node[auto,swap] {$\alpha_{a,b\otimes c, d}$} (2b.west);
	\draw[1c] (2b.east) to node[auto,swap] {$\idd{a} \otimes \alpha_{b,c,d}$} (3.west);
	\draw[1c] (0.south) to node[auto,swap] {$\alpha_{a,b,c} \otimes \idd{d}$} (2.north);
	\draw[1c] (1.south) to node[auto] {$\alpha_{a,b,c \otimes d}$} (3.north);
	\node[scale=1.25] at (6.5,-.9) {,};
\end{tikzpicture}
\end{equation*}
	\begin{equation*}
\begin{tikzpicture}[baseline={([yshift=-.5ex]current bounding box.center)}]
	\node[scale=1.25] (0) at (-2.5,.75) {$(a \otimes 1_y) \otimes b$};
	\node[scale=1.25] (1) at (2.5,.75) {$a \otimes (1_y \otimes b)$};
	\node[scale=1.25] (3) at (2.5,-.75) {$a \otimes b$};
	\draw[1c] (0.east) to node[auto] {$\alpha_{a,1_y,b}$} (1.west);
	\draw[1c] (0) to node[auto,swap] {$\rho_a \otimes \idd{b}$} (3);
	\draw[1c] (1.south) to node[auto] {$\idd{a} \otimes \lambda_b$} (3.north);
	\node[scale=1.25] at (3.5,-.8) {;};
\end{tikzpicture}
\end{equation*}
	\item the associators and unitors are isomorphisms, that is, they have inverses with respect to vertical composition.
\end{enumerate}
We say that $B$ is \emph{strict} if the associators and unitors are vertical units. We say that $B$ is \emph{strictly associative} if the associators are vertical units. A bicategory with a single 0-cell is called a \emph{monoidal category}.

Given two bicategories $B$, $C$, a \emph{functor} $f: B \to C$ is a morphism of the underlying 2-graphs that commutes with vertical composition and units, together with
\begin{enumerate}
	\item a family of isomorphisms in $C$ $\{f_{a,b}: f(a) \otimes f(b) \to f(a\otimes b)\}$, indexed by composable pairs of 1-cells $a: x \to y$, $b: y \to z$ of $B$, and
	\item a family of isomorphisms in $C$ $\{f_x: 1_{f(x)} \to f(1_x)\}$, indexed by 0-cells $x$ of $B$,
\end{enumerate}
where the first one is natural in its parameters, that is, for all $p: a \to a'$, $q: b \to b'$, the diagram
\begin{equation*}
\begin{tikzpicture}[baseline={([yshift=-.5ex]current bounding box.center)}]
	\node[scale=1.25] (0) at (-1.5,.75) {$f(a) \otimes f(b)$};
	\node[scale=1.25] (1) at (2,.75) {$f(a \otimes b)$};
	\node[scale=1.25] (2) at (-1.5,-.75) {$f(a') \otimes f(b')$};
	\node[scale=1.25] (3) at (2,-.75) {$f(a' \otimes b')$};
	\draw[1c] (0.east) to node[auto] {$f_{a,b}$} (1.west);
	\draw[1c] (2.east) to node[auto,swap] {$f_{a',b'}$} (3.west);
	\draw[1c] (0.south) to node[auto,swap] {$f(p) \otimes f(q)$} (2.north);
	\draw[1c] (1.south) to node[auto] {$f(p \otimes q)$} (3.north);
\end{tikzpicture}
\end{equation*}
commutes, and both families are compatible with the associators and unitors, in the sense that, for all $a: x \to y$, $b: y \to z$, $c: z \to w$, the following diagrams commute:
		\begin{equation*}
\begin{tikzpicture}[baseline={([yshift=-.5ex]current bounding box.center)}]
	\node[scale=1.25] (0) at (-5,.75) {$(f(a) \otimes f(b)) \otimes f(c)$};
	\node[scale=1.25] (0b) at (0,.75) {$f(a) \otimes (f(b) \otimes f(c))$};
	\node[scale=1.25] (1) at (5,.75) {$f(a) \otimes f(b \otimes c)$};
	\node[scale=1.25] (2) at (-5,-.75) {$f(a \otimes b) \otimes f(c)$};
	\node[scale=1.25] (2b) at (0,-.75) {$f((a \otimes b) \otimes c)$};
	\node[scale=1.25] (3) at (5,-.75) {$f(a \otimes (b \otimes c))$};
	\draw[1c] (0.east) to node[auto] {$\alpha_{f(a),f(b),f(c)}$} (0b.west);
	\draw[1c] (0b.east) to node[auto] {$\idd{f(a)} \otimes f_{b,c}$} (1.west);
	\draw[1c] (2.east) to node[auto,swap] {$f_{a \otimes b,c}$} (2b.west);
	\draw[1c] (2b.east) to node[auto,swap] {$f(\alpha_{a,b,c})$} (3.west);
	\draw[1c] (0.south) to node[auto,swap] {$f_{a,b} \otimes \idd{f(c)}$} (2.north);
	\draw[1c] (1.south) to node[auto] {$f_{a,b\otimes c}$} (3.north);
	\node[scale=1.25] at (6.5,-.9) {,};
\end{tikzpicture}
\end{equation*}
\begin{equation} \label{eq:functor_unitor}
\begin{tikzpicture}[baseline={([yshift=-.5ex]current bounding box.center)}]
\begin{scope}
	\node[scale=1.25] (0) at (-4,.75) {$1_{f(x)} \otimes f(a)$};
	\node[scale=1.25] (1b) at (0,.75) {$f(1_x) \otimes f(a)$};
	\node[scale=1.25] (1) at (4,.75) {$f(1_x \otimes a)$};
	\node[scale=1.25] (3) at (4,-.75) {$f(a)$};
	\draw[1c] (0.east) to node[auto] {$f_x \otimes \idd{f(a)}$} (1b.west);
	\draw[1c] (1b.east) to node[auto] {$f_{1_x,a}$} (1.west);
	\draw[1c] (0) to node[auto,swap] {$\lambda_{f(a)}$} (3);
	\draw[1c] (1.south) to node[auto] {$f(\lambda_a)$} (3.north);
	\node[scale=1.25] at (5,-.9) {,};
\end{scope}
\end{tikzpicture}
\end{equation}
\begin{equation} \label{eq:functor_unitor2}
\begin{tikzpicture}[baseline={([yshift=-.5ex]current bounding box.center)}]
\begin{scope}
	\node[scale=1.25] (0) at (-4,.75) {$f(a) \otimes 1_{f(y)}$};
	\node[scale=1.25] (1b) at (0,.75) {$f(a) \otimes f(1_y)$};
	\node[scale=1.25] (1) at (4,.75) {$f(a \otimes 1_y)$};
	\node[scale=1.25] (3) at (4,-.75) {$f(a)$};
	\draw[1c] (0.east) to node[auto] {$\idd{f(a)} \otimes f_y$} (1b.west);
	\draw[1c] (1b.east) to node[auto] {$f_{a,1_y}$} (1.west);
	\draw[1c] (0) to node[auto,swap] {$\rho_{f(a)}$} (3);
	\draw[1c] (1.south) to node[auto] {$f(\rho_a)$} (3.north);
	\node[scale=1.25] at (5,-.9) {.};
\end{scope}
\end{tikzpicture}
\end{equation}
Bicategories and functors form a large category $\bicat$. We write $\moncat$ for its full subcategory on monoidal categories.
\end{dfn}

\begin{cons} \label{cons:cvu-bicat}
Let $X$ be a tensor representable regular poly-bicategory. We endow the induced 2-graph $GX$ with families of cells as required by the definition of a bicategory:
\begin{enumerate}
	\item for each composable pair of 2-cells $p: (a) \to (b)$ and $q: (b) \to (c)$, let $p;q$ be $\cutt{1,1}(p,q)$;
	\item for each 1-cell $a$, let $\idd{a}: a \to a$ be the unique unit on $a$, whose existence is guaranteed by $X$ being unital;
	\item for each composable pair of 1-cells $a: x \to y$, $b: y \to z$, choose a tensor $t_{a,b}: (a, b) \to (a \otimes b)$, and let $a \otimes b$ be their horizontal composite; for each pair of 2-cells $p: (a) \to (c)$, $q: (b) \to (d)$, such that $\bord{}{+}\bord{}{+}p = \bord{}{-}\bord{}{-}q$, let $p \otimes q$ be the unique 2-cell obtained by the factorisation 
\begin{equation*} 
\input{img/s4_horizontal_tensor.tex}
\end{equation*}
	\item for each 0-cell $x$, choose a tensor unit $1_x: x \to x$ as a horizontal unit;
	\item for each composable triple of 1-cells $a, b, c$, let $\alpha_{a,b,c}$ be the unique isomorphism obtained by the factorisation
\begin{equation*} 
\input{img/s4_associator.tex}
\end{equation*}
	\item choose a coherent family of witnesses of unitality $\{l_a, r_a\}$ as in Theorem \ref{thm:polycoherent}, and let $\lambda_a$ and $\rho_a$ be the unique isomorphisms obtained by the factorisations
\begin{equation*} 
\input{img/s4_unitor.tex}
\end{equation*}
\begin{equation*} 
\input{img/s4_unitor2.tex}
\end{equation*}
\end{enumerate}
Now, let $Y$ be another tensor representable regular poly-bicategory, and $f: X \to Y$ a tensor strong morphism. Having picked families of cells for $GY$ as done for $GX$, we endow $Gf: GX \to GY$ with the structure required by the definition of a functor, as follows:
\begin{enumerate}
	\item for each composable pair of 1-cells $a,b$ in $X$, let $f_{a,b}$ be the unique 2-cell of $Y$ obtained by the factorisation
\begin{equation*} 
\input{img/s4_functor_compositor.tex}
\end{equation*}
	which is an isomorphism because $f(t_{a,b})$ is also universal at $\bord{1}{+}$;
	\item for each 0-cell $x$ of $X$, let $f_x$ be the inverse of the 2-cell $\bar{f}_x$ obtained by the factorisation
	\begin{equation} \label{eq:functor_unitor_def}
\input{img/s4_functor_unitor.tex}
	\end{equation}
	Since $f$ is tensor strong, $f(1_x)$ is tensor universal, and because $f(l_{1_x})$ is universal at $\bord{1}{+}$, it is in fact everywhere universal; this implies that $\bar{f}_x$ is indeed an isomorphism.
\end{enumerate}
\end{cons}

\begin{prop} \label{prop:cvu-bicat}
Let $X$ be a tensor representable regular poly-bicategory. Then $GX$ with the structure defined in Construction \ref{cons:cvu-bicat} is a bicategory.

If $f: X \to Y$ is a tensor strong morphism of tensor representable regular poly-bicategories, $Gf: GX \to GY$ with the structure defined in Construction \ref{cons:cvu-bicat} is a functor of bicategories. 
\end{prop}
\begin{proof}
The conditions relative to vertical composition and coherent associativity of the horizontal composition are handled exactly as in \cite[Definition 9.6]{hermida2000representable}, so we only need to show that the unitors are natural in their parameters, and that they satisfy the triangle equations. By our definition of the unitors, these are all simple consequences of Theorem \ref{thm:polycoherent}.

Now, if $f: X \to Y$ is a tensor strong morphism, the fact that $Gf: GX \to GY$ is compatible with the associators is shown as in \cite[Proposition 9.7]{hermida2000representable}. For compatibility with the unitors, observe that, by our definition of $f_x$ (omitting labels of 1-cells),
\begin{equation*} 
\input{img/s4_functor_unitor2.tex}
\end{equation*}
where the first equation is an instance of (\ref{eq:natural2}), as $l_{1_y} = r_{1_y}$ for all $y$. This implies that, in $GY$, the diagram
\begin{equation*}
\begin{tikzpicture}[baseline={([yshift=-.5ex]current bounding box.center)}]
	\node[scale=1.25] (0) at (-3.5,.75) {$1_{f(x)} \otimes 1_{f(x)}$};
	\node[scale=1.25] (1) at (3.5,.75) {$1_{f(x)}$};
	\node[scale=1.25] (2) at (-3.5,-.75) {$f(1_x) \otimes f(1_x)$};
	\node[scale=1.25] (2b) at (0,-.75) {$f(1_x \otimes 1_x)$};
	\node[scale=1.25] (3) at (3.5,-.75) {$f(1_x)$};
	\draw[1c] (0.east) to node[auto] {$\lambda_{1_{f(x)}}$} (1.west);
	\draw[1c] (2.east) to node[auto,swap] {$f_{1_x,1_x}$} (2b.west);
	\draw[1c] (2b.east) to node[auto,swap] {$f(\lambda_{1_x})$} (3.west);
	\draw[1c] (0.south) to node[auto,swap] {$f_x \otimes f_x$} (2.north);
	\draw[1c] (1.south) to node[auto] {$f_x$} (3.north);
\end{tikzpicture}
\end{equation*}
commutes; by \cite[Proposition 3.5]{kock2008elementary}, this suffices to prove that all the diagrams (\ref{eq:functor_unitor}) and (\ref{eq:functor_unitor2}) commute.
\end{proof}
\begin{cor}
A choice of structure as in Construction \ref{cons:cvu-bicat} for each tensor representable regular multi-bicategory $X$ and tensor strong morphism $f: X \to Y$ determines a functor $G: \mbcat_\otimes \to \bicat$, restricting to $G: \mcat_\otimes \to \moncat$.
\end{cor}

The inverse construction, from bicategories to tensor representable regular multi-bicategories, does not present any significant difference compared to Hermida's \cite[Definition 9.2]{hermida2000representable}, so we will treat it only briefly, skipping the details. In what follows, we only need a partial version of the coherence theorem for bicategories and functors. We write $[a_1 \otimes \ldots \otimes a_n]_\beta$ for a \emph{bracketing} $\beta$ of $a_1 \otimes \ldots \otimes a_n$; for example, we can have $[a \otimes b \otimes c]_\beta = (a\otimes b) \otimes c$ or $a \otimes (b \otimes c)$.

\begin{thm}\emph{\cite[Corollary 1.8]{joyal1993braided}} \label{thm:functocoherence}
Let $f: B \to C$ be a functor of bicategories. Then, for any sequentially composable sequence $a_1, \ldots, a_n$ of 1-cells of $B$, and any pair of bracketings $\beta, \gamma$ of $n$ elements, there is a unique 2-cell built from
\begin{enumerate}
	\item vertical units, associators and their inverses in $C$,
	\item the images through $f$ of the associators and their inverses in $B$, and
	\item the structural 2-cells $f_{a,b}$ of $f$
\end{enumerate}
between $f([a_1 \otimes \ldots \otimes a_n]_\beta)$ and $[f(a_1) \otimes \ldots \otimes f(a_n)]_\gamma$.
\end{thm}
\begin{cor}\emph{\cite[Theorem 3.1]{maclane1963natural}} \label{cor:assocoherence}
Let $B$ be a bicategory. Then, for any sequentially composable sequence $a_1, \ldots, a_n$ of 1-cells of $B$, there is a unique 2-cell built from vertical units, associators and their inverses between any two bracketings of $a_1 \otimes \ldots \otimes a_n$.
\end{cor}

\begin{cons} \label{cons:groth_monoidal}
Let $B$ be a bicategory; we define a regular multi-bicategory $\int\! B$ as follows. The 0-cells and 1-cells of $\int\! B$ are the same as the 0-cells and 1-cells of $B$. For any compatible choices of 1-cells, the 2-cells $(a_1,\ldots, a_n) \to (b)$ in $\int\! B$ correspond to 2-cells $p: (\ldots(a_1\otimes a_2) \ldots \otimes a_{n-1}) \otimes a_n \to b$ in $B$. Composition is induced by the composition in $B$, using the unique coherence 2-cells as in Corollary \ref{cor:assocoherence} to re-bracket as needed.

Given a functor $f: B \to C$ of bicategories, we define a morphism $\int\! f: \int\! B \to \int\! C$ of regular multi-bicategories, which is the same as $f$ on 0-cells and 1-cells, and maps a 2-cell $(a_1, \ldots, a_n) \to (a)$, corresponding to $p: (\ldots(a_1\otimes a_2) \ldots )\otimes a_n \to b$, to the 2-cell obtained by precomposing $f(p)$ with the unique coherence 2-cell 
\begin{equation*}
	f((\ldots(a_1\otimes a_2) \ldots )\otimes a_n) \to (\ldots(f(a_1)\otimes f(a_2))\ldots )\otimes f(a_n)
\end{equation*}
given by Theorem \ref{thm:functocoherence}. The theorem also guarantees that $\int\! f$ commutes with compositions in $\int\! B$ and $\int\! C$.
\end{cons}

\begin{prop}
Let $B$ be a bicategory. Then $\int\! B$ is a tensor representable regular multi-bicategory. Moreover, if $f: B \to C$ is a functor, $\int\! f: \int\! B \to \int\! C$ is a tensor strong morphism.
\end{prop}
\begin{proof}
This is \cite[Proposition 9.4]{hermida2000representable}, combined with the fact that the unitors of $B$ induce a coherent family of witnesses of unitality in $\int\! B$, and Corollary \ref{cor:isomorphism_cor} for the fact that preservation of tensor units implies preservation of tensor universal cells for $\int\! f$.
\end{proof}
We can then adapt \cite[Theorem 9.8]{hermida2000representable} to obtain the following.
\begin{cor} \label{cor:equiv_bicat_rep}
Construction \ref{cons:groth_monoidal} defines a functor $\int: \bicat \to \mbcat_\otimes$, restricting to a functor $\int: \moncat \to \mcat_\otimes$, inverse to $G$ up to natural isomorphism.
\end{cor}

\begin{remark}
Combined with the bicategorical version of \cite[Theorem 9.8]{hermida2000representable}, this result implies that a unit constructed in a non-regular poly-bicategory from universal cells with degenerate boundaries is also a tensor unit in our sense, so the restriction of a non-regular poly-bicategory with units to its regular 2-cells is tensor 0-representable.
\end{remark}

\begin{remark} 
We do not wish at the moment to extend the equivalence to the level of natural transformations; this is because there does not seem to be a good theory of transformations between morphisms of poly-bicategories that are not equal on 0-cells (something that is not problematic when limited to multicategories and monoidal categories, as in Hermida's treatment). Instead, we will develop it for \emph{merge-bicategories} in the next section.
\end{remark}

\begin{remark}
The connection between tensor representable multi-bicategories and bicategories becomes weaker in the regular case when considering \emph{lax} functors instead of strong ones. In particular, while arbitrary morphisms of non-regular tensor representable multi-bicategories automatically induce lax functors, in the regular case units are not even preserved in the lax sense. It is plausible that the definitions of \cite[Section 6]{kock2008elementary} can be adapted to the poly-bicategorical context, but we will only treat proper functors here.
\end{remark}

An analogue of Hermida's theory for closed categories \cite{eilenberg1966closed} and closed (non-regular) multicategories was developed by Manzyuk in \cite{manzyuk2012closed}. It is straightforward to adapt our arguments relating to units to that setting, and prove a version of \cite[Proposition 4.3]{manzyuk2012closed} for regular multicategories; we leave it as an exercise to fill in the details.
\begin{prop}
There are equivalences between the following pairs of large categories:
\begin{enumerate}[label=(\alph*)]
	\item right closed (left closed) regular multicategories with right closed (left closed) morphisms, and right closed (left closed) categories with strong right closed (left closed) functors;
	\item tensor representable, right and left closed regular multicategories with tensor strong, right and left closed morphisms, and closed monoidal categories with strong closed monoidal functors.
\end{enumerate}
\end{prop}
Although it is easy to define, we do not know of any applications of the variant of closed categories with ``many 0-cells'', corresponding to closed multi-bicategories.

Poly-bicategories were defined by Cockett, Koslowski, and Seely in order to have a coherence-via-universality approach to \emph{linear bicategories} \cite{cockett2000introduction}, which in turn they had defined as the variant with many 0-cells of linearly distributive categories \cite{cockett1997weakly}. While they developed a sophisticated theory of linear functors \cite{cockett1999linearly}, that preserve linear adjoints while being as lax as possible, we will only consider their ``strongest possible'' (but not strict in any sense) variant, where many structural cells and coherence equations are identified. This is because we are interested in linear bicategories only as a stepping stone towards an alternative approach to bicategories and functors.
\begin{dfn}
A \emph{linear bicategory} $B$ is a 2-graph together with two structures of bicategory, sharing the same vertical composites and units, but with two different families of horizontal composites and structural 2-cells:
\begin{enumerate}
	\item the \emph{tensors} $\{a \otimes b: x \to z\}$ of 1-cells and $\{p \otimes q: a \otimes c \to b \otimes d\}$ of 2-cells, with their \emph{tensor units} $\{1_x: x \to x\}$, associators $\{\alpha^\otimes_{a,b,c}: (a \otimes b) \otimes c \to a \otimes (b \otimes c)\}$, and unitors $\{\lambda^\otimes_a: 1_x \otimes a \to a\}$, $\{\rho^\otimes_a: a \otimes 1_y \to a\}$, and
	\item the \emph{pars} $\{a \parr b: x \to z\}$ of 1-cells and $\{p \parr q: a \parr c \to b \parr d\}$ of 2-cells, with their \emph{par units} $\{\bot_x: x \to x\}$, associators $\{\alpha^\parr_{a,b,c}: (a \parr b) \parr c \to a \parr (b \parr c)\}$, and unitors $\{\lambda^\parr_a: 1_x \parr a \to a\}$, $\{\rho^\parr_a: a \parr 1_y \to a\}$,
\end{enumerate}
together with two additional families of 2-cells, 
\begin{enumerate}[resume]
	\item the \emph{left distributors} $\{\delta^L_{a,b,c}: a \otimes (b \parr c) \to (a \otimes b) \parr c\}$, and the \emph{right distributors} $\{\delta^R_{a,b,c}: (a \parr b) \otimes c \to a \parr (b \otimes c)\}$, both indexed by composable triples of 1-cells $a: x \to y$, $b: y \to z$, and $c: z \to w$.
\end{enumerate}

In addition to the conditions that make $B$ a bicategory in two separate ways, there are many conditions relating the distributors to the rest of the structure, that we omit here; see \cite[Section 2.1]{cockett1997weakly} for all the details. A linear bicategory with a single 0-cell is called a (non-symmetric) \emph{linearly distributive category}.

Given two linear bicategories $B, C$, a \emph{strong linear functor} $f: B \to C$ is a morphism of the underlying 2-graphs that is a functor for both bicategory structures separately, with separate families of structural isomorphisms $\{f^\otimes_{a,b}: f(a) \otimes f(b) \to f(a \otimes b)\}$, $\{f^\otimes_x: 1_{f(x)} \to f(1_x)\}$ and $\{f^\parr_{a,b}: f(a) \parr f(b) \to f(a \parr b)\}$, $\{f^\parr_x: \bot_{f(x)} \to f(\bot_x)\}$, which furthermore make the diagrams
\begin{equation} \label{eq:distr_commute}
\begin{tikzpicture}[baseline={([yshift=-.5ex]current bounding box.center)}]
	\node[scale=1.25] (0) at (-5,.75) {$f(a) \otimes f(b \parr c)$};
	\node[scale=1.25] (0b) at (0,.75) {$f(a) \otimes (f(b) \parr f(c))$};
	\node[scale=1.25] (1) at (5,.75) {$(f(a) \otimes f(b)) \parr f(c)$};
	\node[scale=1.25] (2) at (-5,-.75) {$f(a \otimes (b\parr c))$};
	\node[scale=1.25] (2b) at (0,-.75) {$f((a \otimes b) \parr c)$};
	\node[scale=1.25] (3) at (5,-.75) {$f(a \otimes b) \parr f(c)$};
	\draw[1c] (0.east) to node[auto] {$\idd{f(a)} \otimes \invrs{(f^\parr_{b,c})}$} (0b.west);
	\draw[1c] (0b.east) to node[auto] {$\delta^L_{f(a),f(b),f(c)}$} (1.west);
	\draw[1c] (2.east) to node[auto,swap] {$f(\delta^L_{a,b,c})$} (2b.west);
	\draw[1c] (2b.east) to node[auto,swap] {$\invrs{(f^\parr_{a\otimes b, c})}$} (3.west);
	\draw[1c] (0.south) to node[auto,swap] {$f^\otimes_{a,b\parr c}$} (2.north);
	\draw[1c] (1.south) to node[auto] {$f^\otimes_{a,b} \parr \idd{f(c)}$} (3.north);
	\node[scale=1.25] at (6.5,-.9) {,};
\end{tikzpicture}
\end{equation}
\begin{equation} \label{eq:distr_commute2}
\begin{tikzpicture}[baseline={([yshift=-.5ex]current bounding box.center)}]
	\node[scale=1.25] (0) at (-5,.75) {$f(a \parr b) \otimes f(c)$};
	\node[scale=1.25] (0b) at (0,.75) {$(f(a) \parr f(b)) \otimes f(c)$};
	\node[scale=1.25] (1) at (5,.75) {$f(a) \parr (f(b) \otimes f(c))$};
	\node[scale=1.25] (2) at (-5,-.75) {$f((a \parr b) \otimes c)$};
	\node[scale=1.25] (2b) at (0,-.75) {$f(a \parr (b \otimes c))$};
	\node[scale=1.25] (3) at (5,-.75) {$f(a) \parr f(b\otimes c)$};
	\draw[1c] (0.east) to node[auto] {$\invrs{(f^\parr_{a,b})} \otimes \idd{f(c)}$} (0b.west);
	\draw[1c] (0b.east) to node[auto] {$\delta^R_{f(a),f(b),f(c)}$} (1.west);
	\draw[1c] (2.east) to node[auto,swap] {$f(\delta^R_{a,b,c})$} (2b.west);
	\draw[1c] (2b.east) to node[auto,swap] {$\invrs{(f^\parr_{a, b \otimes c})}$} (3.west);
	\draw[1c] (0.south) to node[auto,swap] {$f^\otimes_{a \parr b, c}$} (2.north);
	\draw[1c] (1.south) to node[auto] {$\idd{f(a)} \parr f^\otimes_{b,c}$} (3.north);
	\node[scale=1.25] at (6.5,-.9) {,};
\end{tikzpicture}
\end{equation}
commute.

Linear bicategories and strong linear functors form a large category $\linbicat$. We write $\lindist$ for its full subcategory on linearly distributive categories.
\end{dfn}

\begin{cons} \label{cons:cvu-linbicat}
Let $X$ be a representable regular poly-bicategory; we endow $GX$ with families of cells as required by the definition of a linear bicategory, as follows. For the two bicategory structures, first use tensor representability, as in Construction \ref{cons:cvu-bicat}, then use par representability, applying the same construction to $\coo{X}$. Finally, for each composable triple of 1-cells $a, b, c$, let $\delta^L_{a,b,c}$ be the unique 2-cell obtained by the factorisation 
\begin{equation*} 
\input{img/s4_distributor_left.tex}
\end{equation*}
and $\delta^R_{a,b,c}$ the unique 2-cell obtained by the factorisation
\begin{equation*} 
\input{img/s4_distributor_right.tex}
\end{equation*}
If $Y$ is another representable regular poly-bicategory, and $f: X \to Y$ a strong morphism, we give $Gf$ the structure required by the definition of a strong linear functor, using the fact that $f$ is tensor strong as in Construction \ref{cons:cvu-bicat}, then the fact that $f$ is par strong, applying the same construction to $\coo{f}$.
\end{cons}

Once the general structure of these coherence-via-universality proofs has been understood, it should be a straightforward exercise to prove the following variation on Proposition \ref{prop:cvu-bicat}, which, for what concerns associators and distributors, overlaps with \cite[Theorem 2.1]{cockett1997weakly}.
\begin{prop} 
Let $X$ be a representable regular poly-bicategory. Then $GX$ with the structure defined in Construction \ref{cons:cvu-linbicat} is a linear bicategory. If $f: X \to Y$ is a strong morphism of representable regular poly-bicategories, $Gf: GX \to GY$ with the structure defined in Construction \ref{cons:cvu-linbicat} is a strong linear functor of linear bicategories.
\end{prop}

The following construction is also described in the proof of \cite[Theorem 2.1]{cockett1997weakly}.
\begin{cons} \label{cons:groth_linear}
Let $B$ be a linear bicategory; we define a regular poly-bicategory $\int\! B$ as follows. The 0-cells and 1-cells of $\int\! B$ are the same as the 0-cells and 1-cells of $B$. For any compatible choices of 1-cells, the 2-cells $p: (a_1,\ldots,a_n) \to (b_1,\ldots,b_m)$ in $\int\! B$ correspond to 2-cells 
\begin{equation*}
	p: (\ldots(a_1\otimes a_2) \ldots \otimes a_{n-1}) \otimes a_n \to (\ldots(b_1\parr b_2) \ldots \parr b_{m-1}) \parr b_m
\end{equation*}
in $B$. Composition is induced by the composition in $B$, using associators and distributors to re-bracket as needed; for example, given 2-cells $p: (a) \to (b,c)$ and $q: (c,d) \to (e)$, corresponding to 2-cells $p: a \to b \parr c$ and $q: c \otimes d \to e$ of $B$, the composite $\cutt{2,1}(p,q): (a,d) \to (b,e)$ in $\int\! B$ corresponds to the 2-cell
\begin{equation*}
\begin{tikzpicture}[baseline={([yshift=-.5ex]current bounding box.center)}]
	\node[scale=1.25] (0) at (-4.5,0) {$a \otimes d$};
	\node[scale=1.25] (1) at (-1.5,0) {$(b \parr c) \otimes d$};
	\node[scale=1.25] (2) at (1.5,0) {$b \parr (c \otimes d)$};
	\node[scale=1.25] (3) at (4.5,0) {$b \parr e$};
	\draw[1c] (0.east) to node[auto] {$p \otimes \idd{d}$} (1.west);
	\draw[1c] (1.east) to node[auto] {$\delta^R_{b,c,d}$} (2.west);
	\draw[1c] (2.east) to node[auto] {$\idd{b} \parr q$} (3.west);
\end{tikzpicture}
\end{equation*}
of $B$. 

Given a strong linear functor $f: B \to C$ of linear bicategories, we define a morphism $\int\! f: \int\! B \to \int\! C$, which maps a 2-cell $(a_1,\ldots,a_n) \to (b_1,\ldots, b_m)$, corresponding to a 2-cell $p$ of $B$, to the 2-cell $(f(a_1),\ldots,f(a_n)) \to (f(b_1),\ldots,f(b_m))$, obtained by precomposing $f(p)$ with the appropriate unique coherence cell for tensors, and postcomposing it with the appropriate unique coherence cell for pars. The coherence theorem for bicategories, applied to the two structures separately, together with the commutativity of diagrams (\ref{eq:distr_commute}) and (\ref{eq:distr_commute2}) ensure that composition is preserved.
\end{cons}

The following is also an easy consequence of \cite[Theorem 2.1]{cockett1997weakly} together with our treatment of units in the case of bicategories, now to be applied twice, for tensor and par units.
\begin{prop}
Let $B$ be a linear bicategory. Then $\int\! B$ is a representable regular poly-bicategory. Moreover, if $f: B \to C$ is a strong linear functor, $\int\! f: \int\! B \to \int\! C$ is a strong morphism.
\end{prop}
\begin{cor} \label{cor:equiv_linbicat_rep}
A choice of structure as in Construction \ref{cons:cvu-linbicat} for each representable regular poly-bicategory $X$ and strong morphism $f: X \to Y$ determines a functor $G: \pbcat_\otimes^\parr \to \linbicat$, restricting to a functor $G: \pcat_\otimes^\parr \to \lindist$.

Construction \ref{cons:groth_linear} defines a functor $\int: \linbicat \to \pbcat_\otimes^\parr$, restricting to a functor $\int: \lindist \to \pcat_\otimes^\parr$, inverse to $G$ up to natural isomorphism.
\end{cor}

Before moving on, we briefly treat the theory of linear adjunctions, as defined in \cite[Definition 3.1]{cockett2000introduction}.
\begin{dfn}
Let $B$ be a linear bicategory, and $a: x \to y$, $b: y \to x$ two 1-cells in $B$. A \emph{linear adjunction} $(\varepsilon, \eta): a \dashv b$ is a pair of 2-cells $\varepsilon: b \otimes a \to 1_y$ (the \emph{counit}) and $\eta: 1_x \to a \parr b$ (the \emph{unit}), such that the diagrams
\begin{equation*}
\begin{tikzpicture}[baseline={([yshift=-.5ex]current bounding box.center)}]
	\node[scale=1.25] (0) at (-3.5,.75) {$1_x \otimes a$};
	\node[scale=1.25] (1) at (3.5,.75) {$a$};
	\node[scale=1.25] (2) at (-3.5,-.75) {$(a \parr b) \otimes a$};
	\node[scale=1.25] (2b) at (0,-.75) {$a \parr (b \otimes a)$};
	\node[scale=1.25] (3) at (3.5,-.75) {$a \parr \bot_y$};
	\draw[1c] (0.east) to node[auto] {$\lambda^\otimes_a$} (1.west);
	\draw[1c] (2.east) to node[auto,swap] {$\delta^R_{a,b,a}$} (2b.west);
	\draw[1c] (2b.east) to node[auto,swap] {$a \parr \varepsilon$} (3.west);
	\draw[1c] (0.south) to node[auto,swap] {$\eta \otimes a$} (2.north);
	\draw[1c] (1.south) to node[auto] {$\invrs{(\rho^\parr_a)}$} (3.north);
	\node[scale=1.25] at (4.5,-.9) {,};
\end{tikzpicture}
\end{equation*}
\begin{equation*}
\begin{tikzpicture}[baseline={([yshift=-.5ex]current bounding box.center)}]
	\node[scale=1.25] (0) at (-3.5,.75) {$b \otimes 1_x$};
	\node[scale=1.25] (1) at (3.5,.75) {$b$};
	\node[scale=1.25] (2) at (-3.5,-.75) {$b \otimes (a \parr b)$};
	\node[scale=1.25] (2b) at (0,-.75) {$(b \otimes a) \parr b$};
	\node[scale=1.25] (3) at (3.5,-.75) {$\bot_y \parr b$};
	\draw[1c] (0.east) to node[auto] {$\rho^\otimes_a$} (1.west);
	\draw[1c] (2.east) to node[auto,swap] {$\delta^L_{b,a,b}$} (2b.west);
	\draw[1c] (2b.east) to node[auto,swap] {$\varepsilon \parr a$} (3.west);
	\draw[1c] (0.south) to node[auto,swap] {$b \otimes \eta$} (2.north);
	\draw[1c] (1.south) to node[auto] {$\invrs{(\lambda^\parr_b)}$} (3.north);
\end{tikzpicture}
\end{equation*}
commute in $B$. In a linear adjunction, $a$ is called the left linear adjoint of $b$, and $b$ the right linear adjoint of $a$.
\end{dfn}
The equational characterisation of linear adjunctions in a poly-bicategory of \cite[Equation 4]{cockett2003morphisms} relies on degenerate boundaries, so it cannot be used as is. Nevertheless, an equivalent characterisation based on homs and cohoms still works, relativised to a choice of weak units.
\begin{prop} \label{prop:adjoints}
Let $X$ be a unital regular poly-bicategory, $a: x \to y$, $b: y \to x$ be 1-cells, and $1_x: x \to x$, $\bot_y: y \to y$ a tensor and a par unit in $X$, respectively. The following are all equivalent:
\begin{enumerate}
	\item there exists a 2-cell $\varepsilon: (b, a) \to (\bot_y)$ universal at $\bord{1}{-}$;
	\item there exists a 2-cell $\varepsilon: (b, a) \to (\bot_y)$ universal at $\bord{2}{-}$;
	\item there exists a 2-cell $\eta: (1_x) \to (a,b)$ universal at $\bord{1}{+}$;
	\item there exists a 2-cell $\eta: (1_x) \to (a,b)$ universal at $\bord{2}{+}$.
\end{enumerate}
Moreover, if $X$ is representable, each condition is equivalent to the existence of a linear adjunction $a \dashv b$ in $GX$.
\end{prop}
\begin{proof}
Fix coherent witnesses of tensor and par unitality as in Theorem \ref{thm:polycoherent} and its dual. Then the proof of \cite[Proposition 1.7]{cockett2003morphisms}, based on a characterisation of ordinary adjunctions by Street and Walters \cite[Proposition 2]{street1978yoneda}, goes through, relative to this choice.
\end{proof}

If any of the conditions of Proposition \ref{prop:adjoints} holds, we call $a$ a left linear adjoint of $b$ and $b$ a right linear adjoint of $a$ in $X$.

If $X$ is left and right closed (respectively, coclosed), and has both tensor and par units, then every 1-cell has both a left and a right linear adjoint. A choice of adjoints induces an equivalence between $X$ and $\coo{(\opp{X})}$; in particular, $X$ is automatically left and right coclosed (respectively, closed), and it is tensor representable if and only if it is par representable. Conversely, a representable regular poly-bicategory where every 1-cell has both a left and a right linear adjoint is closed and coclosed on both sides: this follows immediately from the analogous statement for linear bicategories, pulled through the equivalence.

As discussed in \cite[Section 3]{cockett2000introduction}, a linearly distributive category where every 1-cell has a left and a right linear adjoint is the same as a nonsymmetric $*$-autonomous category, in the sense of \cite{barr1995nonsymmetric}. Therefore, we immediately obtain the following. Let $\staraut$ be the full subcategory of $\lindist$ on nonsymmetric $*$-autonomous categories.
\begin{prop}
The equivalence between $\pcat_\otimes^\parr$ and $\lindist$ restricts to an equivalence between $\pcat_*$ and $\staraut$.
\end{prop}

We conclude this section by re-evaluating one of the main examples of \cite{cockett2003morphisms}.

\begin{exm} \label{exm:chu}
In \cite[Example 1.8(2)]{cockett2003morphisms}, the following generalisation of the Chu construction \cite[Appendix]{barr1979autonomous} was outlined.

Let $M$ be a (regular or non-regular) multi-bicategory; we construct a poly-bicategory $\chu{M}$ as follows. The 0-cells of $\chu{M}$ are endo-1-cells $a: x \to x$, and a 1-cell $A: (a: x\to x) \to (b: y\to y)$ of $\chu{M}$ is given by a pair of 1-cells $A: x \to y$, $A^\bot: y \to x$, and a pair of 2-cells $e_A: (A,A^\bot) \to (a)$, $e_{A^\bot}: (A^\bot,A) \to (b)$ of $M$. There is an evident involution $A \mapsto A^\bot$ on the 1-cells of $\chu{M}$, reversing the roles of the 1-cells and 2-cells of $M$.

For $n > 1$, given a sequence $(A_1, \ldots, A_n)$ of 1-cells with $\bord{}{+}A_i = \bord{}{-}A_{i+1}$, for $i = 1,\ldots, n-1$, and $\bord{}{+}A_n = \bord{}{-}A_1$, a \emph{Chu band} $p$ of type $(A_1,\ldots,A_n)$ is defined as an $n$-tuple of 2-cells
\begin{align*}
	p_1: & \; (A_2, \ldots, A_n) \to (A^\bot_1), \\
	p_2: & \; (A_3, \ldots, A_n, A_1) \to (A^\bot_2), \\
	& \quad \quad \quad \quad \vdots \\
	p_n: & \; (A_1, \ldots, A_{n-1}) \to (A^\bot_n),
\end{align*} 
in $M$, satisfying 
\begin{align*} 
	\cutt{1,1}(p_1,e_{A^\bot_1}) & \;=\; \cutt{1,2}(p_2,e_{A_2}), \\
	\cutt{1,1}(p_2,e_{A^\bot_2}) & \;=\; \cutt{1,2}(p_3,e_{A_3}), \\
	& \;\;\vdots \\
	\cutt{1,1}(p_n,e_{A^\bot_n}) & \;=\; \cutt{1,2}(p_1,e_{A_1}).
\end{align*}
A 2-cell $(A_1,\ldots,A_n) \to (B_1, \ldots, B_m)$, with $n + m > 1$, is defined to be a Chu band of type $(A_1,\ldots,$ $A_n,B^\bot_m,\ldots,B^\bot_1)$. Notice that this only defines 2-cells with nullary input $() \to (B_1,\ldots,B_m)$ or 2-cells with nullary output $(A_1,\ldots,A_n) \to ()$ when $n > 1$ and $m > 1$.

Given two composable 2-cells, there is only one way of putting their components together that is consistent with the typing; for example, given 2-cells $p: (A) \to (B,C)$ and $q: (C,D) \to (E)$, corresponding to Chu bands $p$ of type $(A, C^\bot, B^\bot)$ and $q$ of type $(C, D, E^\bot)$, their composite $\cutt{2,1}(p,q)$ has to be a Chu band $r$ of type $(A,D,E^\bot, B^\bot)$; the typing forces us to define
\begin{align*}
	r_1 & \; := \; \cutt{1,1}(q_1,p_1): (D,E^\bot,B^\bot) \to (A^\bot), \\
	r_2 & \; := \; \cutt{1,2}(p_2,q_2): (E^\bot, B^\bot, A) \to (D^\bot), \\
	r_3 & \; := \; \cutt{1,1}(p_2,q_3): (B^\bot, A, D) \to (E), \\
	r_4 & \; := \; \cutt{1,2}(q_1,p_3): (A, D, E^\bot) \to (B).
\end{align*}
It is an exercise to check that this is, indeed, a Chu band.

If $M$ is non-regular, and $A$ is an endo-1-cell $A: (a: x\to x) \to (a: x\to x)$, it is possible to define a unary Chu band of type $(A)$ as a single 2-cell $p: () \to (A)$ of $M$, satisfying $\cutt{1,1}(p, e_{A^\bot}) = \cutt{1,2}(p, e_{A})$. On the other hand, there is no sensible notion of nullary Chu band, nor a way of composing a Chu band of type $(A)$ and one of type $(A^\bot)$, which would be necessary in order to compose a 2-cell $() \to (A)$ and a 2-cell $(A) \to ()$.

Thus, contradicting \cite[Example 2.5(4)]{cockett2003morphisms}, we claim that $\chu{M}$ should be defined as a regular poly-bicategory or, at least, a poly-bicategory with no 2-cells with zero inputs and one output, nor 2-cells with one input and zero outputs. 

In particular, the construction of units from universal 2-cells with nullary input or output does not apply, as it requires 2-cells $() \to (1_x)$ and $(\bot_x) \to ()$, unless done in an entirely circular way (\emph{defining} the ``missing'' 2-cells to be those of the linear bicategory one wants to induce, strictified in order to recover an associative composition). That is, $\chu{M}$ is never ``representable for tensor units'' in the sense of \cite[Section 2]{cockett2003morphisms}.

Nevertheless, it is true that, if $M$ is tensor 0-representable, $\chu{M}$ is both tensor and par 0-representable. Fix a tensor unit $1_x: x \to x$ on each 0-cell $x$ of $M$, and a coherent family of witnesses of unitality $\{l_a, r_a\}$ as in Theorem \ref{thm:polycoherent}. For each 0-cell $a: x \to x$ of $\chu{M}$, we claim that the 1-cell $1_a$ of $\chu{M}$ given by the pair of 1-cells $1_x: x \to x$, $a: x \to x$, together with the pair of 2-cells $l_a: (1_x,a) \to (a)$ and $r_a: (a,1_x) \to (a)$, is a tensor unit on $a$ in $\chu{M}$. 

To prove this, let $A: (a: x \to x) \to (b: y \to y)$ be a 1-cell in $\chu{M}$. The triple of 2-cells of $M$
\begin{align*}
	l_{A,1} & \; := e_A : (A, A^\bot) \to (a), \\
	l_{A,2} & \; := r_{A^\bot} : (A^\bot, 1_x) \to (A^\bot), \\
	l_{A,3} & \; := l_A : (1_x, A) \to (A),
\end{align*} 
defines a Chu band of type $(1_a, A, A^\bot)$, that is a 2-cell $l_A: (1_a, A) \to (A)$ of $\chu{M}$: the three Chu band conditions are instances of equations (\ref{eq:natural2}), (\ref{eq:trianglepoly}), and (\ref{eq:natural}), respectively.

It is then a straightforward exercise to show that this 2-cell is universal at $\bord{1}{+}$ and at $\bord{2}{-}$. For example, given a 2-cell $p: (1_a, \Gamma) \to (A, \Delta)$ in $\chu{M}$, we can obtain a unique factorisation through a 2-cell $\tilde{p}: (\Gamma) \to (A,\Delta)$, whose components of the form $(\Gamma_1,\Gamma_2) \to (B)$ are the unique factorisations through a witness of unitality in $M$ of the components $(\Gamma_1,1_x,\Gamma_2) \to (B)$ of $p$.

Similarly, we define a 2-cell $r_A: (A, 1_b) \to (A)$ of $\chu{M}$, with components
\begin{align*}
	r_{A,1} & \; := l_{A^\bot} : (1_y, A^\bot) \to (A^\bot), \\
	r_{A,2} & \; := e_{A^\bot} : (A^\bot, A) \to (b), \\
	r_{A,3} & \; := r_{A} : (A, 1_y) \to (A),
\end{align*} 
and prove that it is universal at $\bord{1}{+}$ and at $\bord{1}{-}$. This proves that the $1_a$ are tensor units in $\chu{M}$. 

The fact that $1^\bot_a$ is a par unit in $\chu{M}$ can either be checked directly, or derived, by duality, from the fact that the morphism $(-)^\bot: \chu{M} \to \coo{\opp{\chu{M}}}$, defined as the identity on 0-cells, as $A \mapsto A^\bot$ on 1-cells, and mapping a Chu band $b$, seen as a 2-cell $(A_1,\ldots,A_n) \to (B_1,\ldots, B_m)$, to the same Chu band seen as a 2-cell $(B^\bot_m,\ldots,B^\bot_1) \to (A^\bot_n, \ldots, A^\bot_1)$, is an involutive isomorphism of regular poly-bicategories. 

It is equally straightforward to check that the 2-cell $\varepsilon_A: (A, A^\bot) \to (1_a^\bot)$ with components $\varepsilon_{A,1} := r_{A^\bot}$, $\varepsilon_{A,2} := l_A$, and $\varepsilon_{A,3} := e_A$ is universal at $\bord{1}{-}$ and $\bord{2}{-}$, which makes it a witness of $A^\bot$ as a left linear adjoint of $A$; the 2-cell $\varepsilon_{A^\bot}: (A^\bot, A) \to (1_b^\bot)$ defined similarly shows that it is also a right linear adjoint.

The conditions for $\chu{M}$ to be tensor and par 1-representable do not differ from those stated in \cite[Example 2.5(4)]{cockett2003morphisms}.
\end{exm}

\begin{remark}
A higher-dimensional version of the Chu construction was proposed by Shulman in \cite{shulman2017mvar}, based on the example of a ``2-multicategory'' of multivariable adjunctions in \cite{cheng2014cyclic}. The same remarks about units should apply, relative to a suitable notion of universality up to isomorphism.
\end{remark}

\section{Merge-bicategories and higher morphisms} \label{sec:probicat}

Let $B$ be a bicategory. Then $B$ can also be seen as a degenerate linear bicategory, whose two bicategory structures coincide, and whose right and left distributors are associators and inverses of associators, respectively. Clearly, a functor of bicategories becomes a strong linear functor of the corresponding linear bicategories, so this determines an inclusion $\imath: \bicat \to \linbicat$ exhibiting $\bicat$ as a full subcategory of $\linbicat$.

Composing this with $\int: \linbicat \to \pbcat_\otimes^\parr$, we can realise every bicategory as a representable regular poly-bicategory. Indeed, there is a certain asymmetry in Construction \ref{cons:groth_monoidal}, where we choose to realise a bicategory as a tensor representable multi-bicategory, rather than its par representable dual: this makes sense at the level of morphisms, if we are interested in lax rather than colax functors; not so much when limiting ourselves to proper functors.

On the other hand, the characterisation of bicategories as degenerate objects of $\pbcat_\otimes^\parr$ is quite unsatisfactory, as it amounts to requiring that for \emph{some} choice of tensors, pars, and their units, the representing 1-cells for tensors and pars coincide, and the tensor and par units coincide. This kind of post-selection, requiring certain equalities or isomorphisms between specific cells, goes against the spirit of coherence-via-universality, where we only care about the properties of cells, and not their names.

Rather, we would like to make the choice of pars and par units a consequence of the choice of tensors and tensor units, the way it is (trivially) in the passage from bicategories to degenerate linear bicategories. This would be possible if a tensor $t: (a, b) \to (a \otimes b)$ had a unique ``inverse'' $\invrs{t}: (a \otimes b) \to (a,b)$, but there is no way to make sense of this in a poly-bicategory, for lack of compositions along multiple 1-cells.

On the other hand, it \emph{is} possible to compose 2-cells in a bicategory ``along composites of multiple 1-cells''. This is implicit in the string-diagrammatic calculus for bicategories \cite{joyal1991geometry}, when string diagrams are seen as Poincar\'e duals of pasting diagrams, since string diagrams can usually be composed along multiple edges.

In this section, we make this a structural possibility, in a variant of poly-bicategories that we call \emph{merge-bicategories}: while ``cut'' is composition along a single edge, and ``mix'' is parallel (or nullary) composition \cite{cockett1997proof}, we speak of ``merge'' for composition along one or more edges, or, in terms of pasting diagrams, the most general composition that preserves up to homeomorphism the topology of diagrams as combinatory polygons. 

This will allow us to recover bicategories, instead of linear bicategories, as the representable objects, and also to develop a satisfactory theory of natural transformations and modifications.

\begin{dfn}
Let $X$ be a regular 2-polygraph. We write $X_1^+$ for the set of composable non-empty sequences of 1-cells of $X$; that is, the elements of $X_1^+$ are sequences $(a_1,\ldots,a_n)$ of 1-cells of $X$, such that $\bord{}{+}a_i = \bord{}{-}a_{i+1}$, for $i=1,\ldots,n-1$. 

For all $n,m > 0$, and $1 \leq i_1 \leq i_2 \leq n$, $1 \leq j_1 \leq j_2 \leq m$, we also define functions
\begin{align*}
	\bord{[i_1,i_2]}{-}: X_2^{(n,m)} \to X_1^+, \quad \quad & p \mapsto (\bord{i_1}{-}p,\ldots,\bord{i_2}{-}p), \\
	\bord{[j_1,j_2]}{+}: X_2^{(n,m)} \to X_1^+, \quad \quad & p \mapsto (\bord{j_1}{+}p,\ldots,\bord{j_2}{+}p). 
\end{align*}
We write simply $\bord{}{-}$ for $\bord{[1,n]}{-}$, and $\bord{}{+}$ for $\bord{[1,m]}{+}$.
\end{dfn}

\begin{dfn}
A \emph{merge-bicategory} is a regular 2-polygraph $X$ together with ``merge'' functions
\begin{equation*}
\begin{tikzpicture}
	\node[scale=1.25] (0) at (0,0) {$X_2^{(n,m)} \pback{\bord{[j_1,j_2]}{+}}{\bord{[i_1,i_2]}{-}} X_2^{(p,q)}$};
	\node[scale=1.25] (1) at (6,0) {$X_2^{(n+p-\ell,m+q-\ell)},$};
	\draw[1c] (0) to node[auto] {$\mrg{[j_1,j_2]}{[i_1,i_2]}$} (1);
\end{tikzpicture}
\end{equation*}
whenever $1 \leq j_1 \leq j_2 \leq m$ and $1 \leq i_1 \leq i_2 \leq p$, such that $\ell := j_2 - (j_1-1) = i_2 - (i_1 -1)$, satisfy the two conditions on any side of the following square:
\begin{equation*}
\begin{tikzpicture}[baseline={([yshift=-.5ex]current bounding box.center)}]
	\node[scale=1.25] (0) at (-1.5,.75) {$i_1=1$};
	\node[scale=1.25] (1) at (1.5,.75) {$j_2=m$};
	\node[scale=1.25] (2) at (-1.5,-.75) {$i_2=p$};
	\node[scale=1.25] (3) at (1.5,-.75) {$j_1=1.$};
	\draw[edge] (0.east) to node[auto] {$(b)$} (1.west);
	\draw[edge] (2.east) to node[auto,swap] {$(d)$} (3.west);
	\draw[edge] (0.south) to node[auto,swap] {$(a)$} (2.north);
	\draw[edge] (1.south) to node[auto] {$(c)$} (3.north);
\end{tikzpicture}
\end{equation*}
Each pair of conditions corresponds to a diagram in (\ref{eq:composable}), where the shared boundary can now comprise multiple 1-cells. The interaction of merge with the boundaries is also evident from the diagrams: explicitly,
\begin{enumerate}[label=($\alph*$)]
	\item $\bord{k}{-}\mrg{[j_1,j_2]}{[1,p]}(t,s) = \bord{k}{-}t, \\
	\bord{k}{+}\mrg{[j_1,j_2]}{[1,p]}(t,s) = \begin{cases}
		\bord{k}{+}t, & 1 \leq k \leq j_1-1, \\
		\bord{k-j_1+1}{+}s, & j_1 \leq k \leq j_1+q-1, \\
		\bord{k-q+p}{+}t, & j_1+q \leq k \leq m+q-p;
	\end{cases}$
		
	\item $\bord{k}{-}\mrg{[j_1,m]}{[1,i_2]}(t,s) = \begin{cases}
		\bord{k}{-}t, & 1 \leq k \leq n, \\
		\bord{k-n+\ell}{-}s, & n+1 \leq k \leq n+p-\ell,
	\end{cases} \\
	\bord{k}{+}\mrg{[j_1,m]}{[1,i_2]}(t,s) = \begin{cases}
		\bord{k}{+}t, & 1 \leq k \leq m-\ell, \\
		\bord{k-m+\ell}{+}s, & m-\ell+1 \leq k \leq m+q-\ell;
	\end{cases}$
	
	\item $\bord{k}{-}\mrg{[1,m]}{[i_1,i_2]}(t,s) = \begin{cases}
		\bord{k}{-}s, & 1 \leq k \leq i_1-1, \\
		\bord{k-i_1+1}{-}t, & i_1 \leq k \leq i_1+n-1, \\
		\bord{k-n+m}{-}s, & i_1+n \leq k \leq n+p-m,
	\end{cases} \\ \bord{k}{+}\mrg{[1,m]}{[i_1,i_2]} = \bord{k}{+}s;$
	
	\item $\bord{k}{-}\mrg{[1,j_2]}{[i_1,p]}(t,s) = \begin{cases}
		\bord{k}{-}s, & 1 \leq k \leq p-\ell, \\
		\bord{k-p+\ell}{-}t, & p-\ell+1 \leq k \leq n+p-\ell;
	\end{cases}\\ 
	\bord{k}{+}\mrg{[1,j_2]}{[i_1,p]}(t,s) = \begin{cases}
		\bord{k}{+}s, & 1 \leq k \leq q, \\
		\bord{k-q+\ell}{+}t, & q+1 \leq k \leq m+q-\ell.
	\end{cases}$
\end{enumerate}

Moreover, the $\mrg{[j_1,j_2]}{[i_1,i_2]}$ satisfy associativity and interchange equations that guarantee the uniqueness of the merger of three or more 2-cells, whenever they can be merged in different orders. In addition to the 9 associativity schemes of diagram (\ref{asso_scheme}) and the 8 interchange schemes of diagram (\ref{inter_scheme}), which now allow the shared boundaries to comprise more than one 1-cell, there are 16 new associativity schemes:
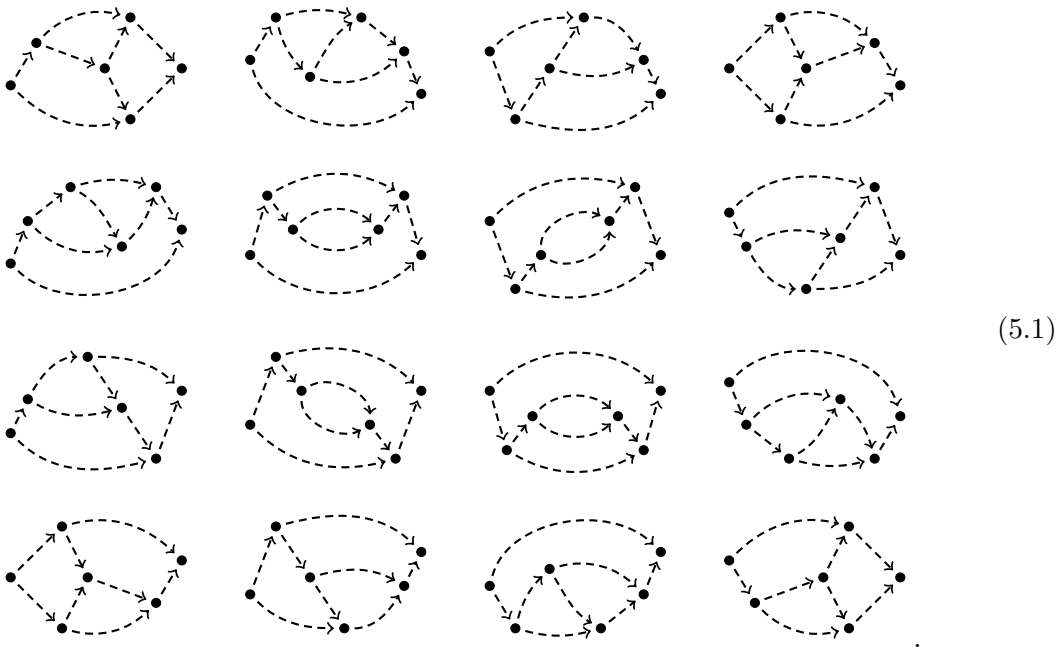
\begin{equation} \label{asso_scheme2}
\input{img/s5_associativity-pro.tex}
\end{equation}

Given two merge-bicategories $X, Y$, a \emph{morphism} $f: X \to Y$ is a morphism of the underlying regular 2-polygraphs that commutes with the $\mrg{[j_1,j_2]}{[i_1,i_2]}$ functions. Merge-bicategories and their morphisms form a large category $\mrgpol$.
\end{dfn}

The merge-composable pasting diagrams of 2-cells in a merge-bicategory are precisely those whose shape is constructible 2-molecule, as defined in \cite{hadzihasanovic2018combinatorial}. The obvious forgetful functor $U: \mrgpol \to \rtwopol$ is monadic: its left adjoint freely adds all the merge-composable pasting diagrams of 2-cells to a regular 2-polygraph. 

The forgetful functor factors through a functor $U_1: \mrgpol \to \pbcat$, which endows the underlying regular 2-polygraph with the operations 
\begin{equation*}
	\cutt{j,i} := \mrg{[j,j]}{[i,i]}.
\end{equation*}
So a merge-bicategory can be seen as a regular poly-bicategory with additional structure.

\begin{remark}
The functor $U_1$ is faithful, but it is neither full nor essentially surjective: not every regular poly-bicategory admits a structure of merge-bicategory, and if it does admit one, it may not be unique (that is, being a merge-bicategory is not a \emph{property} of a regular poly-bicategory).

For example, let $B$ be the regular poly-bicategory corresponding to a Boolean algebra, as in Remark \ref{rmk:booleanalgebra}. Then $B$ has a 2-cell $s: (\top) \to (\top,\bot)$, since $\top = \top \lor \bot$, and a 2-cell $t: (\top,\bot) \to (\bot)$, since $\top \land \bot = \bot$. However, unless $\top = \bot$, there is no 2-cell $(\top) \to (\bot)$ in $B$, hence no way to define $\mrg{[1,2]}{[1,2]}(s,t)$. Therefore $B$ is not $U_1 \tilde{B}$ for any merge-bicategory $\tilde{B}$. This proves that $U_1$ is not essentially surjective.

Next, let $X$ be a regular polycategory with four 1-cells $a, b, c, d$, and four 2-cells $s: (c) \to (a,b)$, $t: (a,b) \to (d)$, $r_1: (c) \to (d)$, and $r_2: (c) \to (d)$; notice that there are no cut-composable pairs of 2-cells. However, there are two possible structures of merge-bicategory on the underlying regular 2-polygraph, $\tilde{X}_1$ with $\mrg{[1,2]}{[1,2]}(s,t) := r_1$ and $\tilde{X}_2$ with $\mrg{[1,2]}{[1,2]}(s,t) := r_2$, such that $U_1\tilde{X}_1 = U_1\tilde{X}_2 = X$. 

Moreover, $X$ has an involutive automorphism which exchanges $r_1$ and $r_2$, but it does not lift to an automorphism of either $\tilde{X}_1$ or $\tilde{X}_2$. This proves that $U_1$ is not full.
\end{remark}

The expanded possibilities for composition open up possibilities for division. In the following, we assume that it is clear from context what a well-formed equation is.
\begin{dfn}
Let $t \in X_2^{(n,m)}$ be a 2-cell in a merge-bicategory $X$. We say that $t$ is \emph{universal at $\bord{[j_1,j_2]}{+}$} if, for all 2-cells $s$ and well-formed equations $\mrg{[j_1,j_2]}{[i_1,i_2]}(t,x) = s$, there exists a unique 2-cell $r$ such that $\mrg{[j_1,j_2]}{[i_1,i_2]}(t,r) = s$.

We say that $t$ is \emph{universal at $\bord{[i_1,i_2]}{-}$} if, for all 2-cells $s$ and well-formed equations $\mrg{[j_1,j_2]}{[i_1,i_2]}(x,t) = s$, there exists a unique 2-cell $r$ such that $\mrg{[j_1,j_2]}{[i_1,i_2]}(r,t) = s$.

We say that $t$ is \emph{left universal} if it is universal at $\bord{}{+}$, and \emph{right universal} if it is universal at $\bord{}{-}$. We say that $t$ is \emph{universal} if it is both left and right universal. 
\end{dfn}

It is clear that any cell of a merge-bicategory $X$ that is universal at $\bord{j}{+}$ or $\bord{i}{-}$ in $U_1X$ is universal at $\bord{[j,j]}{+}$ or $\bord{[i,i]}{-}$ in $X$. Therefore, any notion of universality that we have defined in the context of regular poly-bicategories still makes sense for merge-bicategories, and we will casually employ the same terminology for special universal cells, speaking of tensors, pars, tensor and par units. We will also use ``universal at $\bord{i}{-}$'' (or $\bord{j}{+}$) as an abbreviation of ``universal at $\bord{[i,i]}{-}$'' (or $\bord{[j,j]}{+}$).

As we have universal 2-cells with multiple inputs and outputs, so we have unit 2-cells with multiple inputs and outputs.
\begin{dfn}
Let $\Gamma = a_1,\ldots,a_n$ be a composable sequence of 1-cells in a merge-bicategory. A 2-cell $\idd{\Gamma}: (\Gamma) \to (\Gamma)$ is a \emph{unit} on $\Gamma$ if, for all $t \in X_2^{(n,m)}$, if $\bord{[j_1,j_2]}{+}t = (\Gamma)$ then $\mrg{[j_1,j_2]}{[1,n]}(t,\idd{\Gamma}) = t$, and if $\bord{[i_1,i_2]}{-}t = (\Gamma)$ then $\mrg{[1,n]}{[i_1,i_2]}(\idd{\Gamma},t) = t$.
\end{dfn}

It is easy to adapt the proofs of Lemma \ref{lem:2divunit}, Proposition \ref{prop:2units}, and Proposition \ref{cor:oneuniversal} to obtain the following.
\begin{lem}
Let $p: (a_1,\ldots,a_n) \to (b_1,\ldots,b_m)$ be a universal 2-cell in a merge-bicategory $X$, and let $\idd{}: (a_1,\ldots,a_n) \to (a_1,\ldots,a_n)$, $\idd{}': (b_1,\ldots,b_m) \to (b_1,\ldots,b_m)$ be the unique 2-cells such that 
\begin{equation*}
	\mrg{[1,n]}{[1,n]}(\idd{},p) = p, \quad\quad \mrg{[1,m]}{[1,m]}(p,\idd{}') = p.
\end{equation*}
Then $\idd{}$ is a unit on $(a_1,\ldots,a_n)$, and $\idd{}'$ is a unit on $(b_1,\ldots,b_m)$.
\end{lem}
\begin{prop} \label{prop:2unitsmerge}
Let $X$ be a merge-bicategory. The following are equivalent:
\begin{enumerate}
	\item for all composable sequences of 1-cells $(\Gamma)$ in $X$, there exist a composable sequence $(\overline{\Gamma})$ and a universal 2-cell $p: (\Gamma) \to (\overline{\Gamma})$;
	\item for all composable sequences of 1-cells $(\Gamma)$ in $X$, there exist a composable sequence $\overline{\Gamma}$ and a universal 2-cell $p': (\overline{\Gamma}) \to (\Gamma)$;
	\item for all composable sequences of 1-cells $(\Gamma)$ in $X$, there exists a (necessarily unique) unit $\idd{\Gamma}$ on $\Gamma$.
\end{enumerate}
\end{prop}

\begin{dfn} \label{dfn:mergeunital}
A merge-bicategory $X$ is \emph{unital} if it satisfies any of the equivalent conditions of Proposition \ref{prop:2unitsmerge}.
\end{dfn}

This stronger notion of unitality is, in fact, all that we need to make the two bicategory structures induced by a representable regular poly-bicategory collapse.

\begin{prop} \label{prop:mergeunital}
Let $X$ be a unital merge-bicategory, and $p: (\Gamma) \to (\Delta)$ a 2-cell of $X$. The following are equivalent:
\begin{enumerate}
	\item $p$ is universal at $\bord{}{+}$;
	\item $p$ is universal at $\bord{}{-}$;
	\item $p$ is an isomorphism, that is, it has a unique inverse $\invrs{p}: (\Delta) \to (\Gamma)$ such that $\mrg{[1,m]}{[1,m]}(p,\invrs{p})= \idd{\Gamma}$, and $\mrg{[1,n]}{[1,n]}(\invrs{p},p) = \idd{\Delta}$.
\end{enumerate}
\end{prop}

\begin{cor} \label{cor:tensorparcollapse}
Let $X$ be a unital merge-bicategory, and let $a,b$ be a composable pair of 1-cells. Then each tensor $(a,b) \to (a\otimes b)$ determines a par $(a\otimes b) \to (a,b)$, and vice versa. 

In particular, $U_1 X$ is tensor 1-representable if and only if it is par 1-representable.
\end{cor}
\begin{proof}
A tensor $t: (a,b) \to (a\otimes b)$ is universal at $\bord{}{+}$; by Proposition \ref{prop:mergeunital}, it is an isomorphism, with an inverse $\invrs{t}: (a\otimes b) \to (a,b)$ which is also an isomorphism. In particular, $\invrs{t}$ is universal at $\bord{}{-}$, that is, it is a par.
\end{proof}

We do not quite have an equivalence between tensor 0-representability and par 0-representability, because the inverse of a universal 2-cell exhibiting a right hom may not exhibit a right cohom. On the other hand, when $U_1 X$ is both tensor and par 0-representable, one structure determines the other.

\begin{dfn}
A unital merge-bicategory $X$ is \emph{0-representable} (\emph{1-representable}, \emph{representable}) if the regular poly-bicategory $U_1 X$ is 0-representable (1-representable, representable).
\end{dfn}

\begin{prop}  \label{cor:tensorparunitcollapse}
Let $X$ be a 0-representable merge-bicategory. Then:
\begin{enumerate}[label=($\alph*$)]
	\item a 1-cell $1_x: x \to x$ is a tensor unit if and only if it is a par unit;
	\item a 1-cell $e: x \to y$ is tensor universal if and only if it is par universal.
\end{enumerate}
Moreover, the 2-cells that exhibit $1_x$ as a par unit, or $e$ as par universal, can be chosen as inverses of the 2-cells that exhibit $1_x$ as a tensor unit, or $e$ as tensor universal.
\end{prop}
\begin{proof}
Suppose $1_x: x \to x$ is a tensor unit, and fix witnesses of its tensor unitality $\{l^\otimes_a, r^\otimes_a\}$. Let $\bot_x$ be a par unit on $x$, which exists by par 0-representability of $X$, and also fix witnesses $\{l^\parr_a, r^\parr_a\}$. Then $l^\otimes_{\bot_x}: (1_x,\bot_x) \to (\bot_x)$ and $\invrs{(r^\parr_{1_x})}: (1_x, \bot_x) \to (1_x)$ are both universal at $\bord{1}{+}$, so by factoring the second through the first we obtain an isomorphism $p: \bot_x \to 1_x$. We can postcompose all the $l^\parr_a$ and the $r^\parr_a$ with $p$, to obtain witnesses of the par unitality of $1_x$. But the $\{\invrs{(l^\otimes_a)}, \invrs{(r^\otimes_a)}\}$ are also universal at $\bord{1}{-}$, so we can apply Lemma \ref{lem:twouniversal} to conclude that they are universal at $\bord{2}{+}$ or $\bord{1}{+}$, as needed. The converse implication follows immediately by duality.

Now, let $e: x \to y$ be tensor universal. Since tensors and pars coincide in $X$, and, by the first part, so do tensor and par units, we know that $e$ is ``par invertible'', that is, there exist a tensor universal 1-cell $e^*: y \to x$ and 2-cells exhibiting $e \parr e^* \simeq \bot_x$ and $e^* \parr e \simeq \bot_{y}$. Moreover, because $e$, $e^*$ have tensors with arbitrary 1-cells $a: y \to z$ and $b: x \to z'$, they also have pars with arbitrary 1-cells, so the derivation of par universality from par invertibility as in Proposition \ref{prop:isomorphism} (dualised to $\coo{X}$) goes through even if $X$ is not 1-representable. The fact that the 2-cells exhibiting par universality can be picked as the inverses of the 2-cells exhibiting tensor universality is another immediate consequence of Lemma \ref{lem:twouniversal}.
\end{proof}

\begin{remark} \label{remark:merge-coherent}
In fact, we can pick \emph{coherent} witnesses of tensor and par unitality which are mutual inverses: starting from arbitrary families that are mutual inverses, the construction in the proof of Theorem \ref{thm:polycoherent} and its dual produce coherent families that are still mutual inverses.

Furthermore, even when $X$ is not 0-representable, if $1_x$ is both a tensor unit and a par unit, or if $e$ is both tensor and par universal, witnesses of par unitality or universality can be chosen as inverses of witnesses of tensor unitality or universality.
\end{remark}

\begin{dfn}
Let $X$ be a merge-bicategory. A 1-cell $1_x: x \to x$ in $X$ is a \emph{1-unit} if it is both a tensor unit and a par unit. A 1-cell $e: x \to y$ is \emph{universal} if it is both tensor and par universal.
\end{dfn}

Corollary \ref{cor:tensorparcollapse} allows us to simplify the definition of 1-representability for merge-bicategories, by putting unitality and 1-representability on the same footing; then, Proposition \ref{cor:tensorparunitcollapse} allows us to further simplify the definition of representability, emphasising its symmetry.

\begin{prop} \label{def:representable}
A merge-bicategory $X$ is representable if and only if
\begin{enumerate}
	\item for all 0-cells $x$ in $X$, there exist a 0-cell $\overline{x}$ and a universal 1-cell $e: x \to \overline{x}$, or a universal 1-cell $e': \overline{x} \to x$;
	\item for all 1-cells $a$ in $X$, there exist a 1-cell $\overline{a}$ and a universal 2-cell $p: (a) \to (\overline{a})$, or a universal 2-cell $p': (\overline{a}) \to (a)$;
	\item for all composable pairs $a,b$ in $X$, there exist a 1-cell $a \otimes b$ and a universal 2-cell $t: (a,b) \to (a \otimes b)$, or a universal 2-cell $t': (a\otimes b) \to (a,b)$.
\end{enumerate}
\end{prop}
Universal 2-cells with arbitrarily long sequences as inputs or outputs, as required by Definition \ref{dfn:mergeunital}, are obtained by composing the binary ``tensors'' or ``pars''.

Similarly, we can simplify the definition of a strong morphism.
\begin{prop} \label{def:representablestrong}
A morphism $f: X \to Y$ of representable merge-bicategories is strong if and only if it preserves universal 1-cells and 2-cells.
\end{prop}
\begin{dfn}
We write $\mrgpol_\otimes$ for the large category of representable merge-bicategories and strong morphisms.
\end{dfn}

Now, we have several ways of giving $GX$ a bicategory structure, starting from a representable merge-bicategory $X$, all of them leading to the same result (simply unravel the definitions):
\begin{enumerate}
	\item use the tensor representability of $U_1 X$, applying Construction \ref{cons:cvu-bicat};
	\item use the par representability of $U_1 X$, applying the dual of Construction \ref{cons:cvu-bicat};
	\item apply Construction \ref{cons:cvu-linbicat} to $U_1 X$, choosing tensors and pars, and witnesses of tensor and par unitality, that are mutual inverses, as granted by Corollary \ref{cor:tensorparcollapse} and Proposition \ref{cor:tensorparunitcollapse}: the result is a degenerate linear bicategory, which can be pulled back through the inclusion $\imath: \bicat \to \linbicat$.
\end{enumerate}
In fact, in constructing the bicategory directly from a merge-bicategory, some of the definitions can be simplified, using the algebraic characterisation of universal 2-cells (invertibility), instead of their universal property. For example, given a pair of 2-cells $p: (a) \to (c)$, $q: (b) \to (d)$, such that $\bord{}{+}\bord{}{+}p = \bord{}{-}\bord{}{-}q$, we can define $p \otimes q$ as the composite
\begin{equation*}
\input{img/s5_horizontal_algebraic.tex}
\end{equation*} 

Conversely, we can lift $\int \imath: \bicat \to \linbicat \to \pbcat_\otimes^\parr$ to $\mrgpol_\otimes$, using Theorem \ref{thm:functocoherence}, which is applicable to degenerate linear bicategories whose distributors are associators, to define the composition of 2-cells along multiple 1-cells. 

We keep the notation $G: \mrgpol_\otimes \to \bicat$ and $\int: \bicat \to \mrgpol_\otimes$ for the functors that correspond to the two constructions. We have all that is needed to state the following.
\begin{thm} \label{mrgpol_equivalence}
$G: \mrgpol_\otimes \to \bicat$ and $\int: \bicat \to \mrgpol_\otimes$ are two sides of an equivalence of large categories.
\end{thm}

What advantage does this have over the equivalence with $\mbcat_\otimes$? First of all, unlike (regular) poly-bicategories, merge-bicategories have a natural monoidal closed structure, giving access to higher morphisms. However, we will not try to extend the equivalence of Proposition \ref{mrgpol_equivalence} to some bi-equivalence of bicategories, or tri-equivalence of tricategories: instead, working with closed structures allows us to consider lax transformations that do not fit properly into a higher-categorical structure on the same footing as the pseudo-natural transformations that do. 

\begin{cons} \label{cons:laxgray}
Let $X$, $Y$ be merge-bicategories; we define a new merge-bicategory $X \tensor Y$ as follows.
\begin{itemize}
	\item The 0-cells of $X \tensor Y$ are of the form $x \tensor y$, for $x$ in $X_0$ and $y$ in $Y_0$.
	\item The 1-cells of $X \tensor Y$ are either of the form $x \tensor b: x \tensor y \to x \tensor y'$, for $x$ in $X_0$ and $b: y \to y'$ in $Y_1$, or of the form $a \tensor y: x \tensor y \to x' \tensor y$, for $a: x \to x'$ in $X_1$ and $y$ in $Y_0$.
	\item The 2-cells of $X \tensor Y$ are generated under the merge operations by 
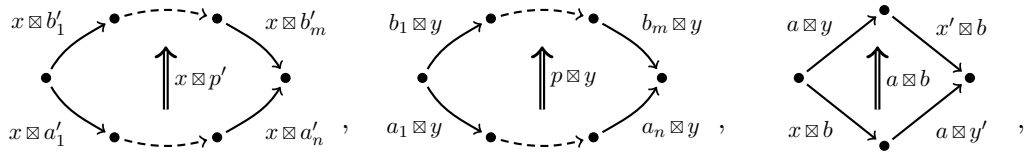
\begin{equation*}
\input{img/s5_grayproduct1.tex}
\end{equation*}
	for all 0-cells $x, x'$ in $X$ and $y, y'$ in $Y$, all 1-cells $a: x \to x'$ in $X$ and $b: y \to y'$ in $Y$, and all 2-cells $p: (a_1,\ldots,a_n) \to (b_1,\ldots,b_m)$ in $X$, $p': (a'_1,\ldots,a'_n) \to (b'_1,\ldots,b'_m)$ in $Y$, subject to the equations 
	\begin{equation*}
		x \tensor \mrg{[j_1,j_2]}{[i_1,i_2]}(p',q') = \mrg{[j_1,j_2]}{[i_1,i_2]}(x\tensor p', x\tensor q'), 
	\end{equation*}
	\begin{equation*}
	 \mrg{[j_1,j_2]}{[i_1,i_2]}(p,q) \tensor y = \mrg{[j_1,j_2]}{[i_1,i_2]}(p\tensor y, q \tensor y),
	 \end{equation*}
	 whenever the left-hand side is defined for 0-cells $x$ in $X$ and $y$ in $Y$, and 2-cells $p, q$ in $X$ and $p',q'$ in $Y$, and
	 \begin{equation} \label{grayproduct2}
\input{img/s5_grayproduct2.tex}
\end{equation} 
	 \begin{equation} \label{grayproduct3}
\input{img/s5_grayproduct3.tex}
\end{equation}

whenever the two sides are well-defined. In the diagrams (\ref{grayproduct2}), (\ref{grayproduct3}), a number of ``squares'' $c\tensor b'_j$, $c \tensor a'_i$, or $a_i \tensor c'$, $b_j \tensor c'$ is implied.
\end{itemize}
We extend this construction to pairs of morphisms $f: X \to X'$, $g: Y \to Y'$, by defining $(f\tensor g)(x \tensor y) := f(x)\tensor g(y)$ on generators, and extending freely to generic cells.
\end{cons}
\begin{dfn}
We call the merge-bicategory $X \tensor Y$, as defined in Construction \ref{cons:laxgray}, the \emph{lax Gray product} of $X$ and $Y$.
\end{dfn}
\begin{prop} \label{prop:graymonoidal}
The lax Gray product defines a monoidal structure on $\mrgpol$, whose unit is the merge-bicategory $1$ with a single 0-cell and no 1-cells or 2-cells.
\end{prop}
\begin{proof}
(Sketch.) The diagrams composable with the merge operations are classified by the constructible 2-molecules of \cite{hadzihasanovic2018combinatorial}; this allows us to realise $\mrgpol$ as a reflective subcategory of the category $\rpol$ of constructible polygraphs, as defined there. 

Briefly, given a merge-bicategory $X$, we obtain a constructible polygraph $\imath X$ whose generators of dimension $n = 0, 1, 2$ are the same as the $n$-cells of $X$, there is a unique 3-dimensional generator $e_{p,q}: p \to q$ between any two diagrams whose composites in $X$ are equal, and a unique $n$-dimensional generator, for $n > 3$, between any two $(n-1)$-dimensional constructible diagrams with compatible boundaries. Conversely, given a regular polygraph $P$, we obtain a merge-bicategory $\tau P$ whose 0-cells and 1-cells are the same as the 0 and 1-dimensional generators of $P$, and whose 2-cells are all 2-dimensional constructible diagrams $p$ in $P$, quotiented by the equation $p = q$ if there exists a 3-dimensional generator $p \to q$ or $q \to p$ in $P$; the merge operations are induced by free composition of the generators.

These define adjoint functors $\imath: \mrgpol \to \rpol$ and $\tau: \rpol \to \mrgpol$, where $\imath$ is full and faithful. The unit $\eta_P: P \to \imath\tau P$ is the identity on generators of dimension $n < 3$; given a 3-dimensional generator $f: p \to q$, since $p = q$ as 2-cells of $\tau P$, it follows that $p$ and $q$, as composable diagrams of generators in $P$, both merge to the same 2-cell in $\tau P$. By construction, there is a unique generator $e_{p,q}: p \to q$ in $\imath \tau P$, and we set $\eta_P(f) := e_{p,q}$. On $n$-cells with $n > 3$, $\eta_P$ is just the quotient of all generating cells whose boundaries are equal in $\imath \tau P$. The counit $\varepsilon_P: \tau\imath X \to X$ is the identity on 0-cells and 1-cells, and maps any 2-cell of $\tau\imath X$, which corresponds to a composable diagram in $X$, to its unique composite in $X$.

The monoidal structure on $\rpol$ then induces a monoidal structure on $\mrgpol$ with $X \tensor Y := \tau(\imath X \tensor \imath Y)$, of which Construction \ref{cons:laxgray} is an explicit description. 
\end{proof}

\begin{remark}
The proof sketch of Proposition \ref{prop:graymonoidal} mirrors the construction of the lax Gray product of strict 2-categories, as induced by the lax Gray product of strict $\omega$-categories \cite{crans1995pasting,steiner2004omega}. One could also try to use this directly, realising $\mrgpol$ as a (non-full) subcategory of the category of strict $\omega$-categories, whose objects satisfy certain properties. However, it seems more complicated to check that the monoidal structure on strict $\omega$-categories induces the one on merge-bicategories.
\end{remark}

The category of merge-bicategories with the lax Gray product is closed on both sides; right homs and left homs can be calculated using the fact that cells of a merge-bicategory are in bijective correspondence with morphisms from certain representing objects in $\mrgpol$. 

\begin{cons}
Given merge-bicategories $X$ and $Y$, we describe explicitly the left hom $[X,Y]$ from $X$ to $Y$.

\begin{itemize}
	\item The 0-cells of $[X,Y]$ are morphisms $X \to Y$.
	\item The 1-cells $\sigma: f \to g$ of $[X,Y]$ are \emph{oplax transformations}, assigning to each 0-cell of $X$ a 1-cell $\sigma_x: f(x) \to g(x)$ of $Y$, and to each 1-cell $a: x \to y$ of $X$ a 2-cell $\sigma_a: (f(a),\sigma_y) \to (\sigma_x, g(a))$ of $Y$, such that, for all $p: (a_1,\ldots,a_n) \to (b_1,\ldots,b_m)$ in $X$, the equation
\begin{equation} \label{eq:oplaxnatural}
\input{img/s5_laxnatural.tex}
\end{equation}
holds in $Y$.
	\item The 2-cells $\mu: (\sigma^1,\ldots,\sigma^n) \to (\tau^1,\ldots,\tau^m)$ are \emph{modifications}, assigning to each 0-cell $x$ of $X$ a 2-cell $\mu_x: (\sigma^1_x,\ldots,\sigma^n_x) \to (\tau^1_x,\ldots,\tau^m_x)$ of $Y$, such that, for all 1-cells $a: x \to y$ of $X$, the equation
\begin{equation} \label{eq:modification}
\input{img/s5_modification.tex}
\end{equation}
	holds in $Y$. Mergers of 2-cells are calculated pointwise in $Y$, that is, $\mrg{[j_1,j_2]}{[i_1,i_2]}(\mu,\nu)$ assigns to the 0-cell $x$ the 2-cell $\mrg{[j_1,j_2]}{[i_1,i_2]}(\mu_x,\nu_x)$.
\end{itemize}
\end{cons}

We follow the convention of \cite[Section 7.5]{borceux1994handbook} regarding lax and oplax transformations; the right hom from $X$ to $Y$ in $\mrgpol$ has the analogues of lax transformations as 1-cells. 

\begin{dfn} A 1-cell $\sigma$ of $[X,Y]$ is a \emph{pseudo-natural transformation} if its 2-cell components are all universal. We say that $\sigma$ is a \emph{pseudo-natural equivalence} if its 1-cell and 2-cell components are all universal.
\end{dfn}

We prove some simple results relating properties of $Y$ to properties of $[X,Y]$.
\begin{prop} \label{prop:hominherit}
Let $X, Y$ be merge-bicategories. Then:
\begin{enumerate}[label=($\alph*$)]
	\item if $Y$ is unital, then so is $[X,Y]$;
	\item if $Y$ is also 1-representable, then so is $[X,Y]$;
	\item if $Y$ is representable, then so is $[X,Y]$.
\end{enumerate}
\end{prop}
\begin{proof}
All the implications are proved by defining units and universal cells in $[X,Y]$ to be units and universal cells pointwise in $Y$. For example, suppose $Y$ is unital. Then, given any sequence of natural transformations $(\sigma^1,\ldots,\sigma^n)$, we define a modification $\idd{(\sigma^1,\ldots,\sigma^n)}: (\sigma^1,\ldots,\sigma^n) \to (\sigma^1,\ldots,\sigma^n)$ whose component at the 0-cell $x$ of $X$ is $\idd{(\sigma^1_x,\ldots,\sigma^n_x)}$. It is straightforward to prove that this is a unit in $[X,Y]$.

Suppose $Y$ is 1-representable. Then, if $\sigma: f \to g$ and $\tau: g \to h$ are oplax transformations, pick universal 2-cells $\mu_x: (\sigma_x, \tau_x) \to (\sigma_x \otimes \tau_x)$ in $Y$ for each 0-cell $x$ of $X$. Then, define an oplax transformation $\sigma \otimes \tau: f \to h$ assigning to each 0-cell $x$ the 1-cell $\sigma_x \otimes \tau_x$, and to each 1-cell $a: x \to y$ the 2-cell
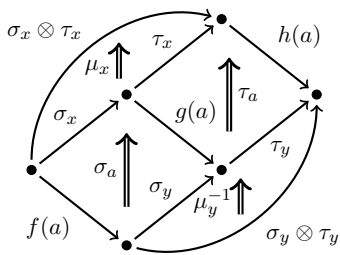
\begin{equation*}
\input{img/s5_tensor_laxnat.tex}
\end{equation*}
of $Y$. It can be checked that $\mu: (\sigma, \tau) \to (\sigma \otimes \tau)$ with components $\mu_x$ is a universal 2-cell in $[X,Y]$, which makes $[X,Y]$ 1-representable.

Finally, let $Y$ be representable, and pick 1-units $1_y$ on each 0-cell $y$ of $Y$, together with a coherent family of witnesses of unitality $\{l_a, r_a\}$. For each 0-cell $f$ of $[X,Y]$, let $1_f: f \to f$ be the oplax transformation assigning to each 0-cell $x$ of $X$ the 1-unit $1_{f(x)}: f(x) \to f(x)$, and to each 1-cell $a: x \to y$ of $X$ the 2-cell
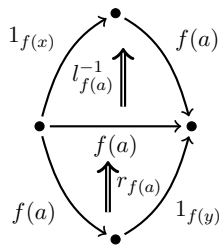
\begin{equation*}
\input{img/s5_unit_laxnat.tex}
\end{equation*}
of $Y$; Theorem \ref{thm:polycoherent} and its dual ensure that this is well-defined. Then $1_f$ is a 1-unit in $[X,Y]$, as witnessed, for any oplax transformations $\sigma: f \to g$ and $\tau: h \to f$, by the modifications $l_\sigma: (1_f, \sigma) \to (\sigma)$ with components $l_{\sigma_x}$, and $r_\tau: (\tau, 1_f) \to (\tau)$ with components $r_{\tau_x}$, for each 0-cell $x$.
\end{proof}

In a bicategory, seen as a degenerate linear bicategory, the notion of linear adjunction collapses to an ordinary adjunction. Through the equivalence between $\bicat$ and $\mrgpol_\otimes$, we can import the well-known result that any equivalence in a bicategory can be improved to an adjoint equivalence; see for example \cite[Proposition 1.5.7]{leinster2004higher}. The following could be stated under less restrictive assumptions, but we will not need the extra generality.
\begin{lem} \label{lem:adjointequivalence}
Let $e: x \to y$ be a universal 1-cell in a representable merge-bicategory. Then there exist a universal 1-cell $e^*: y \to x$ and an adjunction $(\varepsilon, \eta): e^* \dashv e$ whose counit $\varepsilon$ and unit $\eta$ are universal 2-cells.
\end{lem}

The second statement of the following result is also well-known for bicategories and oplax transformations of functors (see the definition below), but the proof goes through in a less restrictive context.
\begin{prop}
Let $X, Y$ be merge-bicategories, and suppose $Y$ is unital. Then:
\begin{enumerate}[label=($\alph*$)]
	\item a modification $\mu$ is a universal 2-cell in $[X,Y]$ if and only if all its components $\mu_x$ are universal 2-cells in $Y$;
	\item if $Y$ is representable, an oplax transformation $\sigma$ is a universal 1-cell in $[X,Y]$ if and only if all its 1-cell components $\sigma_x$ and its 2-cell components $\sigma_a$ are universal in $Y$.
\end{enumerate}
\end{prop}
\begin{proof}
By Proposition \ref{prop:hominherit}, $[X,Y]$ inherits from $Y$ the property of being unital or representable. For the first statement, we can use the algebraic characterisation of universal 2-cells of Proposition \ref{prop:mergeunital}, combined with the description of units in $[X,Y]$ as pointwise units in $Y$.

Let $Y$ be representable, and $\sigma: f \to g$ be universal in $[X,Y]$. Because $[X,Y]$ is representable, we can apply Lemma \ref{lem:adjointequivalence}, and obtain another oplax transformation $\sigma^*: g \to f$, together with universal modifications $\varepsilon: (\sigma, \sigma^*) \to (1_f)$ and $\eta: (1_g) \to (\sigma^*, \sigma)$, where units are chosen as in the proof of Proposition \ref{prop:hominherit}. Then, for all 0-cells $x$ of $X$, the $\varepsilon_x$ and $\eta_x$ are universal 2-cells, that also witness an adjunction between $e^*_x$ and $e_x$. This proves that the $\sigma_x$ are universal 1-cells.

Let $a: x \to y$ be a 1-cell of $X$. Using the fact that $\sigma^*$ is an oplax transformation, and $\varepsilon_x$ and $\eta_x$ are the counit and unit of an adjunction, it is straightforward to check that the 2-cells $\invrs{\sigma_a}: (\sigma_x,g(a)) \to (f(a),\sigma_y)$ obtained by the unique factorisation
\begin{equation} \label{eq:pointwise}
\input{img/s5_pointwise1.tex}
\end{equation}
are inverses of the $\sigma_a: (f(a),\sigma_y) \to (\sigma_x,g(a))$.

Conversely, if all components of $\sigma: f \to g$ are universal, for each 0-cell $x$ in $X$ we can fix an adjunction $(\varepsilon_x, \eta_x): \sigma^*_x \dashv \sigma_x$, and for each 1-cell $a: x \to y$ we can read equation $(\ref{eq:pointwise})$ backwards as a definition of $\sigma^*_a$, starting from the inverse of $\sigma_a$. The $\sigma^*_x$ and $\sigma^*_a$ assemble into an oplax transformation $\sigma^*: g \to f$, and the $\varepsilon_x$ and $\eta_x$ into universal modifications $\varepsilon: (\sigma, \sigma^*) \to (1_f)$ and $\eta: (1_g) \to (\sigma^*, \sigma)$. Proposition \ref{prop:isomorphism} and the representability of $[X,Y]$ allow us to conclude.
\end{proof}

We recall the definition of oplax transformations and modifications in the context of bicategories.
\begin{dfn}
Let $f, g: B \to C$ be functors of bicategories. An \emph{oplax transformation} $\sigma: f \to g$ is the data of
\begin{enumerate}
	\item a family of 1-cells $\{\sigma_x: f(x) \to g(x)\}$ in $C$, indexed by 0-cells $x$ of $B$, and
	\item a family of 2-cells $\{\sigma_a: f(a)\otimes \sigma_y \to \sigma_x \otimes g(a)\}$ in $C$, indexed by 1-cells $a: x \to y$ of $B$,
\end{enumerate}
such that, for all 1-cells $a: x \to y$, $b: y \to z$ in $B$, the following diagrams commute in $C$:
\begin{equation} \label{eq:oplax_coherence}
\begin{tikzpicture}[baseline={([yshift=-.5ex]current bounding box.center)}]
	\node[scale=1.25] (0) at (-5.9,.75) {$f(a) \!\otimes\! (f(b) \!\otimes\! \sigma_z)$};
	\node[scale=1.25] (a2) at (0,.75) {$(f(a) \!\otimes\! f(b)) \!\otimes\! \sigma_z$};
	\node[scale=1.25] (a3) at (5.9,.75) {$f(a \!\otimes\! b) \!\otimes\! \sigma_z$};
	\node[scale=1.25] (b1) at (-5.9,-.75) {$f(a) \!\otimes\! (\sigma_y \!\otimes\! g(b))$};
	\node[scale=1.25] (b2) at (-5.9,-2.25) {$(f(a) \!\otimes\! \sigma_y) \!\otimes\! g(b)$};
	\node[scale=1.25] (b3) at (-2,-2.25) {$(\sigma_x \!\otimes\! g(a)) \!\otimes\! g(b)$};
	\node[scale=1.25] (b4) at (2,-2.25) {$\sigma_x \!\otimes\! (g(a) \!\otimes\! g(b))$};
	\node[scale=1.25] (1) at (5.9,-2.25) {$\sigma_x \!\otimes\! g(a\!\otimes\! b),$};
	\draw[1c] (0.east) to node[auto] {$\invrs{\alpha_{f(a),f(b),\sigma_z}}$} (a2.west);
	\draw[1c] (a2.east) to node[auto] {$f_{a,b} \!\otimes\! \idd{\sigma_z}$} (a3.west);
	\draw[1c] (a3.south) to node[auto] {$\sigma_{a \otimes b}$} (1.north);
	\draw[1c] (0.south) to node[auto,swap] {$\idd{f(a)} \!\otimes\! \sigma_b$} (b1.north);
	\draw[1c] (b1.south) to node[auto,swap] {$\invrs{\alpha_{f(a),\sigma_y,g(b)}}$} (b2.north);
	\draw[1c] (b2.east) to node[auto,swap, below=4pt] {$\sigma_a \!\otimes\! \idd{g(b)}$} (b3.west);
	\draw[1c] (b3.east) to node[auto,swap, below=4pt] {$\alpha_{\sigma_x,g(a),g(b)}$} (b4.west);
	\draw[1c] (b4.east) to node[auto,swap, below=4pt] {$\idd{\sigma_x} \!\otimes\! g_{a,b}$} (1.west);
\end{tikzpicture}
\end{equation}
\begin{equation} \label{eq:oplax_unit_coherence}
\begin{tikzpicture}[baseline={([yshift=-.5ex]current bounding box.center)}]
	\node[scale=1.25] (0) at (-4,.75) {$1_{f(x)} \otimes \sigma_x$};
	\node[scale=1.25] (a1) at (4,.75) {$f(1_x) \otimes \sigma_x$};
	\node[scale=1.25] (b1) at (-4,-.75) {$\sigma_x$};
	\node[scale=1.25] (b2) at (0,-.75) {$\sigma_x \otimes 1_{g(x)}$};
	\node[scale=1.25] (1) at (4,-.75) {$\sigma_x \otimes g(1_x)$};
	\draw[1c] (0.east) to node[auto] {$f_x \otimes \idd{\sigma_x}$} (a1.west);
	\draw[1c] (a1.south) to node[auto] {$\sigma_{1_x}$} (1.north);
	\draw[1c] (0.south) to node[auto,swap] {$l_{\sigma_x}$} (b1.north);
	\draw[1c] (b1.east) to node[auto,swap] {$\invrs{r_{\sigma_x}}$} (b2.west);
	\draw[1c] (b2.east) to node[auto,swap] {$\idd{\sigma_x} \otimes g_x$} (1.west);
	\node[scale=1.25] at (5,-.9) {.};
\end{tikzpicture}
\end{equation}
An oplax transformation $\sigma: f \to g$ is a \emph{pseudo-natural transformation} if the 2-cells $\sigma_a$ are all isomorphisms, and a \emph{pseudo-natural equivalence} if, in addition, the 1-cells $\sigma_x$ are all equivalences.

Given oplax transformations $\sigma, \tau: f \to g$, a \emph{modification} $\mu: \sigma \to \tau$ is a family of 2-cells $\{\mu_x: \sigma_x \to \tau_x\}$ of $C$, indexed by 0-cells $x$ of $B$, such that for all 1-cells $a: x \to y$ in $B$ the diagram
\begin{equation*}
\begin{tikzpicture}[baseline={([yshift=-.5ex]current bounding box.center)}]
	\node[scale=1.25] (0) at (-1.5,.75) {$f(a) \otimes \sigma_y$};
	\node[scale=1.25] (a1) at (2,.75) {$\sigma_x \otimes g(a)$};
	\node[scale=1.25] (b1) at (-1.5,-.75) {$f(a) \otimes \tau_y$};
	\node[scale=1.25] (1) at (2,-.75) {$\tau_x \otimes g(a)$};
	\draw[1c] (0.east) to node[auto] {$\sigma_a$} (a1.west);
	\draw[1c] (a1.south) to node[auto] {$\mu_x \otimes \idd{g(a)}$} (1.north);
	\draw[1c] (0.south) to node[auto,swap] {$\idd{f(a)} \otimes \mu_y$} (b1.north);
	\draw[1c] (b1.east) to node[auto,swap] {$\tau_a$} (1.west);
\end{tikzpicture}
\end{equation*}
commutes in $C$.
\end{dfn}

The correspondence between oplax transformations of functors of bicategories and strong morphisms of representable merge-bicategories is not as simple as one would hope. The commutativity of diagram (\ref{eq:oplax_unit_coherence}) implies that the components $\sigma_{1_x}$ at units are isomorphisms for all oplax transformations $\sigma$ of functors; but that does not seem to be automatic for oplax transformations of strong morphisms. Instead, we need a slight specialisation.
\begin{dfn}
Let $\sigma: f \to g$ be an oplax transformation between morphisms $f,g: X \to Y$ of merge-bicategories. We say that $\sigma$ is \emph{fair} if, for all universal 1-cells $e: x \to y$ in $X$, the component $\sigma_e: (f(e),\sigma_y) \to (\sigma_x, g(e))$ is a universal 2-cell in $Y$.
\end{dfn}

\begin{remark} \label{remark:fairness}
If we see (lax, oplax) transformations as the next step after morphisms in a ladder of ``transfors'' \cite{crans2003localizations}, then fair transformations are a natural next step after strong morphisms: if the latter are characterised by their assigning universal $n$-cells to universal $n$-cells, the former are characterised by their assigning universal $(n+1)$-cells to universal $n$-cells, whenever it makes sense. However, we avoid calling them ``strong'' because this term is generally used in opposition to ``lax'', and it would be confusing to speak of ``strong oplax transformations''.

Restricted to transformations of strong morphisms, we can also see fairness as a requirement for ``bi-morphisms'' $Z \tensor X \to Y$ to be strong parametrically in each variable: an oplax transformation of morphisms $X \to Y$ is the same as a morphism $\sigma: \vec{I} \tensor X \to Y$, where $\vec{I}$ is the merge-bicategory with two 0-cells and a single 1-cell between them, and it is a fair oplax transformation of strong morphisms precisely when $\sigma(x \tensor -)$ takes universal cells of $X$ to universal cells of $Y$, for all cells $x$ of $\vec{I}$.

Note that all pseudo-natural transformations are trivially fair.
\end{remark}

\begin{lem} \label{lem:fair}
Let $X, Y$ be representable merge-bicategories, $f, g: X \to Y$ strong morphisms, and $\sigma: f \to g$ an oplax transformation. The following are equivalent:
\begin{enumerate}
	\item $\sigma$ is fair;
	\item for all 1-units $1_x: x \to x$ in $X$, the 2-cell $\sigma_{1_x}: (f(1_x),\sigma_x) \to (\sigma_x,f(1_x))$ is universal in $Y$.
\end{enumerate}
\end{lem}
\begin{proof}
The implication from $(a)$ to $(b)$ is obvious. Conversely, let $e: x \to y$ be a universal 1-cell in $X$, and take a 1-cell $e^*: y \to x$ and an adjunction $(\varepsilon, \eta): e^* \dashv e$ in $X$ as by Lemma \ref{lem:adjointequivalence}. Then, postcompose $\sigma_e: (f(e),\sigma_y) \to (\sigma_x,g(e))$ with
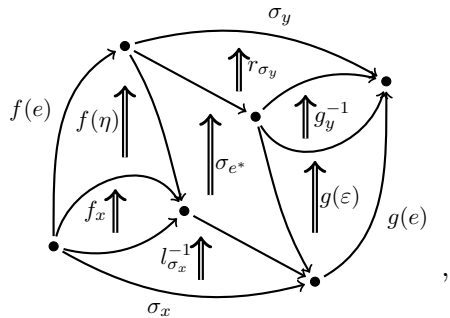
\begin{equation*}
\input{img/s5_fair2.tex}
\end{equation*}
where we omitted the labels of the inner 1-cells to avoid clutter. Using
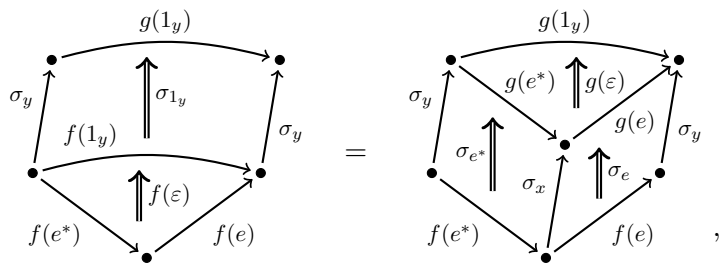
\begin{equation*}
\input{img/s5_fair1.tex}
\end{equation*}
together with the definition of $f_x$, and the fact that $\varepsilon$ and $\eta$ are the counit and unit of an adjunction, we obtain that the result is equal to
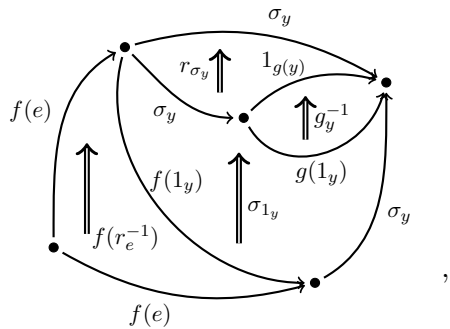
\begin{equation*}
\input{img/s5_fair3.tex}
\end{equation*}
which is a composite of universal 2-cells if $\sigma_{1_y}$ is universal (in fact, it will turn out to be a unit).  From this, we deduce that $\sigma_e$ has a left inverse, and we can show similarly that it has a right inverse, which proves the claim.
\end{proof}

\begin{dfn}
Let $X, Y$ be representable merge-bicategories. We write $[X,Y]_\text{s}$ for the restriction of the left hom $[X,Y]$ that has strong morphisms as 0-cells, fair oplax transformations as 1-cells, and all their modifications as 2-cells. We write $[X,Y]_\text{ps}$ for the restriction of $[X,Y]_\text{s}$ that has only pseudo-natural transformations as 1-cells.
\end{dfn}
Fair oplax transformations and pseudo-natural transformations are closed under tensors, and 1-units in $[X,Y]$ are pseudo-natural by construction, so $[X,Y]_\text{s}$ and $[X,Y]_\text{ps}$ are still representable when $[X,Y]$ is.

\begin{thm} \label{thm:fair_coherence}
The following correspond through the equivalence between $\bicat$ and $\mrgpol_\otimes$:
\begin{enumerate}
	\item oplax transformations and modifications in $\bicat$, and fair oplax transformations and modifications in $\mrgpol_\otimes$;
	\item pseudo-natural transformations in $\bicat$ and in $\mrgpol_\otimes$;
	\item pseudo-natural equivalences in $\bicat$ and in $\mrgpol_\otimes$.
\end{enumerate}
\end{thm}
\begin{proof}
Let $f, g: B \to C$ be functors of bicategories, and $\sigma: f \to g$ an oplax transformation. By construction, there are 1-cells $\sigma_x: (\int\! f)(x) \to (\int\! g)(x)$ of $\int\! C$, and 2-cells $\sigma_a: ((\int\! f)(a),\sigma_y) \to (\sigma_x, (\int\! g)(a))$ of $\int\! C$ corresponding to the components of $\sigma$, as required by the definition of an oplax transformation between the strong morphisms $\int\! f$ and $\int\! g$. 

Unravelling the interpretation of the equations (\ref{eq:oplaxnatural}) from $\int\! C$ into $C$, and using the commutativity of the diagrams (\ref{eq:oplax_coherence}) in $C$, we can check that this is indeed an oplax transformation, so it suffices to prove that it is fair. This follows from Lemma \ref{lem:fair}, since the commutativity of the diagrams (\ref{eq:oplax_unit_coherence}) in $C$ implies that $\sigma_{1_x}$ is a universal 2-cell for each unit $1_x$.

Conversely, if $f, g: X \to Y$ are strong morphisms of representable merge-bicategories, and $\sigma: f \to g$ is a fair oplax transformation, the components $\sigma_x: f(x) \to g(x)$ are automatically 1-cells $Gf(x) \to Gg(x)$ in $GY$, and the components $\sigma_a$ induce 2-cells $Gf(a) \otimes \sigma_y \to \sigma_x \otimes Gg(a)$ in $GY$ by representability, so it suffices to prove that they define an oplax transformation between the functors $Gf$ and $Gg$.

Given 1-cells $a: x \to y$, $b: y \to z$, with a chosen tensor $t_{a,b}$ in $X$, the commutativity of diagram (\ref{eq:oplax_coherence}) in $GY$ follows straightforwardly from 
\begin{equation*}
\input{img/s5_fair_coherence1.tex}
\end{equation*}
factorising the two sides through the same universal 2-cells in $Y$. Moreover, from
\begin{equation*}
\input{img/s5_fair_coherence2.tex}
\end{equation*}
by the defining equations (\ref{eq:functor_unitor_def}) of $f_x$ and $g_x$, and coherence of witnesses of unitality in $Y$, it follows that
\begin{equation*}
\input{img/s5_fair_coherence3.tex}
\end{equation*}
Using the universality of $f(1_x)$, and that of $\sigma_{1_x}$, as granted by the fairness of $\sigma$, we can cancel the leftmost 2-cell on both sides. The remaining equation implies the commutativity of (\ref{eq:oplax_unit_coherence}) in $GY$.

The correspondence between modifications is immediate from the definitions, and the correspondence for pseudo-natural transformations and equivalences is a simple specialisation of the first part, using the fact that pseudo-natural transformations are automatically fair.
\end{proof}

It follows from Theorem \ref{thm:fair_coherence} that, for all bicategories $B$ and $C$, we can define the bicategory of functors, oplax transformations, and modifications between $B$ and $C$ as $G[\int\! B,\int\! C]_\text{s}$, and the bicategory of functors, pseudo-natural transformations, and modifications as $G[\int\! B,\int\! C]_\text{ps}$. The latter is part of a monoidal closed structure on $\bicat$ with the better-known ``pseudo'' version of the Gray product, see for example \cite{bourke2016gray}, but we cannot see any added insight from the perspective of merge-bicategories.

As usual, we obtain a dual correspondence between lax transformations of functors and fair lax transformations of strong morphisms. Theorem \ref{thm:fair_coherence} also allows us to transport the notion of equivalence from bicategories to representable merge-bicategories.
\begin{dfn}
A strong morphism $f: X \to Y$ of representable merge-bicategories is an \emph{equivalence} if there exist a strong morphism $g: Y \to X$ and pseudo-natural equivalences $\eta: \idd{X} \to gf$ and $\varepsilon: \idd{Y} \to fg$.

A functor $f: B \to C$ of bicategories is an \emph{equivalence} if there exist a functor $g: C \to B$ and pseudo-natural equivalences $\eta: \idd{B} \to gf$ and $\varepsilon: \idd{C} \to fg$.
\end{dfn}

\begin{cor} \label{cor:equivalences}
Let $f: X \to Y$ be an equivalence of representable merge-bicategories. Then $Gf: GX \to GY$ is an equivalence of bicategories. Conversely, if $f: B \to C$ is an equivalence of bicategories, $\int\! f: \int\! B \to \int\! C$ is an equivalence of representable merge-bicategories.
\end{cor}

\section{Semi-strictification} \label{sec:strictify}

The definitions and constructions of Section \ref{sec:probicat} were all given with an eye towards higher dimensions. In particular, Proposition \ref{def:representable} and Proposition \ref{def:representablestrong} generalise to the definition of \emph{representable constructible polygraphs} and strong morphisms, stated in terms of the existence and preservation of universal cells \cite[Appendix B]{hadzihasanovic2018combinatorial}, of which representable merge-bicategories are a truncated version. 

We essentially defined the lax Gray product of merge-bicategories as the truncation of the lax Gray product of constructible polygraphs, and our characterisation of ``fairness'' of transformations, as explained in Remark \ref{remark:fairness}, can also scale up to higher morphisms.

Our aim, in this section, is to develop a strictification strategy, based on this theory, that also has the potential to scale up. This means that, at the 2-dimensional level, we need to do a little worse than the strongest strictification result for bicategories, which says that they are equivalent to strict ones, as this is known to be false for tricategories. 

We warn that a general theory in higher dimensions has not yet been fully developed, and that, in any case, it remains to be shown that representable constructible polygraphs will capture the desired properties of weak higher categories (for example, that the groupoidal ones satisfy a form of the homotopy hypothesis). 

In the context of merge-bicategories, strictness is the property of a choice of representing cells, and related witnesses: it does not make sense to speak of ``strict representable merge-bicategories'' \emph{tout court}. In \cite{hermida2000representable}, Hermida encoded such a choice for a representable multicategory into the structure of a pseudoalgebra for a 2-monad; following a general scheme for coherence theorems \cite{power1989general, lack2002codescent}, he then formulated strictification as an equivalence between pseudoalgebras and strict algebras.

Our approach, employing an ``ordinary'' monad $\tmonad{}$, will differ in two ways:
\begin{enumerate}
	\item in order to streamline the discussion, and focus on local combinatorics rather than a global theorem on a category of pseudoalgebras, we show directly that a representable merge-bicategory is equivalent to one that supports the structure of a strict $\tmonad{}$-algebra;
	
	\item because we are interested in semi-strictification, we separate $\tmonad{}$ into the composite of two monads $\imonad{}$ (for \emph{inflate}) and $\mmonad{}$ (for \emph{merge}), related by a distributive law: roughly, the first encodes the structure relative to units, and the second the structure relative to composition in the ``merger'' sense. Supporting the structure of a strict $\imonad{}$-algebra is a weak enough condition that any representable merge-bicategory satisfies it, whereas strictification is needed to support an $\mmonad{}$-algebra structure. 
\end{enumerate}
The main point is that the construction of tensors, that is, witnesses of composition can be separated in two steps:
\begin{enumerate}
	\item first we ``only add units'', which suffices to create unit 2-cells $\idd{\Gamma}: (\Gamma) \to (\Gamma)$ with multiple inputs and outputs;
	\item then we ``only add composites'', whereby a composable sequence $\Gamma$ produces a single 1-cell $\sequ{\Gamma}$, and the formal composition of the output boundary of $\idd{\Gamma}$ produces a universal 2-cell $t_{\Gamma}: (\Gamma) \to (\sequ{\Gamma})$, that is, a tensor.
\end{enumerate}
Semi-strictification corresponds to the second step only.

\begin{cons} \label{cons:inflate}
Let $X$ be a merge-bicategory. The merge-bicategory $\imonad{X}$ is generated by the cells of $X$, together with
\begin{itemize}
	\item for each 0-cell $x$ in $X$, a 1-cell $\thin{x}: x \to x$ and a triple of 2-cells $\thin{\thin{x}}: (\thin{x}) \to (\thin{x})$, $r_{\thin{x}} \equiv l_{\thin{x}}: (\thin{x},\thin{x}) \to (\thin{x})$ and $r^*_{\thin{x}} \equiv l^*_{\thin{x}}: (\thin{x}) \to (\thin{x},\thin{x})$;
	\item for each 1-cell $a: x \to y$ in $X$, 2-cells $\thin{a}: (a) \to (a)$, $l_a: (\thin{x},a) \to (a)$, $l^*_a: (a) \to (\thin{x},a)$, $r_a: (a,\thin{y}) \to (a)$, and $r^*_a: (a) \to (a,\thin{y})$,
\end{itemize}
subject to the following equations:
\begin{itemize}
	\item for all 2-cells $p, q$ in $X$, if $\mrg{[j_1,j_2]}{[i_1,i_2]}(p,q)$ is defined, it is equal in $\imonad{X}$ to its value in $X$;
	\item for all 2-cells $p$ in $\imonad{X}$, if $\bord{i}{+}p = a$ and $\bord{j}{-}p = b$, then $\mrg{[i,i]}{[1,1]}(p, \thin{a}) = p$ and $\mrg{[1,1]}{[j,j]}(\thin{b}, p) = p$;
	\item for all 1-cells $a$ in $\imonad{X}$, $\mrg{[1,2]}{[1,2]}(l_a^*,l_a) = \thin{a}$ and $\mrg{[1,2]}{[1,2]}(r_a^*,r_a) = \thin{a}$;
	\item for all 2-cells $p: (a,\Gamma) \to (b,\Delta)$,
	\begin{equation} \label{eq:naturleft_i}
		\mrg{[1,1]}{[2,2]}(p,l_b) = \mrg{[1,1]}{[1,1]}(l_a,p), \quad \quad \mrg{[2,2]}{[1,1]}(l^*_a,p) = \mrg{[1,1]}{[1,1]}(p,l^*_b)
	\end{equation} 
	and dually for all 2-cells $p': (\Gamma', a) \to (\Delta',b)$ in $X_2^{(n,m)}$,
	\begin{equation} \label{eq:naturright_i}
		\mrg{[m,m]}{[1,1]}(p',r_b) = \mrg{[1,1]}{[n,n]}(r_a,p'), \quad \quad \mrg{[1,1]}{[n,n]}(r^*_a,p') = \mrg{[m,m]}{[1,1]}(p',r^*_b);
	\end{equation}
	\item for all 2-cells $p: (\Gamma) \to (\Delta)$, if $\bord{j}{-}p = a, \bord{j+1}{-}p = b, \bord{i}{+}p = a', \bord{i+1}{+}p = b'$
	\begin{equation} \label{eq:triangle_i}
		\mrg{[1,1]}{[j,j]}(r_a,p) = \mrg{[1,1]}{[j+1,j+1]}(l_b,p), \quad \mrg{[i,i]}{[1,1]}(p,r^*_{a'}) = \mrg{[i+1,i+1]}{[1,1]}(p,l^*_{b'}).
	\end{equation}
\end{itemize}
There is an obvious inclusion morphism $\eta_X: X \to \imonad{X}$. 

Given a morphism $f: X \to Y$, we define a morphism $\imonad{f}: \imonad{X} \to \imonad{Y}$, which is equal to $\eta_Y f$ on $\eta_X$, maps the 1-cells $\thin{x}$ to $f(\thin{x}) := \thin{f(x)}$, and the 2-cells $\thin{a}, l_a, l^*_a, r_a, r_a^*$ to $\thin{f(a)}, l_{f(a)}, l^*_{f(a)}, r_{f(a)}$, $r^*_{f(a)}$, respectively. This assignment makes $\imonad{}$ an endofunctor on $\mrgpol$.

Moreover, consider the iterated construction $\imonad{\imonad{X}}$, and denote the cells added at the second iteration with a tilde. There is a morphism $\mu_X: \imonad{\imonad{X}} \to \imonad{X}$, which is the identity on $\eta_{\imonad{X}}$, maps $\tilde{\thin{x}}$ to $\thin{x}$ (which completely specifies $\mu_X$ on 1-cells), and $\tilde{\thin{a}}, \tilde{l}_a, \tilde{l}^*_a, \tilde{r}_a, \tilde{r}_a^*$ to $\thin{\mu_X a}, l_{\mu_X a}, l^*_{\mu_X a}, r_{\mu_X a}, r_{\mu_X a}^*$, respectively.

The $\eta_X$ and $\mu_X$ are components of natural transformations $\text{Id} \to \imonad{}$ and $\imonad{\imonad{}} \to \imonad{}$, such that $(\imonad{}, \mu, \eta)$ determines a monad on $\mrgpol$.
\end{cons}

\begin{dfn}
We call $(\imonad{}, \mu, \eta)$, as defined in Construction \ref{cons:inflate}, the \emph{inflate} monad on $\mrgpol$.
\end{dfn}

\begin{lem} \label{lem:imonad_rep}
Let $X$ be a merge-bicategory. Then $\imonad{X}$ is 0-representable, the $\{\thin{x}: x \to x\}$ are 1-units, and the $\{l_a, r_a\}$ are coherent witnesses of unitality in the sense of Theorem \ref{thm:polycoherent}.
\end{lem}
\begin{proof}
First, let us show that $\imonad{X}$ is unital. By construction, $\thin{a}: (a) \to (a)$ is a unit on $a$. It suffices to show that a unit exists on any composable pair $(a,b)$ of 1-cells: units on longer sequences can be composed together from these. We define
\begin{equation*}
	\thin{(a,b)} := \mrg{[2,2]}{[1,1]}(r_a^*, l_b).
\end{equation*}
Then, for all 2-cells $p$ with $\bord{j}{-}p = a$ and $\bord{j+1}{-}p = b$, we have
\begin{align*}
	\mrg{[1,2]}{[j,j+1]}(\thin{(a,b)},p) & = \mrg{[1,2]}{[j,j+1]}(r_a^*, \mrg{[1,1]}{[j+1,j+1]}(l_b,p)) = \mrg{[1,2]}{[j,j+1]}(r_a^*, \mrg{[1,1]}{[j,j]}(r_a,p)) = \\
		& = \mrg{[1,1]}{[j,j]}(\mrg{[1,2]}{[1,2]}(r_a^*,r_a),p) = \mrg{[1,1]}{[j,j]}(\thin{a},p) = p,
\end{align*}
and dually on the other side. 

Now, let $x$ be a 0-cell of $\imonad{X}$; we want to show that $\thin{x}: x \to x$ is a 1-unit on $x$, and that the $l_a$ and $r_a$ are witnesses of unitality. First of all, $l^*_a$ is the inverse of $l_a$: $\mrg{[1,2]}{[1,2]}(l_a^*,l_a) = \thin{a}$ is true by definition, and
\begin{equation*}
	\mrg{[1,1]}{[1,1]}(l_a,l^*_a) = \mrg{[1,1]}{[2,2]}(l^*_{\thin{x}},l_a) = \mrg{[1,1]}{[2,2]}(r^*_{\thin{x}},l_a) = \thin{(\thin{x},a)},
\end{equation*}
where we used the rightmost equation in (\ref{eq:naturleft_i}) and the fact that $r^*_{\thin{x}} = l^*_{\thin{x}}$. Dually, we find that $r^*_a$ is the inverse of $r_a$. 

It only remains to show that $l_a$ is universal at $\bord{2}{-}$ and $r_a$ at $\bord{1}{-}$. This follows from their invertibility together with the equations (\ref{eq:naturleft_i}) and (\ref{eq:naturright_i}): given $p: (\thin{x},c,\Gamma) \to (a,\Delta)$, there is a factorisation 
\begin{equation*}
	p = \mrg{[1,1]}{[1,1]}(l_c,p') = \mrg{[1,1]}{[2,2]}(p',l_a),
\end{equation*}
for a unique $p'$, and dually with $r_b$.
\end{proof}

\begin{remark}
The equations that we imposed on $\imonad{X}$ hold in all 0-representable merge-bicategories, when the $\thin{x}$ are 1-units, the $\thin{a}$ are unit 2-cells, and the $\{l_a, r_a\}$ are a coherent family of witnesses of unitality.

Thus we can see $\imonad{X}$ as a free (relatively to $X$) 0-representable merge-bicategory with chosen 1-units and coherent witnesses of unitality. More precisely, $\imonad{}$ arises from an adjunction between $\mrgpol$ and a category whose
\begin{itemize}
	\item objects are 0-representable merge-bicategories with a choice of 1-units and coherent witnesses of unitality, and
	\item morphisms strictly preserve both the 1-units and the witnesses of unitality.
\end{itemize}
\end{remark}

\begin{remark} \label{remark:imonad-no1rep}
Note that $\imonad{X}$ is never 1-representable, unless there are no composable pairs of 1-cells in $X$ at all (in which case, there are no composable pairs of 1-cells in $\imonad{X}$ unless one of them is a 1-unit, so 0-representability implies 1-representability). Otherwise, given a composable pair $a, b$ in $X$, there is no 2-cell $(a,b) \to (c)$ through which $\thin{(a,b)}$ may factor in $\imonad{X}$.

In general, no universal 2-cell or 1-cell of $X$ is universal in $\imonad{X}$, and the inclusion $\eta_X: X \to \imonad{X}$ is not a strong morphism.
\end{remark}

\begin{prop} \label{prop:ialg_0rep}
Let $X$ be a merge-bicategory. The following are equivalent:
\begin{enumerate}
	\item $X$ admits an $\imonad{}$-algebra structure $\alpha: \imonad{X} \to X$ where $\alpha$ is a strong morphism;
	\item $X$ is 0-representable.
\end{enumerate}
\end{prop}
\begin{proof}
Suppose $\alpha: \imonad{X} \to X$ is an $\imonad{}$-algebra structure on $X$. By Lemma \ref{lem:imonad_rep}, there are units in $\imonad{X}$ on all composable sequences $\Gamma$ of 1-cells coming from $X$, and 1-units on all 0-cells $x$ coming from $X$. Since $\alpha\eta_X$ is the identity on $X$, $\alpha$ maps all cells originally in $X$ to themselves, so if $\alpha$ is strong, it maps a unit on $\Gamma$ in $\imonad{X}$ to one in $X$, and a 1-unit on $x$ in $\imonad{X}$ to one in $X$. This proves one implication.

Conversely, suppose $X$ is 0-representable, and fix 1-units $1_x: x \to x$ on each 0-cell of $X$, together with coherent witnesses of tensor and par unitality $\{\tilde{l}_a, \tilde{r}_a\}$ and $\{\invrs{\tilde{l}_a}, \invrs{\tilde{r}_a}\}$, inverse to each other (see Remark \ref{remark:merge-coherent}). Then, define a morphism $\alpha: \imonad{X} \to X$ which is the identity on $\eta_X$, maps $\thin{x}$ to $1_x$, and $\thin{a}, l_a, l^*_a, r_a, r^*_a$ to $\idd{a}, \tilde{l}_a, \invrs{\tilde{l}}_a, \tilde{r}_a, \invrs{\tilde{r}}_a$, respectively. 

By Remark \ref{remark:imonad-no1rep}, there are no universal cells in $\imonad{X}$ except the ones freely added, that is, the 1-cells $\thin{x}$, the 2-cells $\thin{a}, l_a, r_a$ and their inverses. These are all mapped to universal cells, so $\alpha$ is a strong morphism. It is straightforward to verify that it defines an $\imonad{}$-algebra structure. 
\end{proof}

We move on to the second component of $\tmonad{}$.

\begin{cons} \label{cons:mergemonad}
Let $X$ be a merge-bicategory. We define a new merge-bicategory $\mmonad{X}$ as follows. 
\begin{itemize}
	\item The 0-cells of $\mmonad{X}$ are the same as the 0-cells of $X$.
	\item For each composable sequence $\Gamma$ of 1-cells in $X$, starting at the 0-cell $x$ and ending at the 0-cell $y$, there is a 1-cell $\sequ{\Gamma}: x \to y$ in $\mmonad{X}$ (the ``merger'' of $\Gamma$).
	\item For each 2-cell $p: (\Gamma) \to (\Delta)$, and each pair of partitions $\Gamma = \Gamma_1,\ldots,\Gamma_n$ and $\Delta = \Delta_1,\ldots,\Delta_m$ of the inputs and outputs of $p$ into finitely many subsequences, there is a 2-cell $p_{\sequ{\Gamma_1}\ldots\sequ{\Gamma_n}}^{\sequ{\Delta_1}\ldots\sequ{\Delta_m}}: (\sequ{\Gamma_1},\ldots,\sequ{\Gamma_n}) \to (\sequ{\Delta_1},\ldots,\sequ{\Delta_m})$ in $\mmonad{X}$.
\end{itemize}
Composition of 2-cells in $\mmonad{X}$ is informally obtained by completely ``unmerging'' the boundaries of the components, composing in $X$, then merging appropriately the boundaries of the result, as in the following example:
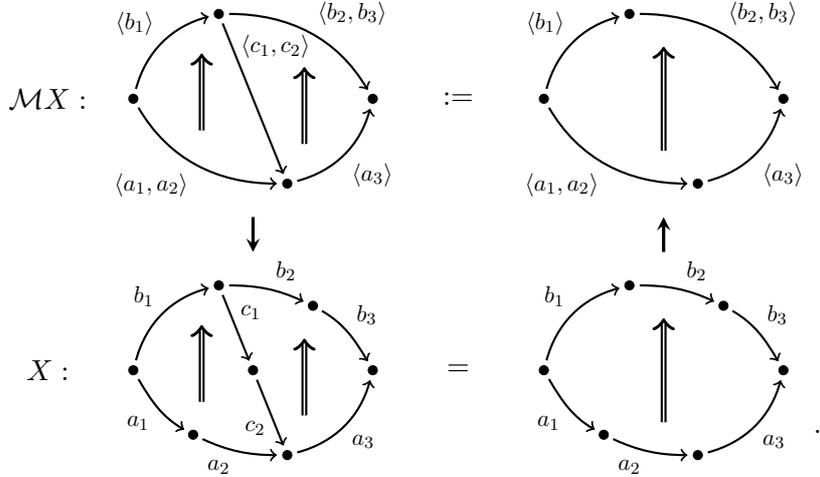
\begin{equation*}
\input{img/s6_merge_multiply.tex}
\end{equation*}

Given a morphism $f: X \to Y$, we define a morphism $\mmonad{f}: \mmonad{X} \to \mmonad{Y}$, equal to $f$ on $0$-cells, mapping 1-cells $\sequ{\Gamma}$ to $\sequ{f(\Gamma)}$, and 2-cells $p_{\sequ{\Gamma_1}\ldots\sequ{\Gamma_n}}^{\sequ{\Delta_1}\ldots\sequ{\Delta_m}}$ to $f(p)_{\sequ{f(\Gamma_1)}\ldots\sequ{f(\Gamma_n})}^{\sequ{f(\Delta_1)}\ldots\sequ{f(\Delta_m)}}$. This assignment makes $\mmonad{}$ and endofunctor on $\mrgpol$.

For each merge-bicategory $X$, there is an inclusion $\zeta_X: X \to \mmonad{X}$, mapping 1-cells $a: x \to y$ to $\sequ{a}: x \to y$, and 2-cells $p: (a_1,\ldots,a_n) \to (b_1,\ldots,b_m)$ to $p_{\sequ{a_1}\ldots\sequ{a_n}}^{\sequ{b_1}\ldots\sequ{b_m}}: (\sequ{a_1},\ldots,\sequ{a_n}) \to (\sequ{b_1},\ldots,\sequ{b_m})$. 

Moreover, there is a morphism $\nu_X: \mmonad{\mmonad{X}} \to \mmonad{X}$, mapping 1-cells $\sequ{\sequ{\Gamma_1},\ldots,\sequ{\Gamma_n}}$ to $\sequ{\Gamma_1,\ldots,\Gamma_n}$, extending to 2-cells in the obvious way. These assemble into natural transformations $\zeta: \text{Id} \to \mmonad{}$ and $\nu: \mmonad{\mmonad{}} \to \mmonad{}$, such that $(\mmonad{}, \nu, \zeta)$ determines a monad on $\mrgpol$.
\end{cons}

\begin{dfn}
We call $(\mmonad{}, \nu, \zeta)$, as defined in Construction \ref{cons:mergemonad}, the \emph{merge} monad on $\mrgpol$.
\end{dfn}

When $X$ is unital, $\mmonad{X}$ can be seen as a free (relatively to $X$) 1-representable merge-bicategory with a chosen \emph{strictly associative} family of tensors, in the following sense.

\begin{dfn}
Let $X$ be a 1-representable merge-bicategory with a chosen family $\{t_{a,b}: (a,b) \to (a \otimes b)\}$ of tensors, indexed by 1-cells $a: x \to y$ and $b: y \to z$ of $X$. We say that $\{t_{a,b}\}$ is \emph{strictly associative} if, for all composable triples $a, b, c$ of 1-cells, $a \otimes (b \otimes c) = (a \otimes b) \otimes c =: a \otimes b \otimes c$ and
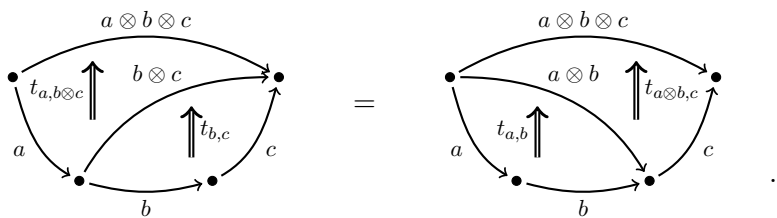
\begin{equation*}
	\input{img/s6_strict_associator.tex}
\end{equation*}
\end{dfn}

\begin{exm}
Not every 1-representable merge-bicategory admits a strictly associative family of tensors. 

For example, there is a merge-bicategory with a single 0-cell and a single 1-cell $\mathbb{N}$, whose 2-cells with $n$ inputs and $m$ outputs correspond to functions $\mathbb{N}^n \to \mathbb{N}^m$, and composition is as in a strict cartesian monoidal category of sets and functions. Any bijective function $\mathbb{N} \times \mathbb{N} \to \mathbb{N}$ (a \emph{pairing} function) defines a tensor $t_{\mathbb{N},\mathbb{N}}$, but a simple argument shows that no such pairing can be strictly associative. 
\end{exm}

Unlike the inclusion into $\imonad{X}$, the inclusion into $\mmonad{X}$ preserves the universal properties of cells of $X$: this will be crucial for the semi-strictification argument.
\begin{lem} \label{lem:merge_inclusion}
Let $X$ be a unital merge-bicategory. Then:
\begin{enumerate}[label=(\alph*)]
	\item $\mmonad{X}$ is 1-representable and admits a strictly associative family of tensors;
	\item the inclusion $\zeta_X: X \to \mmonad{X}$ is a strong morphism;
	\item if $X$ is representable, $\zeta_X$ is an equivalence.
\end{enumerate}
\end{lem} 
\begin{proof}
It follows from the definition of composition in $\mmonad{X}$ that, for each sequence $\sequ{\Gamma_1},\ldots,\sequ{\Gamma_n}$ of 1-cells in $\mmonad{X}$, letting $\Gamma := \Gamma_1,\ldots,\Gamma_n$, the 2-cell $(\idd{\Gamma})_{\sequ{\Gamma_1},\ldots,\sequ{\Gamma_n}}^{\sequ{\Gamma_1},\ldots,\sequ{\Gamma_n}}$ is a unit. This implies both that $\mmonad{X}$ is unital, and that $\zeta_X$ preserves units and universal 2-cells.

Let $\sequ{\Gamma_1}, \sequ{\Gamma_2}$ be a composable pair of 1-cells in $\mmonad{X}$, and let $\Gamma := \Gamma_1, \Gamma_2$ be the corresponding composable sequence of 1-cells in $X$. Then 
\begin{align*}
	t_{\sequ{\Gamma_1},\sequ{\Gamma_2}} & := (\idd{\Gamma})_{\sequ{\Gamma_1}\sequ{\Gamma_2}}^{\sequ{\Gamma}}: (\sequ{\Gamma_1},\sequ{\Gamma_2}) \to (\sequ{\Gamma}), \\
	\invrs{t}_{\sequ{\Gamma_1},\sequ{\Gamma_2}} & := (\idd{\Gamma})_{\sequ{\Gamma}}^{\sequ{\Gamma_1}\sequ{\Gamma_2}}: (\sequ{\Gamma}) \to (\sequ{\Gamma_1},\sequ{\Gamma_2})
\end{align*} 
are mutually inverse 2-cells in $\mmonad{X}$, so they exhibit $\sequ{\Gamma}$ as a tensor (and par) of $\sequ{\Gamma_1}$ and $\sequ{\Gamma_2}$. This proves that $\mmonad{X}$ is 1-representable. Because units compose to units in $X$, the $\{t_{\sequ{\Gamma_1},\sequ{\Gamma_2}}\}$ form a strictly associative family of tensors in $\mmonad{X}$.

Suppose that $1_x: x \to x$ is a tensor unit in $X$, and fix witnesses of unitality $\{l_a, r_a\}$ for $1_x$. Let $\sequ{\Gamma}$ be a 1-cell in $\mmonad{X}$, corresponding to a composable sequence $\Gamma = a_1,\ldots,a_n$ of 1-cells of $X$ with $\bord{}{-}a_1 = x$, and define 
\begin{align*}
	l_\Gamma & := \mrg{[1,1]}{[1,1]}(l_{a_1},\idd{\Gamma}): (1_x,\Gamma) \to (\Gamma) \text{ in $X$}, \\
	l_{\sequ{\Gamma}} & := (l_\Gamma)_{\sequ{1_x}\sequ{\Gamma}}^{\sequ{\Gamma}}: (\sequ{1_x},\sequ{\Gamma}) \to (\sequ{\Gamma}) \text{ in $\mmonad{X}$.} 
\end{align*}
Then $l_{\sequ{\Gamma}}$ is universal at $\bord{1}{+}$ and at $\bord{2}{-}$, hence is a witness of left tensor unitality of $\sequ{1_x}$; one similarly obtains witnesses $r_{\sequ{\Gamma}}$ of right tensor unitality. Together with the dual proof for par units, since $\mmonad{X}$ is 1-representable, this suffices to prove that $\zeta_X$ is a strong morphism, by Corollary \ref{cor:isomorphism_cor}.

Next, suppose that $X$ is representable. For each 0-cell $x$ of $X$, fix a 1-unit $1_x: x \to x$, together with coherent witnesses of unitality $\{l_a, r_a\}$, $\{\invrs{l_a},\invrs{r_a}\}$, and for each composable sequence $\Gamma$ of 1-cells of $X$, fix a 1-cell $c_\Gamma$ and a universal 2-cell $t_\Gamma: (\Gamma) \to (c_\Gamma)$; if $\Gamma$ is a single 1-cell $a$, we assume that $c_a = a$ and $t_a = \idd{a}$. Then, we define a morphism $f: \mmonad{X} \to X$ which is the identity on 0-cells, maps a 1-cell $\sequ{\Gamma}$ to $c_\Gamma$, and maps a 2-cell $p_{\sequ{\Gamma_1}\ldots\sequ{\Gamma_n}}^{\sequ{\Delta_1}\ldots\sequ{\Delta_m}}$ to the unique 2-cell $(c_{\Gamma_1},\ldots,c_{\Gamma_n}) \to (c_{\Delta_1},\ldots,c_{\Delta_m})$ obtained by factorising $p: (\Gamma_1,\ldots,\Gamma_n) \to (\Delta_1,\ldots,\Delta_m)$ through the $t_{\Gamma_i}$ and the $\invrs{t_{\Delta_j}}$, for $i = 1,\ldots, n$ and $j = 1,\ldots,m$.

By construction, $f \zeta_X$ is the identity on $X$, so a 1-unit on $\idd{X}$ in $[X,X]$, which exists by Proposition \ref{prop:hominherit}, is a pseudo-natural equivalence $\idd{X} \to f \zeta_X$. Moreover, we can define families of cells of $\mmonad{X}$ as required by the definition of an oplax transformation $\varepsilon: \idd{\mmonad{X}} \to \zeta_X f$, as follows:
\begin{itemize}
	\item for each 0-cell $x$ of $\mmonad{X}$, let $\varepsilon_x := \sequ{1_x}$;
	\item for each 1-cell $\sequ{\Gamma}: x \to y$ of $\mmonad{X}$, let $\varepsilon_{\sequ{\Gamma}}: (\sequ{\Gamma}, \sequ{1_y}) \to (\sequ{1_x}, \sequ{c_\Gamma})$ be the composite $r_{\sequ{\Gamma}}$, followed by $(t_\Gamma)_{\sequ{\Gamma}}^{\sequ{c_\Gamma}}: (\sequ{\Gamma}) \to (\sequ{c_\Gamma})$, followed by $\invrs{l_{\sequ{c_\Gamma}}}$.
\end{itemize}
It is straightforward to check that this is, in fact, an oplax transformation, and because the $(t_\Gamma)_{\sequ{\Gamma}}^{\sequ{c_\Gamma}}$ are universal, all components of $\varepsilon$ are universal, that is, $\varepsilon$ is a pseudo-natural equivalence. Thus, $\zeta_X$ is an equivalence of representable merge-bicategories. 
\end{proof}

\begin{remark}
Given morphisms of merge-bicategories $f, g: X \to Y$ that agree on 0-cells, let an \emph{icon} $\sigma: f \to g$ be the assignment, to each 1-cell $a: x \to y$ of $X$, of a 2-cell $\sigma_a: (f(a)) \to (g(a))$ of $Y$, such that, for all 2-cells $p: (a_1,\ldots,a_n) \to (b_1,\ldots,b_m)$ of $X$, the equation
\begin{equation*}
\input{img/s6_icon.tex}
\end{equation*}
holds in $Y$. Through the equivalence between $\bicat$ and $\mrgpol_\otimes$, icons between strong morphisms of representable merge-bicategories correspond to icons between functors of bicategories \cite{lack2010icons}.

Since $\zeta_X$ and $f$ are the identity on 0-cells, we can say that $\zeta_X$ is an ``equivalence in the sense of icons'' even when $X$ is 1-representable but not 0-representable. More precisely, there are icons $\idd{X} \to f \zeta_X$ and $\idd{\mmonad{X}} \to \zeta_X f$ whose components are all invertible 2-cells: namely, the icon whose component at a 1-cell $a$ of $X$ is the unit on $a$, and the icon whose component at a 1-cell $\sequ{\Gamma}$ of $\mmonad{X}$ is $(t_\Gamma)_{\sequ{\Gamma}}^{\sequ{c_\Gamma}}: (\sequ{\Gamma}) \to (\sequ{c_\Gamma})$.
\end{remark}

\begin{prop} \label{prop:malg_1rep}
Let $X$ be a unital merge-bicategory. The following are equivalent:
\begin{enumerate}
	\item $X$ admits an $\mmonad{}$-algebra structure $\alpha: \mmonad{X} \to X$ where $\alpha$ is a strong morphism;
	\item $X$ is 1-representable and admits a strictly associative family of tensors.
\end{enumerate}
\end{prop}
\begin{proof}
By Lemma \ref{lem:merge_inclusion}, $\mmonad{X}$ is 1-representable and admits a strictly associative family of tensors $\{t_{\sequ{\Gamma_1},\sequ{\Gamma_2}}\}$. If $\alpha: \mmonad{X} \to X$ is an $\mmonad{}$-algebra structure on $X$, then $\alpha(\sequ{a}) = a$ for each 1-cell $a$ of $X$, and if $\alpha$ is strong, then $\{\alpha(t_{\sequ{a},\sequ{b}})\}$ is a strictly associative family of tensors in $X$, indexed by composable pairs of 1-cells $a, b$ in $X$.

Conversely, suppose $X$ is 1-representable and has a strictly associative family of tensors $\{t_{a,b}\}$. For each composable sequence $\Gamma$ of 1-cells in $X$, there is a unique 1-cell $\otimes \Gamma$ and a unique invertible 2-cell $t_\Gamma: (\Gamma) \to (\otimes\Gamma)$ obtained as a composite of the $t_{a,b}$. We define $\alpha: \mmonad{X} \to X$ to be the morphism that sends each 1-cell $\sequ{\Gamma}$ to $\otimes \Gamma$, and each 2-cell $p_{\sequ{\Gamma_1}\ldots\sequ{\Gamma_n}}^{\sequ{\Delta_1}\ldots\sequ{\Delta_m}}$ to the unique $p': (\otimes \Gamma_1, \ldots, \otimes \Gamma_n) \to (\otimes \Delta_1, \ldots, \otimes \Delta_m)$ such that $p$ factors as a composite of the $t_{\Gamma_i}$, followed by $p'$, followed by the $\invrs{t}_{\Delta_j}$, for $i = 1,\ldots,n$, $j = 1,\ldots,m$. This defines an $\mmonad{}$-algebra structure on $X$.

Clearly $\alpha$ preserves unit 2-cells, so to prove it is strong, by Corollary \ref{cor:isomorphism_cor}, it suffices to show that it preserves 1-units. If $\langle \Gamma \rangle$ is a 1-unit in $\mmonad{X}$, then $\langle \otimes \Gamma \rangle$ is also a 1-unit, so it suffices to look at 1-cells in the image of $\zeta_X$. But any such 1-cell comes, \emph{a fortiori}, from a 1-unit in $X$.
\end{proof}

\begin{remark} 
In fact, to obtain a necessary condition for representability, we could require that $\alpha$ be only compatible with the unit, but not the multiplication of $\mmonad{}$. 
\end{remark}

\begin{lem} \label{lem:mmonad_strong}
Let $f: X \to Y$ be a strong morphism of unital merge-bicategories, and suppose $X$ is 0-representable. Then $\mmonad{f}: \mmonad{X} \to \mmonad{Y}$ is a strong morphism.
\end{lem}
\begin{proof}
By Lemma \ref{lem:merge_inclusion}, both $\mmonad{X}$ and $\mmonad{Y}$ are 1-representable, and it is immediate that $\mmonad{f}$ maps units in $\mmonad{X}$ to units in $\mmonad{Y}$. By Corollary \ref{cor:isomorphism_cor}, it then suffices that $\mmonad{f}$ map a 1-unit on each 0-cell $x$ of $\mmonad{X}$ to a 1-unit in $\mmonad{Y}$. Because $X$ is 0-representable, there is such a 1-unit coming from $X$ through $\zeta_X$, and $\mmonad{f}\circ \zeta_X = \zeta_Y \circ f$ is strong. This proves the claim.
\end{proof}

Let $\tmonad{} := \mmonad{\imonad{}}$. To show that $\tmonad{}$ admits the structure of a monad, we introduce a distributive law from $\imonad{}$ to $\mmonad{}$ \cite{beck1969distributive}. Intuitively, the distributive law encodes the fact that, whenever we merge some 1-cells in a sequence, and then ``inflate'' it to a unit, we can first inflate the original sequence, then merge on both sides of the unit's boundary, instead:
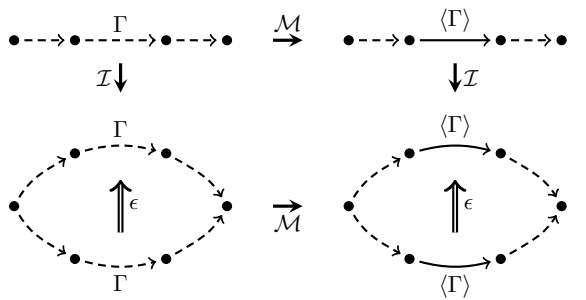
\begin{equation*}
\input{img/s6_distributive.tex}
\end{equation*}

\begin{cons}
Let $X$ be a merge-bicategory. We define $\sigma_X: \imonad{\mmonad{X}} \to \mmonad{\imonad{X}}$ as follows.
\begin{itemize}
	\item For each cell $p$ of $X$, $\eta_{\mmonad{X}}\zeta_X(p)$ is mapped to $\zeta_{\imonad{X}}\eta_X(p)$ (that is, $\sigma_X$ ``is the identity'' on the inclusion of $X$ into both sides).
	\item For each 0-cell $x$ of $\mmonad{X}$ (equivalently, of $X$), $\thin{x}$ is mapped to $\sequ{\thin{x}}$ and $\thin{\thin{x}}$ to ${\thin{\thin{x}}}_{\sequ{\thin{x}}}^{\sequ{\thin{x}}}$.
	\item For each 1-cell $\sequ{\Gamma}: x \to y$ of $\mmonad{X}$, corresponding to a sequence $\Gamma = a_1,\ldots,a_n$ of 1-cells of $X$, $\thin{\sequ{\Gamma}}$ is mapped to $(\thin{(\Gamma)})_{\sequ{\Gamma}}^{\sequ{\Gamma}}$, $l_{\sequ{\Gamma}}$ and $r_{\sequ{\Gamma}}$ to 
	\begin{align*} 
		& (l_\Gamma)_{\sequ{\thin{x}}\sequ{\Gamma}}^{\sequ{\Gamma}}, \text{ where } l_\Gamma := \mrg{[1,1]}{[1,1]}(l_{a_1},\thin{(\Gamma)}) \text{ in $\imonad{X}$}, \\
		& (r_\Gamma)_{\sequ{\Gamma}\sequ{\thin{y}}}^{\sequ{\Gamma}}, \text{ where } r_\Gamma := \mrg{[1,1]}{[n,n]}(r_{a_n},\thin{(\Gamma)}) \text{ in $\imonad{X}$},
	\end{align*}
	respectively, and $l^*_{\sequ{\Gamma}}$ and $r^*_{\sequ{\Gamma}}$ to their respective inverses.
\end{itemize}
The $\sigma_X$ assemble into a natural transformation $\sigma: \imonad{\mmonad{}} \to \mmonad{\imonad{}}$.
\end{cons}

Let $\tilde{\mu},\tilde{\eta}: \tmonad{\tmonad{}} \to \tmonad{}$ be the natural transformations with components 
\begin{equation*}
	\tilde{\mu}_X := \nu_{\imonad{X}}\circ\mmonad{\mmonad{\mu_X}}\circ\mmonad{\sigma_{\imonad{X}}}, \qquad \tilde{\eta}_X := \zeta_{\imonad{X}}\circ \eta_X,
\end{equation*}
for each merge-bicategory $X$.
\begin{prop}
The natural transformation $\sigma: \imonad{\mmonad{}} \to \mmonad{\imonad{}}$ is a distributive law from $(\imonad{},\mu,\eta)$ to $(\mmonad{},\nu,\zeta)$. Consequently, $(\tmonad{}, \tilde{\mu}, \tilde{\eta})$ determines a monad on $\mrgpol$.
\end{prop}
\begin{proof}
An exercise in unpacking definitions.
\end{proof}

\begin{lem} \label{lem:tmonad_rep}
Let $X$ be a merge-bicategory. Then $\tmonad{X}$ is representable and admits a strictly associative family of tensors.
\end{lem}
\begin{proof}
By Lemma \ref{lem:imonad_rep}, $\imonad{X}$ is 0-representable, and by Lemma \ref{lem:merge_inclusion} the inclusion of $\imonad{X}$ into $\tmonad{X}$ preserves 1-units; because the 0-cells of $\tmonad{X}$ are the same as the 0-cells of $\imonad{X}$, we conclude that $\tmonad{X}$ is also 0-representable. By Lemma \ref{lem:merge_inclusion}, $\tmonad{X}$ is also 1-representable and admits a strictly associative family of tensors, which completes the proof.
\end{proof}

\begin{prop} \label{prop:talg_rep}
Let $X$ be a merge-bicategory. The following are equivalent:
\begin{enumerate}
	\item $X$ admits a $\tmonad{}$-algebra structure $\alpha: \tmonad{X} \to X$ where $\alpha$ is a strong morphism;
	\item $X$ is representable and admits a strictly associative family of tensors.
\end{enumerate}
\end{prop}
\begin{proof}
By Lemma \ref{lem:merge_inclusion}, $\zeta_{\imonad{X}}: \imonad{X} \to \tmonad{X}$ is a strong morphism, so $\beta := \alpha\zeta_{\imonad{X}}: \imonad{X} \to X$ is both a strong morphism and an $\imonad{}$-algebra structure on $X$; by Proposition \ref{prop:ialg_0rep}, $X$ is 0-representable.

Moreover, by Lemma \ref{lem:tmonad_rep}, $\tmonad{X}$ is 1-representable and admits a strictly associative family of tensors $\{t_{\sequ{\Gamma_1},\sequ{\Gamma_2}}\}$; then $\alpha(\sequ{a}) = a$ for each 1-cell $a$ of $X$, and $\{\alpha(t_{\sequ{a},\sequ{b}})\}$ is a strictly associative family of tensors in $X$.

Conversely, by Proposition \ref{prop:ialg_0rep} and Proposition \ref{prop:malg_1rep}, $X$ admits an $\imonad{}$-algebra structure $\alpha: \imonad{X} \to X$ and an $\mmonad{}$-algebra structure $\beta: \mmonad{X} \to X$, such that both $\alpha$ and $\beta$ are strong morphisms. Then $\beta \circ \mmonad{\alpha}: \mmonad{\imonad{X}} \to X$ is a $\tmonad{}$-algebra structure, and by Lemma \ref{lem:mmonad_strong} $\beta \circ \mmonad{\alpha}$ is a composite of strong morphisms, therefore it is strong.
\end{proof}

The structure of $\tmonad{}$-algebra leads to a stricter form of bicategory, as follows.

\begin{cons} \label{cons:strictlyasso}
Let $\alpha: \tmonad{X} \to X$ be a $\tmonad{}$-algebra such that $\alpha$ is a strong morphism; then $GX$ admits ``canonically'' the structure of a bicategory. For vertical composition and units, there are already unique choices. As a horizontal composite of $a: x \to y$, $b: y \to z$, pick $\alpha(\sequ{a,b})$. The horizontal composites are witnessed by the strictly associative family of tensors $\{\alpha((\thin{(a,b)})_{\sequ{a}\sequ{b}}^{\sequ{a,b}})\}$. 

As a horizontal unit on the 0-cell $x$, pick $\alpha(\sequ{\thin{x}})$. The $\{\alpha\zeta_X(l_a), \alpha\zeta_X(r_b)\}$ are coherent witnesses of unitality of $\alpha(\sequ{\thin{x}})$. 

The rest of the bicategory structure on $GX$ is determined by these choices; the fact that the chosen family of tensors is strictly associative implies that the bicategory structure is strictly associative.
\end{cons}

\begin{lem} \label{lem:imonad_to_tmonad}
Let $\alpha: \imonad{X} \to X$ be an $\imonad{}$-algebra. Then $\mmonad{X}$ admits a canonical $\tmonad{}$-algebra structure $\beta: \tmonad{(\mmonad{X})} \to \mmonad{X}$, and if $\alpha$ is a strong morphism, then so is $\beta$.
\end{lem}
\begin{proof}
Let $\beta: \tmonad{(\mmonad{X})} \to \mmonad{X}$ be the composite
\begin{equation*}
	\mmonad{\alpha} \circ \nu_{\imonad{X}} \circ \mmonad{\sigma_X};
\end{equation*}
a straightforward calculation shows that it is compatible with the multiplication and unit of $\tmonad{}$, hence it defines a $\tmonad{}$-algebra structure on $\mmonad{X}$.

Suppose that $\alpha$ is a strong morphism; by Proposition \ref{prop:ialg_0rep}, we know that $X$ is 0-representable and, \emph{a fortiori}, unital, and by Lemma \ref{lem:imonad_rep} we know that $\imonad{X}$ and $\imonad{(\mmonad{X})}$ are both 0-representable. Moreover, by Remark \ref{remark:imonad-no1rep}, we can characterise the universal cells of $\imonad{\mmonad{X}}$, and the morphism $\sigma_X: \imonad{\mmonad{X}} \to \mmonad{\imonad{X}}$ maps them to universal cells of $\tmonad{X}$, as described in the proof of Lemma \ref{lem:tmonad_rep}. 

Thus, $\sigma_X$ is a strong morphism, and both $\sigma_X$ and $\alpha$ satisfy the conditions of Lemma \ref{lem:mmonad_strong}; it follows that $\mmonad{\sigma_X}$ and $\mmonad{\alpha}$ are strong morphisms. Since we know how to construct 1-units and units in $\tmonad{X}$ and $\mmonad{\tmonad{X}}$, which are both representable merge-bicategories, it is not hard to check that $\nu_{\imonad{X}}$ is also a strong morphism. We conclude that $\beta$ is a strong morphism.
\end{proof}

We have reached the conclusion of our argument.
\begin{thm} \label{thm:semistrict}
Let $X$ be a representable merge-bicategory. Then there exist an equivalence $f: X \to Y$ of representable merge-bicategories, and a $\tmonad{}$-algebra $\beta: \tmonad{Y} \to Y$ such that $\beta$ is a strong morphism.
\end{thm}
\begin{proof}
Since $X$ is, \emph{a fortiori}, 0-representable, by Proposition \ref{prop:ialg_0rep} it admits an $\imonad{}$-algebra structure $\alpha: \imonad{X} \to X$ such that $\alpha$ is a strong morphism. It follows from Lemma \ref{lem:imonad_to_tmonad} that $\mmonad{X}$ admits a $\tmonad{}$-algebra structure $\beta: \tmonad{(\mmonad{X})} \to \mmonad{X}$ such that $\beta$ is a strong morphism. Finally, by Lemma \ref{lem:merge_inclusion}, $\zeta_X: X \to \mmonad{X}$ is an equivalence of representable merge-bicategories, so the statement follows with $Y := \mmonad{X}$.
\end{proof}

To claim that semi-strictification for bicategories follows from Theorem \ref{thm:semistrict} through the equivalence between $\mrgpol_\otimes$ and $\bicat$ would be at least partially circular, given that we relied on Mac Lane's coherence theorem to define one side of the equivalence. Nevertheless, if we read this equivalence as a definition of sorts --- that is, we define a bicategory as a representable merge-bicategory $X$ with the necessary structure on $GX$ --- it makes sense to state the following. In higher dimensions, where coherence results for weak $n$-categories are not at hand, the definitional reading may be the only one available.

\begin{cor}
Let $B$ be a bicategory. Then there exist a strictly associative bicategory $C$ and an equivalence $f: B \to C$.
\end{cor}
\begin{proof}
We know that $B$ is equivalent to $GX$ for some representable merge-bicategory $X$. Let $f: X \to Y$ be an equivalence, and $\beta: \tmonad{Y} \to Y$ a $\tmonad{}$-algebra as in the statement of Theorem \ref{thm:semistrict}. By Construction \ref{cons:strictlyasso}, $C := GY$ has the structure of a strictly associative bicategory, and by Corollary \ref{cor:equivalences} $Gf: GX \to C$ is an equivalence of bicategories. Precomposing $Gf$ with an equivalence $B \to GX$ leads to the conclusion.
\end{proof}

\begin{remark}
Even in the higher-dimensional case, we do not think that this is the strictest structure for which one can aim; it seems reasonable, for instance, to ask for weak units to be strictly idempotent. However, this constraint requires for $\mmonad{}$ to ``recognise'' the units created by $\imonad{}$, so we would lose the decomposition of $\tmonad{}$ into a pair of monads.

Since the first version of this article, we have found that this particular scheme, with two monads related by a distributive law, does not generalise well to higher dimensions; what generalises, however, is the core idea, which is to construct witnesses of composition first by ``only adding units'', then by ``only composing cells'', with strictification corresponding to the second step. This will be discussed in a forthcoming article.
\end{remark}

\bibliographystyle{alpha}
\small \bibliography{main}

\end{document}